\pdfoutput=1
\documentclass[11pt]{amsart}
\usepackage{pinlabel}
\usepackage{amsmath} 
\usepackage{cjhebrew}
\usepackage{graphicx} 
\usepackage{subcaption}
\usepackage{tikz-cd}
\setkeys{Gin}{keepaspectratio}
\usepackage{hyperref}
\usepackage{url}
\usepackage{bm}
\usepackage{mathabx}
\usepackage[left=1.25in,top=1in,right=1.25in,bottom=1in,head=.1in]{geometry}
\usepackage{xcolor}
\usepackage{tikz}
\usepackage{adjustbox}
\usepackage[normalem]{ulem} %
\usepackage{enumerate}
\usepackage{hyperref}

\usepackage{verbatim}
\newcommand{\items}{\begin{itemize}[leftmargin=25pt,rightmargin=15pt]
  \setlength\itemsep{2pt}}

\usepackage[
backend=biber,
style=alphabetic,
sorting=nyt, doi=false, isbn=false,  eprint=false,maxnames=15,minnames=15
]{biblatex}
\newcommand{\stopitems}{\end{itemize}}

\usepackage{fancyhdr}
\newtheorem{thm}{Theorem} 

\newtheorem{theorem}{Theorem}[section] 
\newtheorem*{theorem*}{Theorem}
\newtheorem{lemma}[theorem]{Lemma}
\newtheorem{conjecture}[theorem]{Conjecture}
\newtheorem{question}[theorem]{Question}
\newtheorem*{conjecture*}{Conjecture}
\newtheorem*{question*}{Question}
\newtheorem*{lemma*}{Lemma}
\newtheorem{conj}{Conjecture} %

\newtheorem{ques}[conj]{Question} 
\newtheorem{proposition}[theorem]{Proposition}
\newtheorem{corollary}[theorem]{Corollary}
\newtheorem*{corollary*}{Corollary}
\theoremstyle{definition}
\newtheorem{definition}[theorem]{Definition}

\newtheorem{remark}[theorem]{Remark}

\newtheorem{assumption}[theorem]{Assumption}
\newtheorem{example}[theorem]{Example}

\newtheorem*{example*}{Example}
\newtheorem*{remark*}{Remark}
\newtheorem*{remarks*}{Remarks}
\newtheorem*{addenda*}{Addenda}
\newtheorem*{construction*}{Construction}

\DeclareMathOperator{\diff}{Diff}

\DeclareMathOperator{\grad}{gr}
\DeclareMathOperator{\redu}{red}

\DeclareMathOperator{\FSW}{FSW} 

\newcommand{\R}{\mathbb R}
\newcommand{\Z}{\mathbb Z}

\newcommand{\bp}{\mathbb P}

\renewcommand{\phi}{\varphi}

\DeclareMathOperator{\MCG}{MCG}

\newcommand{\ev}{\operatorname{ev}}

\newcommand{\unred}[1]{ \ignorespaces}  

\newcommand{\0}{\textup{0}}

\title{On four-dimensional Dehn twists and Milnor fibrations}

\author{Hokuto Konno}
\address{Graduate School of Mathematical Sciences, the University of Tokyo, 3-8-1 Komaba, Meguro, Tokyo 153-8914, Japan \\and\\
RIKEN iTHEMS, Wako, Saitama 351-0198, Japan}
\email{konno@ms.u-tokyo.ac.jp}

\author{Jianfeng Lin}
\address{Yau Mathematical Sciences Center, Tsinghua University, Beijing, 100871, China}
\email{linjian5477@mail.tsinghua.edu.cn}

\author{Anubhav Mukherjee}
\address{Department of Mathematics, Princeton University, Princeton, 08540, USA}
\email{anubhavmaths@princeton.edu}

\author{Juan Muñoz-Echániz}
\address{Simons Center for Geometry and Physics, State University of New York, Stony Brook, 11794, USA}
\email{jmunozechaniz@scgp.stonybrook.edu}

\begin{document}
\setlength{\headheight}{12.0pt}
\begin{abstract}
We study the monodromy diffeomorphism of Milnor fibrations of isolated complex surface singularities, by computing the family Seiberg--Witten invariant of Seifert-fibered Dehn twists using recent advances in monopole Floer homology. More precisely, we establish infinite order non-triviality results for boundary Dehn twists on indefinite symplectic fillings of links of minimally elliptic surface singularities.
Using this, we exhibit a wide variety of new phenomena in dimension four: (1) smoothings of isolated complex surface singularities whose Milnor fibration has monodromy with infinite order as a diffeomorphism but with finite order as a homeomorphism, (2) robust Torelli symplectomorphisms that do not factor as products of Dehn--Seidel twists, (3) compactly supported exotic diffeomorphisms of exotic $\mathbb{R}^4$'s and contractible manifolds.

\end{abstract}
\maketitle
\section{Introduction}\label{section:introduction}



In this article we use recent advances in monopole Floer theory (namely, the Morse-Bott model \cite{FLINthesis} and the family gluing theorem \cite{JLIN2022}) to study mapping class groups of 4-manifolds. We consider the isotopy problem for the $4$-dimensional Dehn twist diffeomorphism along an embedded 3-manifold, and we obtain new applications to natural questions in \textit{singularity theory} (monodromy of Milnor fibrations, see \S\ref{intro:singularities}), \textit{symplectic topology} (robust elements in the symplectic Torelli group, see \S\ref{intro:symplectic}) and \textit{low dimensional topology} (mapping class groups of exotic $\R^4$'s and contractible 4-manifolds, see \S\ref{intro: low-dimensional topology}).


One way in which the classical $2$-dimensional Dehn twist diffeomorphism of an annulus $A = [0,1] \times S^1$ along the oriented circle $c = \{ 1/2\} \times S^1$ can be generalised to $4$ dimensions is the following. Consider the $4$-dimensional cylinder $Z = [0,1] \times Y$ where $Y$ is a \textit{Seifert-fibered} $3$-manifold over a closed oriented orbifold surface $C$. The Seifert fibration $Y \rightarrow C$ gives a fibered family of circles, and the induced projection $[0,1] \times Y \rightarrow C$ gives a fibered family of annuli (both of these are ``fibered" in the orbifold sense). Performing the $2$-dimensional Dehn twist fiberwise yields a diffeomorphism $\tau$ of $Z$. Since the diffeomorphism $\tau$ is supported in the interior of $Z$ it can then be implanted into a given smooth oriented $4$-manifold $X$ along a tubular neighbourhood of a given smooth embedding $Y \hookrightarrow X$, giving what is called the \textit{Dehn twist diffeomorphism on $X$ along the Seifert-fibered $3$-manifold $Y$} (see Definition \ref{defn:dehntwist} for the precise definition). This and closely related diffeomorphisms have been studied from the smooth \cite{KM-dehn, baraglia-konno, LinK3K3, konno-mallick-taniguchi}, topological \cite{giansiracusa,KKdehntwist,orson-powell} and symplectic \cite{seidelgraded,WW,barth-geiges} points of view (in the symplectic literature the term ``fibered Dehn twist" is more often used).

For a given compact oriented smooth $4$-manifold $M$ with boundary, let $\diff(M,\partial)$ denote the topological group of orientation-preserving diffeomorphisms of $M$ that agree with the identity near $\partial M$ (equipped with the $C^{\infty}$ topology), and define the \textit{smooth mapping class group} of $M$ as $\MCG(M):=\pi_{0}(\diff(M,\partial))$. A non-trivial element of $\diff (M , \partial )$ is called \emph{exotic} if it is topologically isotopic to the identity fixing a neighbourhood of $\partial M$ pointwise. 
We are mostly interested in the case when $\partial M = Y$ is a Seifert-fibered $3$-manifold, and the diffeomorphism $\tau_M \in \MCG(M)$ which is obtained by implanting the $4$-dimenional Dehn twist inside a collar neighbourhood of the boundary of $M$. We refer to $\tau_M$ as the \textit{boundary Dehn twist} on $M$.

Recently, Mallick, Taniguchi and the first author \cite{konno-mallick-taniguchi} proved that $\tau_M$ has \textit{order at least} $2$ in $\MCG(M)$ when $M$ is one of the Milnor fibers $M(2,3,7)$ or $M(2,3,11)$, or a positive-definite manifold bounded by $Y = \Sigma (2,3,6n+7)$. These results motivated us to investigate the following natural question:

\begin{ques}\label{ques: dehn infinite order} Given $M$, is $\tau_{M}$ an infinite order element in $\MCG(M)$?
\end{ques}

Our main result, Theorem \ref{thm: main}, answers Question \ref{ques: dehn infinite order} affirmatively for large families of $M$'s carrying symplectic structures (with convex boundary). Further motivation for Question \ref{ques: dehn infinite order} comes from the study of the monodromy of surface singularities and will be discussed in \S\ref{intro:singularities} below. 
In order to state this we introduce the following notion: 

\begin{definition}\label{defi: Floer simple}
Let $Y$ be a closed oriented $3$-manifold equipped with a Seifert fibration $Y \rightarrow C$ over a closed oriented $2$-orbifold $C$, or equivalently: $Y = S(N)$ is the unit circle bundle associated to an orbifold complex line bundle $N$ over $C$. We say $Y$ is \textit{Floer simple} if the following condition holds:
\begin{enumerate}
    \item $Y$ a \textit{rational homology $3$-sphere}, or equivalently: $C$ has genus zero and $N$ has non-trivial orbifold degree, $\mathrm{deg}N \neq 0$. We assume $Y$ is \textit{negatively-oriented}, meaning $\mathrm{deg}N < 0$. 
    \item The canonical spin-c structure $\mathfrak{\mathfrak{s}}_c$ on $Y$ induced by the Seifert fibration is \textit{self-conjugate}, or equivalently: the pullback of the orbifold tangent bundle $TC$ along $Y \rightarrow C$ is a trivial vector bundle.
    
    \item the \textit{reduced monopole Floer homology of $(Y, \mathfrak{s}_c)$ has rank one}, i.e. $HM^{\redu}(Y,\mathfrak{s}_c)\cong\mathbb{Z}$.
\end{enumerate}
\end{definition}

An infinite family of Floer-simple Seifert-fibered $3$-manifold is obtained by taking $Y = S(TC)$ the unit tangent bundle of an orbifold hyperbolic surface $C$ of genus zero (i.e. topologically $C$ is a sphere and the orbifold Euler characteristic $\chi(C)$ is negative). More examples are discussed below. Our main result is the following





\begin{thm}\label{thm: main} Let $Y$ be a Floer simple Seifert-fibered 3-manifold with a (positive, co-oriented) contact structure $\xi$ such that $\mathfrak{s}_\xi = \mathfrak{s}_c$. Then for any (weak) symplectic filling $M$ of $(Y,\xi)$  with $b^{+}(M)>0$, the boundary Dehn twist $\tau_{M}$ is an infinite order
element in the smooth mapping class group $\MCG(M)$. 
\end{thm}

When $M$ is simply-connected then $\tau_M$ is \textit{topologically} isotopic to the identity on $M$ (rel. $\partial M$, i.e. fixing a neighbourhood of the boundary of $M$ pointwise) by work of Orson--Powell \cite{orson-powell}; thus, in this situation Theorem \ref{thm: main} shows that $\tau_M$ is an \textit{exotic diffeomorphism of infinite order}.

In plain terms, the condition $\mathfrak{s}_\xi = \mathfrak{s}_c$ means that as a co-oriented $2$-plane field $\xi$ can be homotoped to one which is positively transverse to the fibers of $Y \rightarrow C$. Because $Y$ is negatively-oriented then there is, up to isotopy, a canonical contact structure $\xi = \xi_c$ on $Y$ with $\mathfrak{s}_\xi = \mathfrak{s}_c$, namely an $S^1$-invariant contact structure. More contact structures $\xi$ with this property are obtained by Eliashberg--Thurston pertubation \cite{confoliations} of transverse (hence taut) foliations on $Y$, whose existence was established by Lisca--Stipsicz \cite{lisca-stipsicz} provided $HM^{\redu}(Y) \neq 0$. Furthermore, contact $3$-manifolds $(Y, \xi )$ obtained by perturbation of taut foliations admit strong symplectic fillings $M$ with arbitrarily large $b^{\pm}(M)$ invariants, and which can be taken to be simply-connected \cite{EMM}.  

From another point of view, since $Y$ is a negatively-oriented Seifert-fibered $3$-manifold then it can be realised as the \textit{link of an isolated normal surface singularity} $(V, p )$ of quasi-homogeneous type. When the link $Y$ is a rational homology $3$-sphere, a sufficient condition for it to be Floer simple is that $(V, p)$ is a \textit{minimally elliptic surface singularity} \cite{nemethiplumbed}. Thus, an important class of symplectic fillings $M$ of such links $(Y, \xi_c)$ is given by the \textit{Milnor fibers} of smoothings of $(V, p )$, which have $b^{+}(M) = 2$ for the minimally elliptic singularities with rational homology sphere link \cite{durfee,steenbrink}. On the other hand, the \textit{minimal resolution} of any isolated normal surface singularity $(V, p )$ gives another symplectic filling $M$ of $(Y, \xi_c )$ but $b^{+}(M) = 0$, and one can show that the boundary Dehn twist $\tau_M$ is trivial in $\MCG (M)$ (see Proposition \ref{proposition:minres}); thus the $b^{+}(M) > 0$ hypothesis in Theorem \ref{thm: main} is sharp.

\subsection{Monodromy of surface singularities} \label{intro:singularities}

We now explain our motivation for Question~\ref{ques: dehn infinite order} from the study of the monodromy of Milnor fibrations of hypersurface singularities (see also \S\ref{section:backgroundsection} for a general discussion beyond the hypersurface case). This topic has been extensively investigated (see e.g. \cite{dimca,arnold}) following pioneering work of Brieskorn \cite{brieskorn} and Milnor \cite{milnor}, and lies at the heart of the study of singularities by topological methods. 

Let $f: (\mathbb{C}^3 , 0 )\rightarrow (\mathbb{C}, 0 )$ be a (germ of) complex-analytic function with an isolated critical point at $0$, and consider the associated (germ of) $2$-dimensional isolated hypersurface singularity $V = \{ f = 0\}$. For $0 < \epsilon \ll 1$ one obtains a smooth $3$-manifold $Y = V \cap S_{\epsilon} (\mathbb{C}^3 , 0 )$ called the \textit{link} of the singularity $(V,0)$. The link determines the topological type of the singularity, as $V$ is homeomorphic to the cone over $Y$. In turn, the fibers $V_t = f^{-1}(t)$ are non-singular for small $t \neq 0$. By removing the singular fiber and making all fibers compact we obtain the \textit{Milnor fibration} of $f$
\begin{align}
f^{-1}(B_{\delta}(\mathbb{C}, 0 ) \setminus 0 )\cap B_{\epsilon}(\mathbb{C}^3 , 0 ) \xrightarrow{f} B_{\delta} (\mathbb{C}, 0 )\setminus 0 \label{milnorfibrationv1}
\end{align}
which defines a smooth fiber bundle provided $0 < \delta \ll \epsilon \ll 1$ \cite{milnor}. A fiber $M$ of (\ref{milnorfibrationv1}) is a smooth $4$-manifold with boundary $\partial M \cong Y$, called a \textit{Milnor fiber} of $f$. 
Furthermore, the Milnor fibration is equipped with a canonical (up to homotopy) boundary trivialisation, i.e. (\ref{milnorfibrationv1}) is a principal bundle with structure group $\mathrm{Diff}(M , \partial )$). A fundamental invariant of $f$ is the \textit{monodromy} of the Milnor fibration, namely the diffeomorphism $\psi \in \MCG (M) \cong \pi_1 B\mathrm{Diff}(M, \partial )$ that classifies (\ref{milnorfibrationv1}) as a $\mathrm{Diff}(M, \partial)$-bundle. 

The classical topic of study has been the automorphism $\psi_\ast$ of the intersection form on the middle homology $H_2 (M , \mathbb{Z} )$ of the Milnor fiber induced by the monodromy, but considerably less is known about the monodromy in the smooth mapping class group.

We will consider the case when $f$ is a \textit{weighted-homogeneous} polynomial, meaning that there exist positive integers $(w_1 , w_2 , w_3 , d )$ such that $f(\lambda^{w_1} x_1 , \lambda^{w_2}x_2 , \lambda^{w_3}x_3 ) = \lambda^d f(x_1 , x_2 , x_3 )$ for all $\lambda \in \mathbb{C}^\ast$ (and without loss of generality, we can assume that $\mathrm{gcd}(w_1 , w_2 , w_3 ) =1 $). Then $V$ is called a $2$-dimensional \textit{weighted-homogeneous isolated hypersurface singularity}. We will assume that the link $Y$ is a rational homology $3$-sphere, which is equivalent to the invertibility of $1- \psi_\ast$ over $\mathbb{Q}$ \cite{milnor}. In this situation, the action of $\psi_\ast$ on $H_2 (M , \mathbb{Z} )$ has finite order \cite{milnor}, and by a result of Orson--Powell \cite{orson-powell} $\psi$ has \textit{finite order in the topological mapping class group} $\pi_0 \mathrm{Homeo}(M , \partial )$. It is natural to ask:

\begin{ques}\label{ques : mono infinite order} 
Let $V = \{f = 0\}$ be a $2$-dimensional weighted-homogeneous isolated hypesurface singularity, whose link $Y$ is a rational homology 3-sphere. Does the monodromy $\psi$ of the Milnor fibration (\ref{milnorfibrationv1}) of $f$ have infinite order in the \textbf{smooth} mapping class group $\MCG (M)$?
\end{ques}

A basic invariant of an isolated singularity $V$ is its \textit{geometric genus}, denoted $p_g (V) \in \mathbb{Z}_{\geq 0}$. 
The $2$-dimensional isolated hypersurface singularities $V = \{f = 0\}$ with $p_g (V) = 0$ are precisely the \textit{ADE singularities}. These are all weighted-homogeneous and with rational homology $3$-sphere link. A fundamental fact about the ADE singularities, 
established by Brieskorn \cite{brieskornADE} and first observed by Atiyah \cite{atiyah} for the $A_1$ singularity, is the \textit{simultaneous resolution} property. 
This has the particular consequence that \textit{Question \ref{ques : mono infinite order} has a negative answer when $p_g (V) = 0$}. In turn, when $p_g (V) > 0$ the simultaneous resolution property doesn't hold (see Proposition \ref{prop:nosimresolution}), which suggests the intriguing possibility that this might have a \textit{differential-topological explanation}: namely, that $\psi$ has infinite order in the smooth mapping class group. We establish this in the next simplest case:

\begin{corollary} \label{cor:infiniteordermonodromy}
The answer to Question \ref{ques : mono infinite order} is affirmative when $p_g (V) = 1$. 
\end{corollary}

This result follows from Theorem \ref{thm: main} together with the fact that a suitable non-trivial power of $\psi$ becomes smoothly isotopic (fixing $\partial M$ pointwise) to the boundary Dehn twist on $M$ (Proposition \ref{prop:dehnsmoothing}). The isolated hypersurface singularities with $p_g (V) = 1$ are precisely the \textit{minimally elliptic singularities} with degree $\leq 3$ classified by Laufer \cite{laufer}, among which 42 of them are weighted-homogeneous with rational homology $3$-sphere link (see Figure \ref{figure:table}). For instance, these include the Brieskorn singularities $x^2 + y^3 + z^l = 0$ for $7 \leq l \leq 11$. We refer to Corollary \ref{cor:infiniteordermonodromy2} for a generalisation of the above result, going beyond the hypersurface case. Based on the above, we feel compelled to put forward the following conjecture: 



\begin{conjecture}[special case of Conjecture \ref{conj:infiniteordermono2}]\label{conj:infiniteordermono}
The answer to Question \ref{ques : mono infinite order} is affirmative when $p_g (V) > 0$.
\end{conjecture}

It is worth comparing Corollary \ref{cor:infiniteordermonodromy} with the following result in higher dimensions (see Proposition \ref{highdim}): \textit{the monodromy $\psi \in \MCG(M)$ of an $n$-dimensional weighted-homogeneous isolated hypersurface singularity $V = \{f = 0\} \subset \mathbb{C}^{n+1}$ whose link $Y$ is an integral homology $(2n-1)$-sphere has finite order in $\MCG(M)$ if $n > 4$.} We do not know what occurs in the cases $n = 3$ or $4$. Thus, at least when the link is an integral homology sphere, the phenomenon exhibited by Corollary \ref{cor:infiniteordermonodromy} is special to singularities of complex dimension $n = 2$ 
and perhaps also $n = 3$ or $4$.

\subsection{Obstructions to Dehn twist factorisations}

Returning now to an isolated hypersurface singularity $V = \{ f = 0\} \subset \mathbb{C}^3$ of dimension $n = 2$, one can always ``morsify" $f$ by adding in additional linear terms, and a choice of vanishing paths for the new function then exhibits a factorisation $\psi = \delta_{S_1} \circ \cdots \circ \delta_{S_n}$ of the monodromy $\psi \in \MCG (M)$ of the hypersurface singularity $V = \{f = 0\}$ as a \textit{composition  of Dehn twists on embedded Lagrangian spheres} $S_i$ (this is the familiar Picard--Lefschetz transformation, which gives an element $\delta_{S_i} \in \MCG(M)$ with order $2$), see e.g. \cite[\S 2.1]{arnold}. In particular, the boundary Dehn twist $\tau_M$ factorises in $\MCG(M)$ as a composition of Dehn twists on Lagrangian $2$-spheres. Of course, embedded Lagrangian spheres are examples of smoothly embedded spheres $S$ with $S \cdot S = -2$ whose pairing $K \cdot S$ with the canonical class vanishes (other examples are symplectic $(-2)$-spheres). It is interesting to investigate factorisation questions for more general symplectic fillings, and to this end we establish the following counterpart to the aforementioned classical fact:

\begin{thm}\label{thm: factorization}
Let $Y$ be a Floer simple Seifert-fibered $3$-manifold equipped with a contact structure $\xi$ such that $\mathfrak{s}_\xi = \mathfrak{s}_c$, and let $(M,\omega)$ be a compact symplectic $4$-manifold with (weakly) convex contact boundary $(Y, \xi )$ and $b^+ (M) > 2$. Then no non-trivial power of the boundary Dehn twist $\tau_{M}$ admits a factorisation in $\MCG(M)$ as a composition $\delta_1 \circ \cdots \delta_n$ where each $\delta_i \in \MCG(M)$ is a diffeomorphism supported in a tubular neighbourhood of a $2$-sphere $S_i$ smoothly embedded in the interior of $M$ such that $S_i \cdot S_i < 0$ and $K \cdot S_i  = 0$, where $K$ is the canonical class of $(M, \omega )$.
\end{thm}

The assumption that $b^{+}(M) > 2$ is essential. As discussed above, the conclusion of Theorem \ref{thm: factorization} does not hold when $M$ is the Milnor fiber of an isolated hypersurface singularity $V = \{f = 0\}$ with $p_g = 1$ and rational homology $3$-sphere link; indeed $b^{+}(M) =2$ in this case \cite{durfee}.

\subsection{Robust symplectomorphisms and Donaldson's question} \label{intro:symplectic}

We now discuss some symplectic-topological aspects of Theorem \ref{thm: main} and Theorem \ref{thm: factorization}. 

A phenomenon which is rather special to dimension $4$, and which is intimately bound up with the simultaneous resolution phenomenon, is the existence of \textit{fragile} symplectomorphisms. Let us recall this notion (see \cite{seidel}). We consider a compact symplectic $4$-manifold $(M,\omega)$ with convex contact boundary $(Y,\xi)$ (i.e. $(M, \omega)$ is a strong symplectic filling of $(Y, \xi )$) and let $\mathrm{Symp}(M, \omega, \partial )$ denote the group of symplectomorphisms of $(M, \omega)$ which agree with the identity near $\partial M = Y$. The \textit{symplectic Torelli group} is the subgroup $\mathrm{TSymp}(M,\omega , \partial) \subset \mathrm{Symp}(M , \omega , \partial )$ whose elements act trivially on $H_{*}(M,\mathbb{Z})$. An element $f \in \mathrm{TSymp}(M,\omega)$ is called \emph{fragile} if the following three conditions hold:
\begin{itemize}
    \item there exists a path of symplectic forms $\{\omega_{s}\}_{s\in [0,1)}$ on $M$ such that $\omega_0=\omega$, and $\omega_s=\omega$ near $\partial M = Y$ for all $s$
    \item there exists a path $\{f_{s}\}_{s\in[0,1)}$ in $\diff(M,\partial)$ such that $f^*_{s}(\omega_s)=\omega_s$.
    \item $f_{s}$ belongs to the identity component of $\mathrm{TSymp}(M,\omega_s)$ when $s>0$ but it does not when $s=0$. 
\end{itemize}
Previously known fragile elements in $\mathrm{TSymp}(M,\omega, \partial)$ are given by \textit{squared Dehn--Seidel twists} $(\delta_{S})^2$ along embedded Lagrangian spheres $S \subset M$ \cite{seidel}. It is natural  to search for \textit{robust} elements  $f \in \mathrm{TSymp}(M,\omega, \partial)$, i.e. $f$ is non-fragile and not in the identity component of $\mathrm{TSymp}(M, \omega , \partial)$. We are able to exhibit a large supply of Torelli symplectomorphisms all of whose non-trivial powers are robust (see also \cite{buse} for robust higher-dimensional parametric families of symplectomorphisms).

We suppose that the convex contact boundary $(Y, \xi ) $ of $(M, \omega)$ is a Seifert manifold of negative degree, and $\xi = \xi_c$ is the canonical contact structure. Then one can choose a contact form $\eta$ for $\xi_c $ which is preserved by the Seifert $S^1$-action. Choosing a collar neighborhood of the boundary symplectomorphic to $((0,1]\times Y, d(e^{t}\eta))$ it is straightforward to verify that the boundary Dehn twist $\tau_{M}$ preserves $\omega$ (see Definition (\ref{defn:dehntwist})); hence we have $\tau_M \in \pi_{0}\mathrm{TSymp}(M,\omega)$. Of course, fragile symplectomorphisms and their products are trivial in $\MCG(M)$. The following is an immediate consequence of Theorems \ref{thm: main} and \ref{thm: factorization}

\begin{corollary}\label{symplectic torelli}
Let $Y$ be a Floer simple Seifert $3$-manifold, and $(M,\omega)$ a compact symplectic $4$-manifold with convex contact boundary $(Y,\xi_{c})$ and $b^{+}(M) > 0$, and let $f \in \mathrm{TSymp}(M, \omega , \partial )$ be any Torelli symplectomorphism that is smoothly isotopic rel. $\partial M$ to $\tau_{M}^{k}$ for some $k \neq 0$ (e.g. $f = \tau_{M}^{k}$). Then $f$ is a robust element in $\mathrm{TSymp}(M, \omega , \partial )$. In addition:
\begin{enumerate}
    \item $f$ is not symplectically isotopic (rel. $\partial M$) to a product of \textbf{squared} Dehn--Seidel twists along Lagrangian $2$-spheres embedded in the interior of $M$.
    \item If $b^+(M)>2$, then  $f$ is not symplectically isotopic (rel. $\partial M$) to a product of Dehn--Seidel twists along Lagrangian $2$-spheres embedded in the interior of $M$. 
\end{enumerate}
\end{corollary}

Alternatively, we can complete $(M , \omega )$ by gluing in a symplectisation end $([1, \infty)\times Y , d(e^t \eta) ) $, thus obtaining a non-compact symplectic manifold $(\widehat{M}, \widehat{\omega})$; Corollary \ref{symplectic torelli} can be phrased instead in terms of the compactly-supported Torelli symplectomorphism group of $(\widehat{M}, \widehat{\omega} )$.

The following question is attributed to Donaldson in the literature \cite{sheridan-smith,li-li-wu,arabadji-baykur}: \textit{for a closed simply--connected symplectic $4$-manifold $(M, \omega )$, is every Torelli symplectomorphism of $(M, \omega )$ symplectically isotopic to a product of squared Dehn--Seidel twists?}
The answer is known to be affirmative for positive rational surfaces \cite{li-li-wu}, but otherwise remains widely open. In a similar vein, one could also ask whether every symplectomorphism of $(M, \omega )$ is symplectically isotopic to a product of Dehn--Seidel twists. If one drops the simple--connectivity assumption on $M$ then examples exist for which both answers are negative \cite{arabadji-baykur,Smirnov23}. Corollary \ref{symplectic torelli} provides infinitely many examples where both answers are negative when one drops the closed condition (instead, $(M, \omega )$ has convex boundary or is non-compact with the end modelled on a symplectisation). 

It is worth comparing the different behaviours that \textit{symplectic monodromies} of complex surface singularities can exhibit. Namely, consider again a weighted-homogeneous isolated hypersurface singularity in $\mathbb{C}^3$ with rational homology $3$-sphere link. The restriction of the standard symplectic form on $\mathbb{C}^3$ onto the Milnor fibers makes the Milnor fibration (\ref{milnorfibrationv1}) into a symplectic fiber bundle over $S^1$, or more precisely: a principal bundle over $S^1$ with structure group $\mathrm{Symp}(M, \omega, \partial )$; we still denote its monodromy by $\psi \in \pi_0 \mathrm{Symp}(M , \omega , \partial )$, which can be factorised into a composition of Dehn--Seidel twists on Lagrangian spheres (see \cite[\S 4c]{seidelgraded}). When the geometric genus $p_g = 1$, choosing $d \neq  0$ such that $f := \psi^d$ is smoothly isotopic to $\tau_M$ (see Proposition \ref{prop:dehnsmoothing}) we obtain by Corollary \ref{symplectic torelli} that $f$ and all its non-trivial powers are robust Torelli symplectomorphisms which do not factor as products of squared Dehn--Seidel twists. When $p_g = 0$ (i.e. an ADE singularity) then Seidel proved that $\psi$ has infinite order in $\pi_0 \mathrm{Symp}(M , \partial )$ \cite[\S 4c]{seidelgraded}. However, by the simultaneous resolution property for ADE singularities we may find some $d \neq 0$ for which $f := \psi^d$ and all its non-trivial powers are fragile Torelli symplectomorphisms, and which also factor as products of squared Dehn--Seidel twists (Proposition \ref{prop:ADE}).





\subsection{Exotic Dehn twists on open, contractible, and closed manifolds}\label{intro: low-dimensional topology}

By combining Theorem~\ref{thm: main} with the gluing formula for family Seiberg--Witten invariants proved by the second author \cite{JLIN2022}, we are able to obtain several intriguing applications in low-dimensional topology. 

Our first application concerns the mapping class groups of open manifolds. For a non-compact smooth $4$-manifold $\mathcal{R}$ (without boundary), we define the \textit{compactly-supported mapping class group} $\operatorname{MCG}_{c}(\mathcal{R})$ to be the group of compactly-supported diffeomorphisms of $\mathcal{R}$ modulo compactly-supported isotopy. We call a non-trivial element of $\MCG_{c}(\mathcal{R})$ a \textit{compactly-supported exotic diffeomorphism} of $\mathcal{R}$ if it is isotopic to the identity through compactly-supported homeomorphisms. For $\mathcal{R}=\mathbb{R}^4$, one has 
\[
\MCG_{c}(\mathbb{R}^4)\cong \MCG(D^4)\cong \MCG(S^4).
\]
Watanabe \cite{watanabe} has proved that $\mathrm{Diff}(D^4, \partial)$ is not contractible (thus disproving the four-dimensional Smale conjecture), but the following important question remains open: \textit{is $\MCG_c (\mathbb{R}^4 )$ trivial?} We prove the following relevant result:
\begin{thm}\label{thm: MCG exotic R4}
There exists an exotic $\mathbb{R}^4$, denoted by $\mathcal{R}$, such that $\MCG_{c}(\mathcal{R})$ is non-trivial. Moreover, $\MCG_{c}(\mathcal{R})$ admits a surjective homomorphism onto $\mathbb{Z}$.
\end{thm}

On a related note, Gompf \cite{gompf} has proved that there exists an exotic $\mathbb{R}^4$ that admits orientation-preserving diffeomorphisms that are not smoothly isotopic to identity. However, his examples are \textit{not} compactly-supported, and the exoticness comes from the end.

The non-trivial mapping class in Theorem \ref{thm: MCG exotic R4} is given by a Dehn twist along a Seifert-fibered 3-manifold embedded in $\mathcal{R}$. Furthermore, the exoticness of $\mathcal{R}$ is established by methods different from previous approaches, crucially relying on a ``non-localization" property of the Dehn twist (see the upcoming Theorem \ref{thm: Mazur}).



$ $

Our techniques also enable us to study mapping class group of compact contractible 4-manifolds. We consider the following family of Brieskorn homology $3$-spheres (none of which are Floer simple)

\begin{align}
\label{eq: subcollection of Casson-Harer family}
\left\{
\begin{array}{ll}
\Sigma(p,ps+1,ps+2) & \text{($s\geq1$, and odd $p\geq3$)},\\
\Sigma(p,ps-1, ps-2) & \text{($s \geq 2$, and odd $p\geq5$),}
\end{array}
\right.
\end{align}
which forms an infinite subcollection of the Casson--Harer family \cite{CassonHarer81}, a well-known list of Brieskorn spheres that bound compact contractible (more strongly, Mazur) smooth 4-manifolds. The following ``non-localization" theorem shows that the boundary Dehn twists of such contractible 4-manifolds are infinite order exotic diffeomorphisms that do not localize to a 4-ball:

\begin{thm}
\label{thm: Mazur}
Let $Y$ be one of the families \eqref{eq: subcollection of Casson-Harer family} and $W$ be a smooth rational homology ball bounded by $Y$. Let $i:D^{4}\hookrightarrow \operatorname{int}(W)$ be a smooth embedding of a $4$-ball and let $i_\ast:\MCG(D^4)\to \MCG(W)$ be the induced map. Then for any $k\neq 0$, $\tau_{W}^{k}$ does not belong to the image $i_\ast$. In particular, $\tau_{W}$ has infinite order in $\MCG(W)$.
\end{thm}

In recent work by Galatius--Randal-Williams \cite{GRW},
they proved that for any compact, contractible $n$-manifold $W$ with $n\geq 6$ and any smooth embedding of $D^n\hookrightarrow \mathrm{int}(W)$, the map $i:\operatorname{Diff}(D^n,\partial)\hookrightarrow \operatorname{Diff}(W,\partial)$ defined by extending diffeomorphisms by the identity is a weak homotopy equivalence. Krushkal--Powell--Warren and the third author \cite{KMPW} recently showed that there exists some compact, contractible 4-manifold $W$ such that $i$ is not sujective on $\pi_0$. However, it remains an open question to explicitly describe the contractible manifold and the diffeomorphisms arising from that proof. Theorem~\ref{thm: Mazur} gives such explicit examples. 

\begin{remark}
\label{rem: Kang-Park-Taniguchi}
The infinite-order assertion in Theorem \ref{thm: Mazur} is covered by a recent result by Kang-Park-Taniguchi~\cite{KangParkTaniguchi}, who independently proved (using very different techniques) that the boundary Dehn twist on every positive-definite 4-manifold $W$ with $b_1=0$ bounded by a negatively-oriented Seifert-fibered integral homology 3-sphere $\neq S^3$ has infinite order in $\MCG(W)$.
\end{remark}


Theorem \ref{thm: Mazur} has the following consequence on exotic surface embeddings. Finding exotic surfaces, especially exotic disks in compact 4-manifolds, is a fascinating area of study. 
The existence of such examples in the 4-ball has been demonstrated by Hayden \cite{hayden2021corkscoverscomplexcurves, hayden2021exoticallyknotteddiskscomplex} and Akbulut \cite{akbulut2022corksexoticribbonsb4}, and are detected by showing that the smooth structures on the complement are non-diffeomorphic relative to the boundary.
However, our techniques enable us to construct exotic disks with \textit{diffeomorphic complements} on contractible 4-manifolds.
By considering the images of 2-handle co-cores of Mazur manifolds under powers of boundary Dehn twists we obtain:

\begin{corollary}\label{cor: exotic disk}
Let $Y$ be one of the families \eqref{eq: subcollection of Casson-Harer family} and $W$ be a Mazur manifold bounded by $Y$.
Then there exists infinitely many properly and smoothly embedded disks $D_{i}$ into $W$ such that $D_{i}$'s are mutually topologically isotopic rel boundary and have diffeomorphic complements, but are mutually not smoothly isotopic rel boundary.  
\end{corollary}


As our final application, we consider a closed 4-manifold $X$ that contains an embedded Seifert-fibered $3$-manifold $Y$, and denote the Dehn twist on $X$ along $Y$ by $\tau_{(X,Y)}$. It is natural to ask whether $\tau_{(X,Y)}$ is non-trivial in $\MCG(X)$ and what is its order. The known examples of exotic Dehn twists on closed 4-manifolds \cite{KM-dehn,konno-mallick-taniguchi} are shown to have a order at least 2. 
We provide the first examples of closed 4-manifolds on which Dehn twists along embedded 3-manifolds are confirmed to be infinite order:

\begin{thm}\label{thm: closed}
Let $Y$ be a Floer simple 3-manifold, with a contact structure $\xi$ such that $\mathfrak{s}_\xi = \mathfrak{s}_c$. Let $M$ a compact symplectic 4-manifold with weakly convex boundary $(Y,\xi)$ and $b^{+}(M)> 0$, and let $W$ be a compact smooth 4-manifold with boundary $-Y$ and $b^{+}(W) = 0$. Let $X=M\cup_{Y}W$. Assume there exists a spin-c structure $\mathfrak{s}$ on $X$ that restricts to the canonical spin-c structure on $M$ and satisfies 
    \[
    \frac{c^{2}_{1}(\mathfrak{s})-2\chi(X)-3\sigma(X)}{4}=-1.
    \]
    Then $\tau_{(X,Y)}$ has infinite order in $\MCG(X)$. In particular, $\tau_{W}$ also has infinite order in $\MCG(W)$.
\end{thm}


It is interesting to compare Theorem \ref{thm: closed} with the observation that many natural examples of Seifert-fibered 3-manifolds embedded in closed 4-manifolds turn out to yield trivial Dehn twists:

\begin{example}\label{intro:example}
Consider the standard smooth embedding of the Milnor fiber $M(2,3,k)$, where $k$ is of the form $6n-1$ or $6n-5$, into the elliptic surface $X= E(n)$ (see \cite{gompfstipsicz}). 
The boundary of $M(2,3,k)$ gives a smooth embedding of $Y=\Sigma(2,3,k)$ into $X$.
Then $\tau_{(X,Y)}$ is trivial in $\MCG(X)$ (see Proposition \ref{proposition:E(n)}); furthermore if $\mathring{E(n)}$ is the manifold obtained from $E(n)$ by removing a $4$-ball disjoint from $Y$, then the Dehn twists along $Y$ and $S^3 = \partial \mathring{E(n)}$ represent the same element in $\MCG (\mathring{E(n)} )$.
\end{example}

Another motivation for Theorem \ref{thm: closed} comes from the following major open problem: \textit{does there exist a closed irreducible 4-manifold that admits an exotic diffeomorphism?}
While we do not know whether the closed 4-manifolds in Theorem \ref{thm: closed} are irreducible, Theorem \ref{thm: closed} provides a systematic way to produce (infinite order) exotic diffeomorphisms of closed 4-manifolds, which is different from the known technique pioneered by Ruberman \cite{ruberman} using exotic pairs of closed 4-manifolds.

\subsection{Outline and Comments}

The article is organised as follows.

\S \ref{section:backgroundsection} discusses background material on the Dehn twist diffeomorphism, Seifert fibrations and isolated surface singularities. We discuss here the results and conjectures from \S \ref{intro:singularities} in a more general context. We conclude by discussing the isotopy problem for various natural Dehn twists on the elliptic surface $E(n)$, and establish the assertions made in Example \ref{intro:example}.

\S \ref{section:MOY} reviews a monopole-divisor correspondence for moduli of flowlines on Seifert-fibered $3$-manifolds established by Mrowka--Ozsváth--Yu \cite{MOY}.

\S \ref{section:dirac} studies some spectral properties of the Dirac operator (coupled with a reducible monopole) on a Seifert-fibered $3$-manifold in the \textit{adiabatic limit} as the area of the base orbifold of the Siefert fibration goes to zero. In this regime, we determine the eigenspaces of the Dirac operator and study the $S^1$-action on them induced from the Seifert $S^1$-action.

\S \ref{section:calculations} combines the results of the previous two sections to determine the topology of the critical loci and moduli spaces of flowlines (in the blowup) on Seifert-fibered rational homology $3$-spheres.

\S \ref{section:familiescobordism} discusses the maps on monopole Floer homology induced by families of cobordisms, and the gluing formula for those established in \cite{JLIN2022}.



\S \ref{section: proof main} contains the proof of Theorem \ref{thm: main} and Theorem \ref{thm: factorization}. The proof of Theorem \ref{thm: main} roughly goes as follows. The boundary Dehn twist $\tau_M$ on $M$ has a \textit{relative family Seiberg--Witten invariant} in the canonical spin-c structure $\mathfrak{s}_M$, which takes values in the monopole Floer homology of its boundary 
\[
\mathrm{FSW}( \tau_M , \mathfrak{s}_M ) \in \widecheck{HM}_{-2h(Y, \mathfrak{s}_c )}(Y , \mathfrak{s}_c ) \cong \mathbb{Z}.
\]
We show that this invariant is non-vanishing. For this, we study the map on the monopole Floer homology induced by the Dehn twist $\tau$ on the cylinder $Z = [0,1] \times Y$
\[
\tau_\ast : \widecheck{HM}_\ast (Y, \mathfrak{s}_c) \rightarrow \widecheck{HM}_{\ast +1}(Y , \mathfrak{s}_c).
\]
When $Y = S(N)$ is Floer simple we calculate this map using the Morse--Bott model for the monopole groups developed by F. Lin \cite{FLINthesis}, bringing to bear the spectral analysis of the Dirac operator from \S \ref{section:dirac} and the description of the moduli spaces from \S \ref{section:calculations}. Next, using the assumption that $M$ has a symplectic structure $\omega$ with weakly convex boundary, we establish that the relative Seiberg--Witten invariant $SW (M, \mathfrak{s}_\omega) \in \widecheck{HM}_{-2h(Y, \mathfrak{s}_c )-1} (Y, \mathfrak{s}_c)$ of $M$ is not annhilated by $\tau_\ast$. The proof concludes by applying the gluing formula, which says that
\[
\mathrm{FSW}(\tau_M , \mathfrak{s}_M ) = \tau_\ast ( SW(M, \mathfrak{s}_M ) ) \neq 0 .\]
More precisely, we show that
\[
\mathrm{FSW}(\tau_M , \mathfrak{s}_M ) = \pm \frac{\mathrm{deg}N }{\chi(C)} \in \mathbb{Z}
\]
where $\mathrm{deg}N \in \mathbb{Q} \setminus 0$ and $\chi(C) \in \mathbb{Q} \setminus 0$ denote the orbifold degree of $N \rightarrow C$ and the orbifold Euler characteristic of the orbifold surface $C$, respectively. 

To prove Theorem \ref{thm: factorization}, we show that if $b^+(M)>2$ then $\mathrm{FSW}(\delta , \mathfrak{s}_M )=0$ for any $\delta \in \MCG(M)$ supported near a $(-2)$-sphere with $K \cdot S = 0$, and we then combine this with the \textit{addivity property} of the family Seiberg--Witten invariant \textit{when $b^{+}(M) >2$}. This has the following additional implication: \textit{the family Seiberg--Witten invariant of the monodromy $\psi \in \MCG(M)$ of a $2$-dimensional hypersurface singularity with rational homology $3$-sphere link and $p_g > 1$ vanishes in the canonical spin-c structure}, $\mathrm{FSW}(\psi , \mathfrak{s}_M ) = 0$, since the Milnor fiber $M$ in this case has $b^{+} = 2 p_g > 2$ \cite{durfee} and the monodromy admits a factorisation into Dehn twists along Lagrangian $2$-spheres. In turn, when $p_g = 1$ then $b^+ = 2$ and the addivity property is no longer guaranteed to hold, due to potential wallcrossing phenomena; the fact that $\mathrm{FSW}(\tau_M , \mathfrak{s}_M ) \neq 0 $ in this case (even if $\tau_M$ is a product of Dehn twists on Lagrangian $2$-spheres) reveals intricate wallcrossing structures.

\S \ref{section other main results} discusses the proofs of Theorems \ref{thm: MCG exotic R4}, \ref{thm: Mazur} and \ref{thm: closed}. To prove Theorem \ref{thm: closed},
we calculate the family Seiberg--Witten invariant of $\tau_{(X,Y)}$ using the gluing formula \cite{JLIN2022}. On the other hand, the gluing formula also tells us that this diffeomorphism cannot be be isotoped to a diffeomorphism supported on a 4-ball. By constructing $S^1$-equivariant cobordisms between between Seifert manifolds and embedding them into closed manifolds, one can use Theorem \ref{thm: closed} to study Dehn twists along manifolds that are not Floer simple. In particular, Theorems \ref{thm: MCG exotic R4} and \ref{thm: Mazur} are consequences of Theorem \ref{thm: closed}, shown by embedding the contractible 4-manifolds and open manifolds into closed 4-manifolds.

\subsection*{Acknowledgements:} We are grateful to Bob Gompf, Francesco Lin, Samuel Muñoz-Echániz, Olga Plamenevskaya, Mark Powell, Mohan Swaminathan, Weiwei Wu and Zhengyi Zhou for enlightening discussions.
We thank Sungkyung Kang, JungHwan Park, and Masaki Taniguchi for letting us know about their work \cite{KangParkTaniguchi}, which was conducted independently and shares some overlap with this paper (see Remark~\ref{rem: Kang-Park-Taniguchi}).
The authors were informed that Jin Miyazawa also obtained a result related to this paper independently.
HK was partially supported by JSPS KAKENHI Grant Number 21K13785.
JL was partially supported by NSFC 12271281. 
AM was partially supported by NSF std grant DMS 2405270. 
The project began with HK's talk at the Low-dimensional Topology Conference, followed by HK and JM's visit to AM at Princeton to advance the collaboration. We appreciate Princeton University’s hospitality and supportive environment.

\pagebreak

\section{Dehn twists, Seifert-fibrations, and surface singularities}\label{section:backgroundsection}

\subsection{Dehn twists along Seifert $3$-manifolds}

\subsubsection{Four-dimensional Dehn twists}

We start by gathering some basic facts about the four-dimensional Dehn twist diffeomorphism. We consider the following notion of $4$-dimensional Dehn twist generalising the one discussed in \S \ref{section:introduction}. Let $Y^3$ be a closed, oriented $3$-manifold, equipped with a loop of diffeomorphisms $\gamma = (\gamma_t )_{t \in S^1} \in \pi_1 \mathrm{Diff}(Y)$.

\begin{definition}\label{defn:dehntwist}
The \textit{model Dehn twist} diffeomorphism of the cylinder $Z = [0,1] \times Y$ associated with $\gamma \in \pi_1 \mathrm{Diff}(Y)$ is the diffeomorphism $\tau_\gamma \in \MCG (Z)$ given as follows. Choose any smooth function $\theta : [0,1] \rightarrow [0,2 \pi]$ with $\theta(t) = 0$ near $t = 0$ and $\theta(t) = 2 \pi$ near $t = 1$, and set 
\[
\tau_\gamma (t , y ) = (t , \gamma_{e^{i \theta (t)}} (y) ).
\]
If $Y^3$ is equipped with a smooth embedding into (the interior of) a smooth oriented $4$-manifold $X$, then by implanting the model Dehn twist $\tau_\gamma$ in a neighbourhood of $Y \subset X$ we obtain the \textit{Dehn twist on $X$ along $Y$ associated to $\gamma$}.
\end{definition} 

Let $M^4$ be a compact smooth oriented $4$-manifold with connected boundary $\partial M = Y^3$. By the \textit{boundary Dehn twist} on $M$ associated with $\gamma \in \pi_1 \mathrm{Diff}(Y)$ we mean the Dehn twist on $M$ performed along a parallel copy of the boundary of $M$, and we denote it $\tau_{M , \gamma} \in \MCG(M)$. The following is well-known:

\begin{lemma}
The restriction of diffeomorphisms of $M$ onto the boundary $\partial M = Y$ gives a (Serre) fibration
\begin{align}
\mathrm{Diff}(M , \partial ) \rightarrow \mathrm{Diff}(M) \rightarrow \mathrm{Diff}(Y). \label{fibration}
\end{align}
\end{lemma}

This gives a simple description of the boundary Dehn twist on $M$:

\begin{lemma}\label{lemma:loopextension}
The boundary Dehn twist $\tau_{M, \gamma} \in \MCG (M) = \pi_0 \mathrm{Diff}(M, \partial M )$ is the image of $\gamma$ under the connecting map $\pi_1 \mathrm{Diff}(Y) \rightarrow \pi_0 \mathrm{Diff}(M , \partial )$ of the fibration (\ref{fibration}). In particular $\tau_M$ is trivial in $\MCG (M)$ if and only if the loop of diffeomorphisms of $\partial M  = Y$ given by $\gamma$ extends to a loop of diffeomorphisms of $M$.
\end{lemma}

The homotopy type of the diffeomorphism groups of prime $3$-manifolds is fully understood now, and the next result follows from combined efforts by several authors spanning more than 40 years \cite{hatcher1,ivanov,gabai,mccullough,bk3} (more precisely, it follows from Theorem 1 in \cite{mccullough} combined with Theorem 1.1 in \cite{bk3}):

\begin{theorem}[\cite{hatcher1,ivanov,gabai,mccullough,bk3}] \label{thm:pi1diffY}
Suppose $Y^3$ is a closed, oriented, irreducible $3$-manifold with infinite fundamental group and not diffeomorphic to the $3$-torus. Then the identify component $\mathrm{Diff}_0 (Y)$ in the diffeomorphism group of $Y$ is homotopy equivalent to $S^1$ if $Y$ is a Seifert-fibered $3$-manifold, in which case the Seifert $S^1$-action provides the homotopy equivalence; and $\mathrm{Diff}_0 (Y)$ is contractible if $Y$ is not Seifert-fibered.
\end{theorem}

We recall the following

\begin{definition}
A \textit{Seifert-fibered $3$-manifold} $Y$ is a closed oriented $3$-manifold $Y$ which is the total space of an orbifold principal $S^1$-bundle over an orbifold closed oriented surface $C$. We denote the projection map by $Y \xrightarrow{\pi} C$. Equivalently, $Y$ is presented as the unit sphere bundle $S(N)$ associated with an orbifold hermitian line bundle $N \xrightarrow{\pi} C$.

The $S^1$-action on $Y$ induced from a Seifert-fibered structure is referred to as the \textit{Seifert $S^1$-action}, and the associated loop of diffeomorphisms $\gamma_Y \in \pi_1 \mathrm{Diff}(Y)$ is referred to as the \textit{Seifert loop}.
\end{definition}

The $4$-dimensional Dehn twist along a Seifert-fibered $3$-manifold using the Seifert loop $\gamma = \gamma_Y$ is a natural generalisation of the familiar $2$-dimensional Dehn twist diffeomorphism on the $2$-dimensional annulus $A = [0,1] \times S^1$ along the curve $c = \{1/2\} \times S^1$. Namely, the $4$-dimensional cylinder $Z = [0,1] \times Y$ fibers (in an orbifold sense) over $C$ with annulus fibers; the Dehn twist diffeomorphism on $Z$ along $Y$ is given by performing the $2$-dimensional Dehn twist simultaneously on each annulus fiber.

From now on, whenever a four-dimensional Dehn twist is performed along a Seifert-fibered $3$-manifold, \textit{it will be understood that $\gamma = \gamma_Y$ is taken to be the Seifert loop}, unless otherwise noted.

\subsubsection{Seifert fibrations, orbifold surfaces and line bundles }

We assume familiarity with the ``classical'' theory of orbifolds. We refer to \cite{ruan,furutasteer,MOY} for the basic definitions and to \cite{boyer} for a treatment of differential and complex geometry in the orbifold context. We will only consider \textit{effective} orbifolds throughout this article. We will now review a few basic notions with the purpose of setting up notation (for detailed definitions see the references above).

Let $C$ be an orbifold oriented closed surface. It has an underlying smooth oriented closed surface $|C|$, called the \textit{desingularisation} of $C$, containing marked points $x_i \in |C|$ for $i = 1 , \ldots , n$, called \textit{orbifold points} of $C$. The orbifold $C$ is modelled near each $x_i$ on the orbifold chart $B(\mathbb{C} , 0 ) / (\mathbb{Z}/\alpha_i )$ for some $\alpha_i \geq 2$ (here $\mathbb{Z}/\alpha_i \subset S^1$ acts on the unit ball $B(\mathbb{C} , 0 )$ via the standard weight one $S^1$ action), and away from the orbifold points $C$ is modelled on a smooth chart. The \textit{orbifold Euler characteristic} of $C$ is the rational number
\[
\chi (C) = 2- 2g + \sum_{i =1}^n \big( \frac{1}{\alpha_i} - 1  \big) \in \mathbb{Q}.
\]

Let $N \rightarrow C$ be an orbifold complex line bundle. Over an orbifold chart $B(\mathbb{C} , 0 ) / (\mathbb{Z}/\alpha_i ) $ of $C$ the orbifold line bundle is given by $ ( \mathbb{C} \times B(\mathbb{C} , 0 ) )/(\mathbb{Z}/\alpha_i ) \xrightarrow{\pi}  B(\mathbb{C}, 0 )/ (\mathbb{Z}/\alpha_i )$ where $\pi (x,y) = y$ and $\mathbb{Z}/\alpha_i$ acts on the fiber factor $\mathbb{C}$ with some integer weight $0 \leq \beta_i < \alpha_i$, and over a smooth chart it is an ordinary complex line bundle. If for each orbifold chart we instead set the weights $\beta_i = 0$, the same gluing data of $N$ produces now an ordinary complex line bundle $|N|$ over $|C|$, called the \textit{desingularisation} of $N \rightarrow C$.
The \textit{background degree} of $N$ is the degree of the desingularisation $|N| \rightarrow |C|$, which we denote by $e \in \mathbb{Z}$. 
The \textit{orbifold degree} of $N$ is the rational number
\[
\mathrm{deg} N = e + \sum_{i = 1}^{n} \frac{\beta_i}{\alpha_i} \in \mathbb{Q} .
\]

It is easy to see that, in order for the unit sphere $Y= S(N)$ of an orbifold hermitian line bundle $N \rightarrow C$ to be a smooth $3$-manifold (and hence $Y$ is a Siefert-fibered $3$-manifold over $C$), it is necessary and sufficient that $\alpha_i$ be coprime with $\beta_i$ for all $i = 1 , \ldots , n$.

\begin{definition}
The \textit{Seifert data} of an orbifold complex line bundle $N$ over an orbifold closed oriented surface $C$ is the collection of integers $( e ; ( \alpha_1 , \beta_1 ) , \ldots , ( \alpha_n , \beta_n )  )$. When the $\alpha_i$'s are understood we might just denote the Seifert data by $( e ; \beta_1 , \ldots , \beta_n )$, or by $(e ; \frac{\beta_1}{\alpha_1} , \ldots , \frac{\beta_n}{\alpha_n})$ when $\alpha_i$ and $\beta_i$ are coprime for each $i$.
\end{definition}

\begin{example}Given an orbifold holomorphic structure on an orbifold closed oriented surface $C$, there is a preferred complex line bundle given by the \textit{canonical line bundle} $K_C := T^\ast C$. If $C$ has orbifold points $x_1 , \ldots , x_n$ with isotropy order $\alpha_i$ at $x_i$, then the Seifert data of $K_C$ is given by $(2g-2 , \alpha_1 -1 , \ldots , \alpha_n -1 )$ where $g$ is the genus of $|C|$. The orbifold degree of $K_C $ is given by $- \chi (C)$.
\end{example}

We now record some basic facts that will be used throughout the article. Let $\mathrm{Pic}^t (C)$ stand for the \textit{topological Picard group} of isomorphism classes of orbifold complex line bundles over $C$. 

\begin{lemma}[\cite{furutasteer}] \label{lemma:seifertdata}
The map $\mathrm{Pic}^t (C) \rightarrow \mathbb{Q} \times \prod_{i} \mathbb{Z} / \alpha_i \mathbb{Z}$ given by $N \mapsto (\mathrm{deg}N , \beta_1 , \ldots , \beta_n )$ is an injective group homomorphism. Its image is the subgroup consisting of all $(r , \beta_1 , \ldots , \beta_n )$ such that $r  = \sum_{i} \beta_i / \alpha_i \, \mathrm{mod} \, \mathbb{Z}$. In particular, if the $\alpha_i$ are pairwise coprime then $\mathrm{Pic}^t (C) \cong \mathbb{Z}$ is infinite cyclic.
\end{lemma}

The following examples will come up later on:

\begin{example}\label{seifertdatafamily1}
Consider the family of orbifold closed oriented surfaces $C = S^2 (2,3,6n-5 )$ (i.e. obtained by putting three orbifold points on a $2$-sphere, with isotropies of orders $2,3$ and $6n-5$), $n \geq 2$. The unique generator of $\mathrm{Pic}^t (C) \cong \mathbb{Z}$ with positive orbifold degree has Seifert data $(-2 ; 1, 2 , 5n-4 )$.
\end{example}

\begin{example}\label{seifertdatafamily2}
Consider the family of orbifold closed oriented surfaces $C = S^2 (2,3,6n-1 )$, $n \geq 2$. The unique generator of $\mathrm{Pic}^t (C) \cong \mathbb{Z}$ with positive orbifold degree has Seifert data $(-1 ; 1, 1 , n )$.\end{example}

\begin{lemma}[\cite{furutasteer}]\label{lemma:pullbackline}
Let $Y  = S(N) \xrightarrow{\pi} C$ be a Seifert-fibered $3$-manifold. The kernel of the map $\pi^\ast  : \mathrm{Pic}^t (C) \rightarrow \mathrm{Pic}^{t}(Y) = H^{2}(Y , \mathbb{Z})$, defined by pulling back orbifold complex line bundles from $C$ onto $Y = S(N)$, is generated by $N$. Thus, (cf. Lemma \ref{lemma:seifertdata}) the kernel is infinite cyclic precisely when $\mathrm{deg} N \neq 0$. The torsion subgroup of $\mathrm{Pic}^{t}(Y) = H^{2}(Y, \mathbb{Z})$ is the group $\pi^\ast \mathrm{Pic}^t (C) \cong\mathrm{Pic}^t (C) / \mathbb{Z}[N]$.
\end{lemma}

\subsection{Monodromy of surface singularities}\label{subsection: monodromy}

In this subsection we generalise the discussion from \S \ref{intro:singularities}.

\subsubsection{Basic definitions}

We start by recalling some basic notions on singularities of surfaces. See \cite{popescupampu} for an introductory reference, and \cite{nemethibook,greuel-book} for further details. Throughout this subsection, $V$ will denote a complex-analytic surface (i.e. $\mathrm{dim}_\mathbb{C} V = 2$) with a normal isolated complex-analytic surface singularity at a point $p \in V$ (an ``isolated surface singularity" in short). For our purposes we only care about the \textit{germ} of the pointed complex-analytic space $(V, p )$, and all notions are considered up to germ-equivalence in the discussion that follows. The germ of the singularity $(V,p)$ can be embedded into affine space $(\mathbb{C}^n , 0 )$, where it is given by the vanishing locus of a finite collection of (germs of) holomorphic functions on $(\mathbb{C}^n , 0 )$. 

One can always find a resolution of singularities $\widetilde{V} \xrightarrow{\pi} V$. Furthermore, by blowing up further if needed, it can be ensured that the exceptional divisor $ \pi^{-1}(p) \subset \widetilde{V}$ is a  normal crossings divisor and its irreducible components are smooth (such a resolution is called \textit{good}). The combinatorics of a good resolution is recorded by the \textit{resolution graph} $\Gamma$, given by: one vertex for each irreducible component $E$ of $\pi^{-1}(p)$; whenever two components $E$ and $E^\prime$ intersect there is an edge of multiplicity $E \cdot E^\prime \geq 1$ connecting the vertices $E$ and $E^\prime$, and otherwise no edge; each vertex $E$ is weighted by the genus $g(E) \geq 0$ and self-intersection $E \cdot E <0$ numbers. A neighbourhood of $\pi^{-1}(p) \subset \widetilde{V}$ is diffeomorphic to the $4$-manifold obtained as the \textit{plumbing} of $D^2$-bundles over closed oriented surfaces specified by $\Gamma$. It is a basic fact that $\Gamma$ is negative-definite (i.e. fixing any ordering $E_i$ of the components of $\pi^{-1}(p)$, the matrix with entries $E_i \cdot E_j$ is negative definite). Conversely, any negative-definite such graph $\Gamma$ arises as the good resolution of an isolated surface singularity \cite{grauert}.

Fixing an embedding $(V,p) \subset (\mathbb{C}^n , 0 )$, for  $0 <\epsilon \ll 1$ the intersection $Y = V \cap S_{\epsilon} (\mathbb{C}^n )$ is a $3$-manifold called the \textit{link} of the singularity $(V,p)$, and it carries a (positive, co-oriented) \textit{contact structure} given by the complex tangencies $\xi = TY \cap i TY$. The contact $3$-manifold $(Y, \xi )$ is an invariant of the singularity germ $(V,p)$ \cite{CNP}, i.e. it is independent of the choice of embedding $(V, p ) \subset (\mathbb{C}^n , 0 )$ and $\epsilon \ll 1$. The link $Y$ is the boundary of the plumbing specified by the graph $\Gamma$ associated to a good resolution; furthermore, it is an irreducible $3$-manifold and its homeomorphism type determines uniquely the resolution graph of the minimal good resolution of $(V, p )$ \cite{neumann}.

An important analytic invariant of $(V, p )$ is its \textit{geometric genus} $p_g (V, p ) \in \mathbb{Z}$, which can be defined as the dimension of the complex vector space $H^{1}(\widetilde{U}, \mathcal{O}_{\widetilde{U}})$, where $\widetilde{U}$ is a suitably chosen open neighbourhood of the exceptional divisor of a resolution of singularities of $(V, 0 )$ (see \cite{popescupampu} for the precise definition). In analogy with the Enriques--Kodaira classification of surfaces, an isolated surface singularity $(V, p )$ is called \textit{rational} when $p_g (V, p ) = 0$. 
The isolated surface singularity $(V, p )$ is \textit{Gorenstein} if (up to passing to a smaller germ) there exists a nowhere-vanishing holomorphic $2$-form $\Omega$ defined on the smooth locus $V \setminus p$. If $V = \{ f = 0\} \subset \mathbb{C}^3$ is an isolated hypersurface then $(V, 0 )$ is Gorenstein, with the required $2$-form given by
\[
\Omega = \frac{dx_1 \wedge dx_2 }{\partial f / \partial x_3} =  \frac{ dx_2 \wedge dx_3 }{\partial f / \partial x_1} = \frac{dx_3 \wedge dx_1}{\partial f / \partial x_2} .\]
The \textit{ADE singularities} are the quotient singularities of the form $V = \mathbb{C}^2 / \Gamma$ where $\Gamma$ is a finite subgroup of $SU(2)$. The Gorenstein rational singularities are precisely the ADE singularities \cite{popescupampu}. Beyond rational surface singularities, the next ``simplest" class of isolated surface singularities are the \textit{minimally elliptic} surface singularities introduced by Laufer \cite{laufer}, which can be defined as those Gorenstein singularities with $p_g (V, p ) = 1$. 

It is a remarkable fact that whether or not an isolated surface singularity $(V, p )$ is rational or minimally elliptic only depends on the link $Y$ and not the analytic structure of $(V, p )$, and there exist combinatorial algorithms due to Laufer \cite{laufer-rational,laufer} which determine whether or not a given a resolution graph of a given $(V, p )$ is rational or minimally elliptic. Looking ahead, when $Y$ is a rational homology $3$-sphere, it is known that $(V, p )$ is a rational singularity if and only if $Y$ is an \textit{L-space} \cite{nemethi17,bp} (i.e. $HM^{\redu}(Y) = 0$), and that it is minimally elliptic only if $HM^{\redu}(Y) \cong \mathbb{Z}$ (the latter is proved using lattice homology in \cite{nemethiplumbed}; for the isomorphisms between lattice homology and monopole Floer homology see \cite{zemke,KLT,CGH,taubesECH}) and conjecturally also `if'.

The minimally elliptic hypersurfaces were classified in \cite{laufer} (see Figure \ref{figure:table} for the quasi-homogeneous ones with rational homology $3$-sphere link, which are of special relevance to us) and the resolution graphs of minimally elliptic complete intersections are classified in \cite{degree4}. A simple class of minimally elliptic singularities, which in general are neither hypersurfaces nor complete intersections, is given by the following:

\begin{example}\label{example:S(TC)}
Consider the $3$-manifold $Y = S(TC)$ given by the unit tangent bundle of an orbifold hyperbolic surface $C$ (i.e. with negative orbifold Euler characteristic $\chi (C) < 0$) with genus $0$ (i.e. topologically $C$ is a $2$-sphere). It can be described as the boundary of the plumbing according to the graph shown in Figure \ref{fig:polygon} (see \cite{neumann-raymond} for how to obtain the plumbing description of a Seifert-fibered $3$-manifold). Here the orbifold sphere $C$ has $n$ orbifold points of weights denoted $\alpha_i \geq 2$, and the $\chi(C) < 0$ condition says $\sum_{i = 1}^{n} \frac{1}{\alpha_i} < n-2$. The genus of every vertex of the graph is zero, and all edges have multiplicity $1$. The graph is negative-definite, and hence $Y$ is the link of a isolated surface singularity, which can be verified to be minimally elliptic (see e.g. \cite{nemethibook}). In the case when $C$ is a sphere with three orbifold points we obtain the links $Y$ of a famous class of surface singularities called \textit{Dolgachev's Triangle singularities} (see \cite{looijenga} for their construction). Their links are given by $Y = \widetilde{SL(2, \mathbb{R})} / \Gamma$ where $\Gamma = \Gamma_{\alpha_1,\alpha_2,\alpha_3}$ is the centrally-extended triangle group associated to a hyperbolic triangle with angles $\frac{\pi}{\alpha_i}$ (and $1/\alpha_1 + 1/\alpha_2 + 1 / \alpha_3 < 1$ ), and the universal abelian cover of $Y$ is the Brieskorn sphere $\Sigma (\alpha_1 , \alpha_2 , \alpha_3 ) \cong \widetilde{SL(2, \mathbb{R})}/[\Gamma , \Gamma]$ \cite{milnorbrieskorn}.

\begin{figure}[h!]\label{fig:polygon}

    \centering
 \includegraphics[scale=0.5]{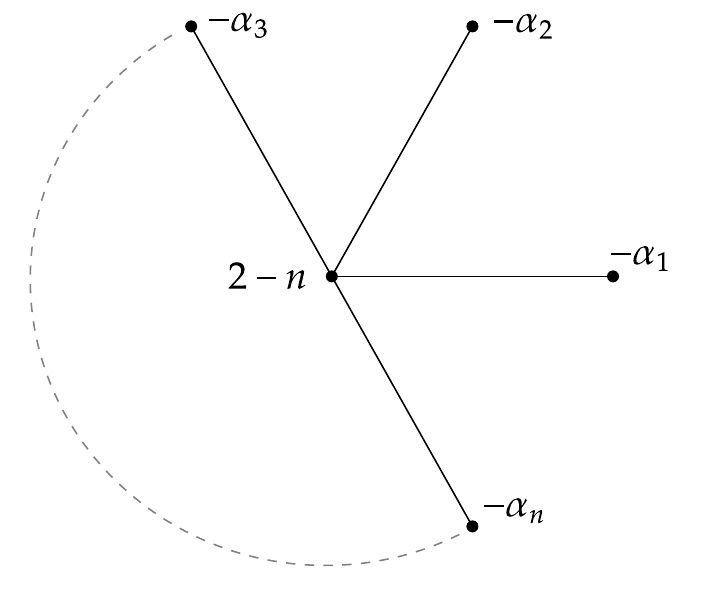}
\caption{Plumbing whose boundary is $Y = S(TC)$}
  \label{fig:polygon}
\end{figure}

\end{example}

\subsubsection{Milnor fibrations}

A deformation of $V$ is a (germ of) flat morphism $f : (\mathcal{V}, p ) \rightarrow (\mathbb{C}, 0 )$ with an identification $V =f^{-1}(0)$. A deformation of $V$ is a \textit{smoothing} when $f^{-1}(t)$ is smooth for sufficiently small $t \neq 0$. Associated to a smoothing is its \textit{Milnor fibration} (\ref{milnorfibration1}). To construct it one fixes an embedding of germs $(\mathcal{V}, p ) \subset (\mathbb{C}^n , 0 )$ and chooses constants $0 < \delta \ll \epsilon \ll 1$; then the mapping   
\begin{align}
f^{-1}(B_{\delta}(\mathbb{C}, 0 ) \setminus 0) \cap B_\epsilon (\mathbb{C}^n , 0 ) \xrightarrow{f}  B_{\delta}(\mathbb{C}, 0 ) \setminus 0\label{milnorfibration1}
\end{align}
is a smooth fiber bundle whose fibers are compact $4$-manifolds with boundary, called the \textit{Milnor fibers} of the smoothing. Milnor \cite{milnor} established this result in the case of a hypersurface singularity, and Lê \ established the more general result \cite{le} (see Hamm \cite{hamm}). Furthermore, the boundary fibration
\begin{align}
f^{-1}(B_\delta (\mathbb{C}, 0 ) )  \cap \partial B_{\epsilon}(\mathbb{C}^n , 0 ) \xrightarrow{f} B_{\delta}(\mathbb{C}, 0 ) \label{boundaryfibration}
\end{align}
is a smooth fibre bundle, whose fiber is diffeomorphic to the link $Y$ of $(V, p)$. It follows that (\ref{boundaryfibration}) has a \textit{canonical trivialisation} up to homotopy because $B_\delta (\mathbb{C}, 0 )$ is contractible. Thus, the Milnor fibration (\ref{milnorfibration1}) is a principal $\mathrm{Diff}(M , \partial )$-bundle over $B_{\delta}(\mathbb{C}, 0 )\setminus 0 \simeq S^1$, where $M$ is any fixed Milnor fiber. We refer the reader to \cite{seade} for a more detailed discussion on the Milnor Fibration Theorem in its various guises.

It is known that $b_1(M) = 0$ \cite{greuel-steenbrink}. The coarse invariants of the intersection form on $H_2(M,\mathbb{R})$ reflect analytic invariants of the singularity or its smoothing. For instance, one has $b^{+}(M) + b^{null} (M) = 2 p_g (V,p)$ \cite{durfee,steenbrink}, where $b^{null}(M)$ is the dimension of the kernel of the intersection form, and $b_1(Y) = b^{null}(M)$ (see also \cite{steenbrink} for a formula for $b^{-}(M)$ in the Gorenstein case).

\subsubsection{$\mathbb{C}^\ast$-actions and Dehn twists}


We discuss the relation between the boundary Dehn twist and the Milnor fibration. This is closely-related to the notion of a $\mathbb{C}^\ast$-smoothing introduced below. 

We first recall some basic definitions, see \cite{neumann-jenkins,orlik-wagreich,neumann-raymond,pinkham} for further details. If $V$ is an affine algebraic variety, then an algebraic $\mathbb{C}^\ast$-\textit{action} can be regarded as a \textit{grading} of its ring of functions $A = \oplus_{i \in \mathbb{Z}} A_i$. We say $V$ is \textit{weighted-homogeneous} if it carries an algebraic $\mathbb{C}^\ast$-action with $A_i = 0$ for $i < 0$ and $A_0 = \mathbb{C}$ (and without loss of generality, we assume that the action is in ``lowest terms", i.e. the greatest common divisor of the collection of integers $i $ with $A_i \neq 0$ is $1$). It follows that the point $p$ associated to the maximal ideal $\oplus_{i > 0}A_i$ is a fixed point and lies in the closure of every orbit of the action (such actions are called ``good''). The basic case is that of a weighted-homogeneous hypersurface, discussed in \S \ref{intro:singularities}. If a germ $(V, p )$ of complex-analytic space with analytic $\mathbb{C}^\ast$-action fixing $p$ is equivalent to the germ of an affine variety with good $\mathbb{C}^\ast$-action, then $(V, p )$ is called \textit{quasi-homogeneous}. 

For a normal isolated quasi-homogeneous surface singularity germ $(V, p )$ the link $Y $ can be identified as $Y = (V \setminus p)/\mathbb{R}_{+}$. The remaining $S^1$-action on $Y$ is fixed-point free; thus $Y$ is a \textit{Seifert-fibered} $3$-manifold over the orbifold Riemann surface $C = (V \setminus p ) / \mathbb{C}^\ast$. Conversely, every Seifert-fibered $3$-manifold of \textit{negative degree} (i.e. $Y = S(N)$ and $\mathrm{deg}N <0$) is the link of a normal isolated quasi-homogeneous surface singularity \cite{neumann-raymond}. For a given Seifert-fibered $3$-manifold $Y = S(N) \rightarrow C$ of negative degree with a given holomorphic structure on the orbifold surface $C$, there is a one-to-one correspondence between normal isolated quasi-homogeneous surface singularities $(V, p )$ with link and underlying orbifold Riemann surface given by $Y$ and $C$ respectively, and holomorphic structures on the orbifold line bundle $N \rightarrow C$, see \cite{neumann-jenkins,pinkham}.

Let $V$ be a weighted-homogeneous affine surface with an isolated singularity at $p$, and consider an algebraic smoothing $f : \mathcal{V} \rightarrow \mathbb{C}$ such that the base $\mathbb{C}$ carries an algebraic $\mathbb{C}^\ast$-action fixing $0$, the total space $\mathcal{V}$ is an affine variety with algebraic $\mathbb{C}^\ast$-action extending the given one on $V = f^{-1}(0 )$, and $f$ is equivariant with respect to these. More generally, if $(V,p)$ is a germ of quasi-homogeneous singularity, by an \textit{equivariant smoothing} of $(V, p )$ (see e.g. \cite{pinkhamdeformations,wahlsmoothings}) 
we mean a complex-analytic smoothing with $\mathbb{C}^\ast$-action equivalent to an algebraic one, as above considered. For instance, every weighted-homogeneous isolated hypersurface or complete intersection singularity admits an equivariant smoothing.

\begin{example}
For integers $n \geq 3$ and $\alpha_1 , \ldots , \alpha_n \geq 2$, the Brieskorn $3$-manifold $Y = \Sigma(\alpha_1 , \ldots , \alpha_n )$ is defined as the link of the following weighted-homogeneous surface singularity $V$. For generic choices of coefficients $a_{ij} \in \mathbb{C}, b_{i} \in \mathbb{C}$, where $1 \leq i \leq n-2$ and $1 \leq j \leq n$ the variety $\mathcal{V} \subset \mathbb{C}^{n}\times \mathbb{C}$ (with variables denoted $x_1 , \ldots , x_n , t $) defined by the equations
\begin{align*}
a_{i1}x_{1}^{a_1} + \ldots + a_{in}x^{a_n} = b_i t \quad , \quad 1 \leq i \leq n-2 
\end{align*}
is a weighted-homoegenous $3$-fold with an isolated singularity at $0$. The projection $f : \mathcal{V} \rightarrow \mathbb{C}$ to the $t$ variable defines an equivariant smoothing of the central fiber $V := f^{-1}(0)$. Here the weights of the $\mathbb{C}^\ast$-action on the variables $x_1, \ldots , x_n$ can be taken to be $w_i = \mathrm{lcm}(a_1 , \ldots , a_n ) / a_i$ and the weight $d$ on $t$ (i.e. the weight on the base of the smoothing) is $d = \mathrm{lcm}(a_1 , \ldots , a_n )$.
\end{example}

The following is a generalisation of the classical fact that the monodromy of the plane curve singularity $\{xy = 0\}$ is a $2$-dimensional Dehn twist on its Milnor fiber $ \approx [0,1] \times S^1 $ along the curve $\{1/2\} \times S^1$,

\begin{proposition}\label{prop:dehnsmoothing}
For any $\mathbb{C}^\ast$-smoothing of an isolated quasi-homogeneous surface singularity $(V, p )$, the $d^{th}$ iterate of the monodromy $\psi \in \MCG (M)$ of its Milnor fibration is smoothly isotopic to the inverse of the boundary Dehn twist on $M$ rel. $\partial M$, i.e. $\psi^{d} = (\tau_M )^{-1}$ in $\MCG(M)$, where $d > 0$ is the weight of the $\mathbb{C}^\ast$-action on the base of the smoothing.
\end{proposition}
\begin{proof}
It suffices to consider the case of an algebraic smoothing $\mathcal{V} \rightarrow \mathbb{C}$ of an isolated weighted-homogeneous affine singularity $V$. 

We first explain how to make the Milnor fibration of $f$ into an $S^1$-equivariant bundle. Fix an affine embedding $V \subset \mathbb{C}^{n-1}$ sending $p$ to $0$. We can realise $\mathcal{V}$ as an embedded deformation, i.e. we can find an embedding of $\mathcal{V}$ in affine space $\mathbb{C}_{x}^{n-1} \times \mathbb{C}_{t}$ extending $V \subset \mathbb{C}^{n-1}_x$ and such that $f$ agrees with the projection to $\mathbb{C}_t$ (see \cite{greuel-book}, Proposition 1.5). Since the $\mathbb{C}^\ast$-action on $\mathcal{V}$ is algebraic then, after a suitable change of coordinates, it can be assumed to be induced by a $\mathbb{C}^\ast$-action on $\mathbb{C}_{x}^{n-1} \times \mathbb{C}_{t}$ in the standard form (see \cite{orlik-wagreich}, Proposition 1.1.3)
\begin{align}
(x_1 , \ldots , x_{n-1} , t ) \mapsto (\lambda^{w_1} x_1 , \ldots , \lambda^{w_{n-1}} x_{n-1} , \lambda^{d} t )\label{eq:C*action}
\end{align}
for some weights $w_i, d \in \mathbb{Z}$. The total space $\mathcal{V} \subset \mathbb{C}^n$ is thus the vanishing locus of polynomial equations $q_{1}(x,t) = 0 , \ldots , q_{l}(x,t ) = 0$ of homogeneous weighted-degree (with respect to the weights $w_i , d$), and the equations for $V \subset \mathbb{C}^{n-1}$ are $p_j (x) := q_j (x,0) = 0$. Since the $\mathbb{C}^\ast$-action on $V$ is good, then we can further assume $w_i > 0$. Since $f : \mathcal{V} \rightarrow \mathbb{C}$ is a smoothing then we must also have $d > 0$. 
With this in place, the ball $B_{\epsilon}(\mathbb{C}^{n} , 0 ) \subset \mathbb{C}^n$ is then invariant under the $S^1$-action given in (\ref{eq:C*action}), and thus the Milnor fibration
\begin{align}
f^{-1}( B_{\delta} (\mathbb{C}, 0 ) \setminus 0 ) \cap B_{\epsilon}(\mathbb{C}^{n}, 0 ) \xrightarrow{f} B_{\delta}(\mathbb{C}, 0 ) \setminus 0 \label{milnorfibration2}
\end{align}
becomes an $S^1$-\textit{equivariant} $\mathrm{Diff}(M , \partial)$-\textit{bundle} over $B_{\delta}(\mathbb{C}, 0 )$.

By pulling back the bundle (\ref{milnorfibration2}) by the degree $d$ basechange $t \mapsto t^d$ we obtain an $S^1$-equivariant bundle denoted $f^\prime : \widetilde{M} \rightarrow B_{\delta^{1/d}}(\mathbb{C}, 0 ) \setminus 0 \simeq S^1$ where the weight of the $S^1$-action on the base is now equal to $1$, and the monodromy is $\delta^d \in \MCG (M)$. The $S^1$-action hence trivialises the bundle $f^\prime$ as a $\mathrm{Diff}(M)$-bundle in an obvious way
\[
M \times S^1 \xrightarrow{\cong} \widetilde{M} \quad , \quad (x, e^{i\theta} ) \mapsto e^{i\theta} \cdot x   \quad 
\]
but this is \textit{not} a trivialisation of $f^\prime$ as a $\mathrm{Diff}(M , \partial )$-bundle, since the $S^1$-action does not preserve the \textit{canonical} trivialisation of the boundary fibration of $f^\prime$ induced from (\ref{boundaryfibration}). The difference between the canonical trivialisation of the boundary fibration and the one induced by the $S^1$-action is precisely the element $\gamma_Y \in \pi_1 \mathrm{Diff}(Y)$ represented by the \textit{Seifert loop} (i.e. the former trivialisation and the latter are related by the bundle isomorphism $Y \times S^1 \xrightarrow{\cong}   Y \times S^1$ given by $(y, e^{i\theta} ) \mapsto (\gamma_Y (e^{i\theta} ) (y) , e^{i\theta} )$ ). The assertion follows.
\end{proof}

Instead of smoothing an isolated surface singularity we can perform its minimal resolution. Then we have the following

\begin{proposition}\label{proposition:minres}
Let $(V, p ) \subset (\mathbb{C}^n , 0 )$ be an isolated quasi-homogeneous normal surface singularity, and let $\widetilde{V} \xrightarrow{\pi} V$ be its minimal resolution. Consider the compact oriented $4$-manifold $M = \widetilde{V} \cap \pi^{-1} ( V \cap B_\epsilon (\mathbb{C}^n , 0 )  )$ with $\partial M = Y$ (here $0 < \epsilon \ll 1$). Then the boundary Dehn twist $\tau_M$ is trivial in $\MCG (M)$.
\end{proposition}
\begin{proof}
The $\mathbb{C}^\ast$-action on $V$ lifts to $\widetilde{V}$, which can be seen as follows. Let $\lambda \in \mathbb{C}^\ast$, and denote by $\rho_\lambda : V \rightarrow V$ map induced by the action of $\lambda$. By the defining property of the minimal resolution there exists a holomorphic map $\widetilde{\rho}_\lambda : \widetilde{V} \rightarrow \widetilde{V}$ such that $\pi \circ \widetilde{\rho}_\lambda = \rho_\lambda \circ \pi $. By the latter identity the map $\widetilde{\rho}_\lambda$ must be unique, as it determines $\widetilde{\rho}_\lambda$ over an open dense subset of $\widetilde{V}$. It then follows that $\widetilde{\rho}_{\lambda}$ defines a $\mathbb{C}^\ast$-\textit{action} on $\widetilde{V}$ lifting that on $V$.

Thus, the Seifert $S^1$-action on $Y$ extends to an $S^1$-action on $M$, and hence the required result follows from Lemma \ref{lemma:loopextension}.
\end{proof}


\subsubsection{Simultaneous resolution}

A special feature of two-dimensional complex surface singularities is the phenomenon of simultaneous resolution. A smoothing $f: (\mathcal{V}, p) \rightarrow (\mathbb{C},0 )$ is said to \textit{resolve simultaneously} if there exists a complex-analytic map $\widetilde{f}: \widetilde{\mathcal{V}} \rightarrow \mathbb{C}$ factoring through $f$ which restricts to a resolution of singularities $\widetilde{V}_t = (\widetilde{f})^{-1}(t) \rightarrow V_t = f^{-1}(t)$ for every $t \in \mathbb{C}$ near $0$. In addition, we say that the smoothing $f$ has a \textit{strong} simultaneous resolution if each resolution $\widetilde{V}_t \rightarrow V_t$ is \textit{minimal} (i.e. $\widetilde{V}_t \rightarrow V_t$ is an isomorphism for $t \neq 0$, and a minimal resolution at $t = 0$). This phenomenon was first observed by Atiyah \cite{atiyah} for the $A_1$ singularity, and later generalised to ADE singularities by Brieskorn \cite{brieskornADE}, and to rational singularities by Wahl \cite{wahl}. The following is a consequence of Wahl's result:

\begin{proposition}[\cite{wahl}] \label{prop:simres}
If a smoothing $f: (\mathcal{V}, p ) \rightarrow (\mathbb{C}, 0 )$ of a normal isolated rational surface singularity $V = f^{-1}(0)$ is parametrised by the Artin component of the semi-universal deformation space of $V$, then $f$ admits a strong simultaneous resolution of singularities after a finite base change, i.e. pulling back by $t \mapsto t^d$ (for some $d > 0)$. 

\end{proposition}

Of course, if a smoothing resolves simultaneously after finite base change then, in particular, its monodromy $\psi$ has \textit{finite} order in the smooth mapping class group $\MCG(M)$. 

\begin{proof}
We refer to \cite{popescupampu} and \cite{wahl} for the definitions of Artin component and semi-universal deformation space. The Artin component $A$ is an irreducible component of the semi-universal deformation space of $(V, p )$. When the singularity is rational then $A$ is smooth and there exists a smooth space $B$ with a Galois branched covering $p : B \rightarrow A = B / W$, where $W$ is certain finite group of deck transformations, with the property that the restriction of the semi-universal family to $A$ simultaneously resolves (strongly) after pullback along $p$, see \cite{wahl}. 
 
 The given smoothing $f$ is obtained by pulling back the semi-universal family carried by $A$ along a (germ of) map $c : (\mathbb{C} , 0 ) \rightarrow A$ (we note that $A$ has a natural basepoint, corresponding to a trivial deformation; we omit this in what follows). We want to show that, after precomposing $c$ by $r_d : t \rightarrow t^d$ for suitable $d> 0$, the map $c \circ r_d$ lifts along $p$ to a map $\widetilde{c} : (\mathbb{C} , 0 ) \rightarrow B$. From this, the required result follows: the simultaneous resolution of $f$ after the base change by $r_d$ is obtained by pulling back along $\widetilde{c}$ the simultaneous resolution carried by $B$.
 
 To establish the lifting assertion, let $\Sigma$ be the complex-analytic curve given as the fiber product of $c$ and $p$. If $\Sigma$ is reducible, then we replace it with any of its irreducible components. Choosing a local complex-analytic parametrisation of $\Sigma$ gives an analytic map $r : (\mathbb{C}, 0 ) \rightarrow \Sigma$, which we can compose with the natural map $\Sigma \rightarrow (\mathbb{C}, 0 )$ (coming from the fiber product) to obtain a complex-analytic map $(\mathbb{C} , 0 ) \rightarrow (\mathbb{C} , 0 )$. The latter, after a complex-analytic reparametrisation of the source, is of the form $r_d$ for some $d > 0$. Taking $\widetilde{c}$ to be the composition of $r $ with the natural map $\Sigma \rightarrow B$ gives a lift of $c \circ r_d$ along $p$, as claimed.
\end{proof}

 In general, there can exist smoothings of rational singularities which do not admit simultaneous resolution after any finite base change (see \cite{pinkhamdeformations}). However, in certain situations, such as when the rational surface singularity $(V, p )$ is a hypersurface or a complete intersection, then the Artin component of $V$ is the whole semi-universal deformation space, and all smoothings of $(V, p)$ have simultaneous resolution after base change. In contrast with the above, the non-rational singularities do not enjoy the simultaneous resolution property:

\begin{proposition}\label{prop:nosimresolution}
If a normal isolated surface singularity $(V, p )$ is not a rational singularity (i.e. $p_g (V,p )> 0$) then no smoothing of $(V, p)$ admits a simultaneous resolution.
\end{proposition}
\begin{proof}
If such a smoothing existed then the central fiber of the simultaneous resolution, which is diffeomorphic to $\widetilde{V} \#_{k} \overline{\mathbb{C}P^2}$ where $\widetilde{V}$ is the minimal resolution of $V$ and $k \geq 0$, must be diffeomorphic to the Milnor fiber $M$ of the smoothing. However, $\widetilde{V}$ is always negative-definite, whereas $M$ is not negative-definite when $(V, p )$ is non-rational (\cite{popescupampu}, Corollary 4.21), giving a contradiction.
\end{proof}

For example, the smoothing of the non-rational singularity $x^2 + y^3 + z^7 = 0$ given by $x^2+ y^3 + z^7 =t$ does not resolve simultaneously after base change, but the analogous smoothing for $x^2 + y^3 + z^5 = 0$ does.

\begin{proposition}\label{prop:finiteordertop}
Let $(V , p )$ be a normal isolated surface singularity and consider a smoothing with monodromy $\psi$. Suppose further that the Milnor fiber $M$ is simply-connected, the link $Y$ is a rational homology $3$-sphere, and that $\psi$ acts on $H_\ast (M , \mathbb{Z} )$ with finite order. Then $\psi$  has finite order in the topological mapping class group $\pi_0 \mathrm{Homeo}(M, \partial )$.
\end{proposition}
\begin{proof}
This is a direct application of the result of Orson--Powell \cite{orson-powell}.
\end{proof}

The simple-connectivity assumption on $M$ holds, for instance, if $(V, 0 )$ is a hypersurface or, more generally, a complete intersection. By Proposition \ref{prop:dehnsmoothing}, the $\mathbb{C}^\ast$-smoothings of normal isolated quasi-homogeneous singularities always satisfy that $\psi$ acts on $H_\ast (M, \mathbb{Z} )$ with finite order.


Proposition \ref{prop:simres} and Proposition \ref{prop:finiteordertop} fall into a \textit{flexibility} paradigm, in the analytic category for the former, and the topological category for the latter. In turn, Proposition \ref{prop:nosimresolution} exhibits a \textit{rigidity} phenomenon in the analytic category. These contrasting behaviours motivate the following:
\begin{question}[generalisation of Question \ref{ques : mono infinite order}]\label{ques : mono infinite order 2}
Let $(V, p )$ be a normal isolated surface singularity with rational homology $3$-sphere link, together with a smoothing whose monodromy $\psi$ acts with finite order on $H_\ast (M , \mathbb{Z} )$. Does $\psi$ have infinite order in the \textbf{smooth} mapping class group $\MCG (M)$?
\end{question}
When the answer is affirmative, the lack of simultaneous resolutions (which is an analytic phenomenon) has a smooth explanation. We provide the first known examples of such phenomenon:

\begin{corollary}[of Theorem \ref{thm: main}, generalising Corollary \ref{cor:infiniteordermonodromy}] \label{cor:infiniteordermonodromy2} 
The answer to Question \ref{ques : mono infinite order 2} is affirmative for an equivariant smoothing of a quasi-homogeneous minimally elliptic surface singularity.
\end{corollary}

We propose the following conjecture:

\begin{conjecture}[generalising Conjecture \ref{conj:infiniteordermono}]\label{conj:infiniteordermono2}
The answer to Question \ref{ques : mono infinite order 2} is affirmative when $p_g (V, p ) > 0$.

\end{conjecture}

\subsubsection{The higher dimensional case}

We record here a result about monodromies of high dimensional singularities, which is in sharp contrast with Corollaries \ref{cor:infiniteordermonodromy} and \ref{cor:infiniteordermonodromy2}.

\begin{proposition}\label{highdim}
Let $V = \{f = 0 \} \subset \mathbb{C}^{n+1}$ be an $n$-dimensional weighted-homogeneous isolated hypersurface singularity with $n > 4$. Let us further suppose that the link $Y$ is an integral homology $(2n-1)$-sphere (which is equivalent to the invertibility of $1- \psi_\ast$ over $\mathbb{Z}$ \cite{milnor}). Then the monodromy $\psi$ of the Milnor fibration of $f$ has finite order in $\MCG (M)$.
\end{proposition}

 \begin{remark}
By Corollary \ref{cor:infiniteordermonodromy}, the assertion above fails when $n = 2$ (e.g. consider either of the singularities $x^2 + y^3 + z^7 = 0$ or $x^2 + y^3 + z^{11} = 0$). We do not know if the same conclusion as above holds in the cases $n = 3, 4$.
\end{remark}

\begin{proof}
By Proposition \ref{prop:dehnsmoothing} (which holds in all dimensions even if it was stated for $n = 2$), a suitable non-trivial power of $\psi$ becomes a boundary Dehn twist on $M$, induced by an element of $\pi_1 \mathrm{Diff}(Y)$ (this element is represented by an $S^1$-action). When $n > 2$ then $\pi_1 Y = 0$ \cite{milnor} and hence $Y$ is a homotopy $(2n-1)$-sphere. Now, by the h-cobordism Theorem any homotopy sphere in dimensions above four is obtained by gluing together two standard disks by a diffeomorphism of their boundaries. Hence, there is a fibration sequence
\[
\mathrm{Diff}(D^{2n-1} , \partial) \rightarrow \mathrm{Diff}(Y) \rightarrow \mathrm{Emb}(D^{2n-1} , Y ) \simeq \mathrm{Fr}(Y) 
\]
where 
$\mathrm{Fr}(Y)$ denotes the total space of the oriented frame bundle of $Y$, which has finite $\pi_1 \mathrm{Fr}(Y)$. 
When $n > 4$ it is known that $\pi_1 \mathrm{Diff}(D^{2n-1} , \partial ) \otimes \mathbb{Q} = 0$ \cite{krannich-rw} (but the answer is not known for $n = 3$ or $4$); hence it follows that $\pi_1 \mathrm{Diff}(Y)$ is torsion and so the Dehn twist on $M$ (and thus also $\psi$) has finite order in $\MCG(M)$.
\end{proof}

\subsubsection{ADE singularities}

We record here the following result, in sharp contrast with Corollary \ref{symplectic torelli}, about the monodromy of the Milnor fibration of an $ADE$ singularity (i.e. a $2$-dimensional isolated hypersurface singularity with $p_g = 0$). The proof should be well-known to experts, but we didn't find a convenient reference. We denote by $\mathrm{Symp}(M , \omega , \partial )$ the topological group of symplectomorphisms of $(M, \omega )$ which agree with the identity near the boundary.

\begin{proposition}\label{prop:ADE}
Let $(M, \omega )$ be the Milnor fiber of an $ADE$ singularity, where $\omega $ is the symplectic structure with convex contact boundary obtained by restriction of the standard symplectic form from $\mathbb{C}^3$. Let $k < \infty$ be the order of the Weyl group of the corresponding $ADE$ graph. Then the $k^{th}$ iterate of the symplectic monodromy $\psi \in \pi_0 \mathrm{Symp}(M , \omega ,\partial )$ of the Milnor fibration is fragile and factorises as a product of squared Dehn--Seidel twists.
\end{proposition}
\begin{proof}
That $\psi$ has infinite order in $\pi_0 \mathrm{Symp}(M , \omega , \partial )$ was proved by Seidel (\cite{seidelgraded}, Example 4.18 (a) ). By Proposition \ref{prop:simres} a degree $k$ basechange of the Milnor fibration of an ADE singularity admits a strong simultaneous resolution, and then it follows that $\psi^k$ is fragile (see \cite{seidel}). The factorisation assertion can be verified as follows. Let $\Gamma$ denote the minimal resolution graph of an ADE singularity, and consider the associated \textit{Artin-Tits group} $A(\Gamma)$ with one generator $s_x$ for each vertex $x \in \Gamma$ and subject to the generalised Braid relations
\begin{align*}
s_x s_y s_x = s_y s_x s_y &\text{  if there is an edge of $\Gamma$ connecting $x$ to $y$}\\
s_x s_y = s_y s_x &\text{  if there is no edge of $\Gamma$ connecting $x$ to $y$}.
\end{align*}
Let $W(\Gamma )$ be the \textit{Weyl group} of $\Gamma$, which is a finite group with a presentation consisting of the same generators and relations as $A(\Gamma )$ together with the additional relation $s_{x}^2 = 1$ for every vertex $x $. Let $P(\Gamma ) \subset A(\Gamma )$ denote the kernel of the natural map $A(\Gamma ) \rightarrow W(\Gamma )$. For example, if $\Gamma$ is the $A_n$ graph then $A(\Gamma ) $ is the \textit{Braid group} on $n+1$ strands, $W(\Gamma )$ is the \textit{symmetric group} $S_{n+1}$, and $P(\Gamma )$ is the \textit{pure Braid group} on $n+1$ strands.

By definition of $P(\Gamma )$ it follows that the elements $s_{x}^2$, where $x$ ranges over all vertices, form a \textit{normal generating set} for $P(\Gamma)$, i.e. $P(\Gamma )$ is generated by the elements $g s_{x}^2 g^{-1} $ ranging over all vertices $x$ and $g \in W(\Gamma )$.

With this in place, it is well-known that the Milnor fiber $(M, \omega )$ has Lagrangian vanishing $2$-spheres $L_x$ corresponding with vertices $x$ of $\Gamma$ and intersecting transversely in a $\Gamma$-configuration. Now, Seidel proved that if two Lagrangian $2$-spheres $L_1, L_2$ in an arbitrary symplectic $4$-manifold intersect transversely at a single point then the Seidel--Dehn twists satisfy the relation $\tau_{L_1} \tau_{L_2}\tau_{L_1} = \tau_{L_2}\tau_{L_1} \tau_{L_2}$ in the symplectic mapping class group (Proposition A.1 in \cite{seidelknotted}); and if $L_1 , L_2$ are disjoint then their Seidel--Dehn twists obviously commute. Putting everything together, we have that there is a well-defined group homomorphism
\begin{align}
A(\Gamma ) \rightarrow \pi_0 \mathrm{Symp}(M , \omega , \partial )\quad , \quad s_x \mapsto \tau_{L_x} .\label{eq:rep}
\end{align}

The symplectic monodromy $\psi$ factorises as a product of Dehn--Seidel twists on the $L_x$'s, and hence lies in the image of the homomorphism (\ref{eq:rep}). Since $k$ is, by definition, the order of $W(\Gamma )$, then $\psi^k$ lies in the image of $P(\Gamma )$ under (\ref{eq:rep}). The fact that $P(\Gamma )$ is normally generated by the $s_{x}^2$, combined with the fact that for any symplectomorphism $f$ and a Lagrangian sphere $L$ we have $f \tau_{L} f^{-1} = \tau_{f(L)}$ in the symplectic mapping class group establishes the following: that $\psi^k$ factorises in $\pi_0 \mathrm{Symp}(M , \omega , \partial )$ as product of squared Dehn--Seidel twists $\tau_{L}^2$ along Lagrangian spheres $L \subset (M , \omega )$, where each $L$ is obtained from a vertex sphere $L_x$ by applying to it Dehn--Seidel twists along vertex spheres.
\end{proof}

\subsection{Some Dehn twists on elliptic surfaces}

For the remainder of this section we will discuss a quite natural example. There are famous smooth embeddings of the following families of Milnor fibers of Brieskorn singularities into the elliptic surface $E(n)$ (see \cite{gompfstipsicz}) 
\[
M(2,3,6n-5 )\,  , \, M(2,3,6n-1) \hookrightarrow E(n) \quad , \quad n \geq 2 .
\]
The complement $N(n) = E(n) \setminus M(2,3,6n-1)$ is the \textit{Gompf nucleus}. It is natural to ask whether the Dehn twists along $\Sigma (2,3,6n-5) = \partial M(2,3,6n-5)$ or $\Sigma(2,3,6n-1) = \partial M(2,3,6n-1 ) $ yield non-trivial diffeomorphisms of $E(n)$, and we take up this task in this subsection. Let $\tau_{\Sigma(2,3,6n-5)}, \tau_{\Sigma(2,3,6n-1)}$ denote these diffeomorphisms of $E(n)$. Let $B \subset E(n)$ be any smoothly embedded $4$-ball disjoint from $\Sigma(2,3,6n-5)$ and $\Sigma(2,3,6n-1)$, and let $\tau_{S^3} \in \MCG (E(n) \setminus B )$ denote the boundary Dehn twist on $E(n) \setminus B$ constructed from the non-trivial element in $\pi_1 \mathrm{Diff}_0 (S^3) = \mathbb{Z}/2$ (note that this element is \textit{not} the Seifert loop, which is trivial in this case). 

\begin{proposition}\label{proposition:E(n)}
For every $n \geq 2$ we have 
\begin{enumerate}
\item[(i)] The Seifert $S^1$-action on $\Sigma(2,3,6n-5)$ extends smoothly over to the complement $E(n) \setminus M(2,3,6n-5)$; in particular, the Dehn twist $\tau_{\Sigma(2,3,6n-5)}$ is trivial in both $\MCG (E(n) \setminus M(2,3,6n-5) )$ and $\MCG (E(n))$. The analogous statement holds for the Dehn twist along $\Sigma(2,3,6n-1)$.
\item[(ii)] All three Dehn twists $\tau_{\Sigma(2,3,6n-5)}$, $ \tau_{\Sigma(2,3,6n-1)}$, $ \tau_{S^3}$ represent the same element in $\MCG (E(n) \setminus B )$.
\end{enumerate}
\end{proposition}

The case $n = 2$ has attracted special attention, since $E(2) = K3$ is the K3 surface. By work of Baraglia and the first author \cite{baraglia-konno} and Kronheimer--Mrowka \cite{KM-dehn} we have that $\tau_{S^3}$ has order two in $\MCG (K3 \setminus B )$. Let $M = M(2,3,7)$ or $M(2,3,11)$ and choose $B $ to be contained in the complement of $M $ in $K3$, which gives a homomorphism $\MCG (M(2,3,7)) \rightarrow \MCG (K3 \setminus B)$ by extending as the identity. Then the aforementioned result combined with Proposition \ref{proposition:E(n)}(ii) show that the boundary Dehn twist on $M$ must have order \textit{at least} two, thus recovering a result already established by different means in work of Mallick, Taniguchi and the first author \cite{konno-mallick-taniguchi}. The main result of this article, Theorem \ref{thm: main}, show that this has infinite order, in fact. 

As a consequence of Proposition \ref{proposition:E(n)} and Theorem \ref{thm: main} we detect an exotic phenomenon related to the smooth embedding space of $\Sigma(2,3,7)$ and $\Sigma(2,3,11)$ into $E(2) = K3$.

\begin{corollary}
Let $Y = \Sigma(2,3,7)$ or $\Sigma(2,3,11)$, embedded into $E(2) = K3$ as above described. There exists an infinite order element in the kernel of the map
\[
\pi_1 \mathrm{Emb}(Y , K3 ) \rightarrow \pi_1 \mathrm{Emb}_{top}(Y , K3 ).
\]
\end{corollary}

Here $\mathrm{Emb}(Y , K3 )$ (resp. $\mathrm{Emb}_{top}(Y , K3 )$) denotes the space of smooth (resp. topological and locally flat) embeddings $Y \hookrightarrow K3$. The exotic loop $e_t : Y \hookrightarrow K3$ of embeddings is obtained by considering a parallel copy $Y^\prime$ of $Y$ in $K3 \setminus M$ (where $M = M(2,3,7)$ or $M(2,3,11)$), together with a path of diffeomorphisms $f_t$ of $K3$ with $f_0 = \tau_Y$ and $f_1 = \mathrm{id}_{K3}$ (given by Corollary \ref{proposition:E(n)}), and setting $e_t = f_t \circ e_0$ where $e_0 : Y \hookrightarrow K3$ is the given embedding of $Y^\prime$. Since the boundary Dehn twist on $M$ is topologically trivial, then the loop $e_t$ is trivial in the space of topological embeddings; but if $e_t$ or any iterates of it were trivial in the smooth embedding space then it is straightforward to see that some power of the boundary Dehn twist on $M$ must become trivial in $\MCG(M)$, contradicting Theorem \ref{thm: main}.  

The remainder of this subsection is devoted to the proof of Proposition \ref{proposition:E(n)}.


\subsubsection{Compactification of Milnor fibers}\label{subsection:compactification}

We now recall some well-known compactifications of Milnor fibers, obtained by taking their closure inside of weighted projective spaces (see \cite{dolgachev82}). This will be used to establish Proposition \ref{proposition:E(n)}.

Fix three pairwise coprime integers $\alpha_1 , \alpha_2 , \alpha_3 \geq 2$, where without loss of generality $\alpha_2 , \alpha_3$ are odd. and consider the \textit{weighted projective space}
\[
\mathbb{P} := \mathbb{P}(\alpha_2 \alpha_3 , \alpha_1 \alpha_3 , \alpha_1 \alpha_2 , 1)
\]
Consider the \textit{Fermat hypersurface} in $\mathbb{P}$ given by the weighted homogeneous equation
\[
X = \{ x_{1}^{\alpha_1} + x_{2}^{\alpha_2} + x_{3}^{\alpha_3} + x_{4}^{\alpha_1 \alpha_2 \alpha_3} = 0  \} \subset \mathbb{P} .
\]
Both $X $ and $\mathbb{P}$ are complex orbifolds with with exactly three orbifold points
\[
0:1:  -1 :0 \quad , \quad 1: 0 : -1:0 \quad \text{and } \,\, 1:-1:0 :0 .
\]

The weighted projective space $\mathbb{P}$ has a tautological orbifold holomorphic line bundle denoted $\mathcal{O}_{\mathbb{P}}(-1)$, and $\mathcal{O}_{\mathbb{P}}(n)$ denoted the $n$th tensor power of its inverse for any given $n \in \mathbb{Z}$. For instance, $x_4$ is the only non-trivial holomorphic section of $\mathcal{O}_{\mathbb{P}}(1)$ up to scaling. Its vanishing locus $C := X \cap \{x_4= 0\}$ is a complex orbifold curve with three orbifold points (the same as the ones shown above) of weights $\alpha_1 , \alpha_2 , \alpha_3$. In turn, the locus $\mathbb{P} \cap \{ x_4 \neq 0\} = \mathbb{C}^3$ is an affine chart, and 
\[X \cap \mathbb{C}^3 = \{ (x_1 , x_2 , x_3 ) \in \mathbb{C}^3 \, | \, x_{1}^{\alpha_1} + x_{2}^{\alpha_2} + x_{3}^{\alpha_3} + 1 = 0\}\] is diffeomorphic to the (interior of the) Milnor fiber $M(\alpha_1 , \alpha_2 , \alpha_3 )$ of the Brieskorn--Pham singularity $x_{1}^{\alpha_1} + x_{2}^{\alpha_2} + x_{3}^{\alpha_3} = 0 $. 

Thus $X$ is an orbifold compactification of the Milnor fiber, obtained by adding the orbifold divisor $C$. The curve $C$ has genus zero (see \cite{dolgachev82}, \S 3.5.2) and its orbifold normal bundle in $X$ is $\nu_C \cong \mathcal{O}_{\mathbb{P}}(1)|_C$. By the orbifold version of the Tubular Neighbourhood Theorem, the orbifold $X$ is diffeomorphic to $M(\alpha_1, \alpha_2 , \alpha_3 )$ glued with the unit disk bundle $D(\nu_C)$,
\[
X = M(\alpha_1 , \alpha_2 , \alpha_3 ) \cup D(\nu_C) \quad .
\]
In particular, there must be an orientation-reversing diffeomorphism of $ \partial M(\alpha_1 , \alpha_2 , \alpha_3 ) = \Sigma(\alpha_1 , \alpha_2 , \alpha_3 ) $ and $ \partial D (\nu_C )$. This could also be verified directly, as the following shows. Recall (see Lemma \ref{lemma:seifertdata}) that the topological Picard group of $C$ is infinite cyclic, generated by the unique orbifold line bundle $N \rightarrow C$ of degree $- \frac{1}{\alpha_1 \alpha_2 \alpha_3}$, and $\Sigma(\alpha_1 , \alpha_2 , \alpha_3 ) = S(N)$. 


\begin{lemma}

The orbifold normal bundle $\nu_C$ is topologically isomorphic to $N^{-1}$.
\end{lemma}
\begin{proof}
By Theorem 3.3.4 in \cite{dolgachev82}, the canonical bundle of $C$ can be calculated as
\[
K_C \cong \mathcal{O}_{\mathbb{P}}(\alpha_1 \alpha_2 \alpha_3 - \alpha_2 \alpha_3  - \alpha_1 \alpha_3 - \alpha_1 \alpha_2 )|_C = \nu^{ \alpha_1 \alpha_2 \alpha_3 - \alpha_2 \alpha_3  - \alpha_1 \alpha_3 - \alpha_1 \alpha_2 }.
\]
On the other hand, the degree of $K_C$ is given by $ 1 - 1/\alpha_1 - 1/\alpha_2 - 1/\alpha_3$. Putting all together, the degree of $\nu_C$ is then $\mathrm{deg}\nu_C = \frac{1}{\alpha_1 \alpha_2 \alpha_3}$.
\end{proof}

The unit disk bundle $D(\nu_C )$ has cyclic quotient singularities at the three points shown above which lie on the zero section $C$. A \textit{cyclic quotient singularity} is the singularity $N_{\alpha ; 1 , \beta } = \mathbb{C}^{2} / \mathbb{Z}_\alpha$ where $\mathbb{Z}_\alpha$ acts on $\mathbb{C}^2$ with weight $(1 , \beta )$ for some $0 \leq \beta < \alpha$. Equivalently $N_{\alpha ; 1 , \beta}$ is the total space of the orbifold line bundle over $\mathbb{C} / \mathbb{Z}_{\alpha}$ with Seifert data $( 0 ; \beta )$. In particular, the singularity $N_{n+1 ; 1, n}$ is the $A_n$ \textit{singularity} ($n \geq 1$). In \S \ref{subsubsection:quotient} below we recall the construction of the minimal resolution of the quotient singularity $N_{\alpha ; 1 , \beta}$, whose exceptional locus is given by a chain of $m$ rational curves $R_1 , \ldots , R_m$ of self-intersections $-b_1 , \ldots , - b_m$, where $m$ and the $b_{i} \geq 2 $ are given by the Hirzebruch--Jung continued fraction expansion of $\alpha / \beta$ as in (\ref{continuedfraction}).

Thus, if the bundle $N^{-1}$ has Seifert data
$(e ; \beta_1 , \beta_2 , \beta_3 )$, then as a variety $X$ has three quotient singularities $N_{p_i ; 1 , \beta_i }$ ($i = 1 , 2 , 3$). After performing the minimal resolution of each of them (i.e. replace each singular point $x \in X$ by the chain of rational curves $R_{x,1} , \ldots , R_{x, m_x}$ from the minimal resolution) we obtain a smooth complex surface $\widetilde{X}$. 
The strict transform $\widetilde{C} \subset \widetilde{X}$ of the curve $C$ is the \textit{curve at infinity}. It is given by a star-shaped configuration of rational curves: the central curve has self-intersection $e$; there are 3 chains of rational curves emanating from the central curve, with self-intersection numbers given by the coefficients of the Hirzebruch--Jung continued fraction expansions of $\alpha_i / \beta_i$ for $i = 1,2,3$.

For the families $M(2,3,6n-5)$ or $M(2,3,6n-1)$, $n \geq 2$, the Seifert data of $N^{-1}$ was calculated in Examples \ref{seifertdatafamily1}-\ref{seifertdatafamily2}. From this, one can determine the curve $\widetilde{C}$, which is shown in Figure \ref{fig:curves}.

\begin{figure}[h!]
    \centering
\caption{The curve $\widetilde{C}$}\label{fig:curves}    
    \begin{subfigure}[b]{0.4\textwidth}
        \includegraphics[width=\textwidth]{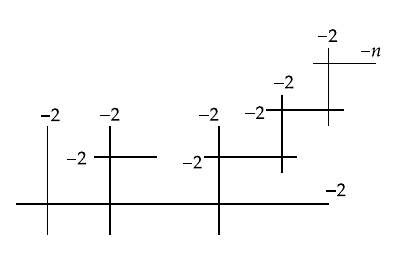}
        \caption{for $M(2,3,6n-5)$, $n \geq 2$}
        \label{fig:gull}
    \end{subfigure}
    ~ 
    \begin{subfigure}[b]{0.4\textwidth}
        \includegraphics[width=\textwidth]{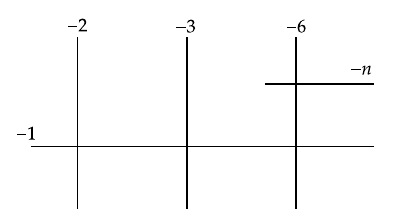}
        \caption{for $M(2,3,6n-1)$, $n \geq 2$}
        \label{fig:tiger}
    \end{subfigure}
    ~ 

\end{figure}

The complex surface $\widetilde{X}$ associated to $M(2,3,6n-5)$ is diffeomorphic to the elliptic surface $E(n)$ (see \cite{ebeling-okonek}). In turn, the complex surface $\widetilde{X}$ associated to $M(2,3,6n-1)$ is \textit{not} minimal, since the central curve of $\widetilde{C}$ is a $(-1)$-sphere. After succesively blowing down the $(-1)$-spheres in $\widetilde{X}$ lying on the curve at infinity, as shown in Figure \ref{fig:blowdowns}, we arrive at a new complex surface $X^\prime$ which is again diffeomorphic to $E(n)$ (see \cite{ebeling-okonek}). 

\begin{figure}[h!]

    \centering
 \includegraphics[scale=0.9]{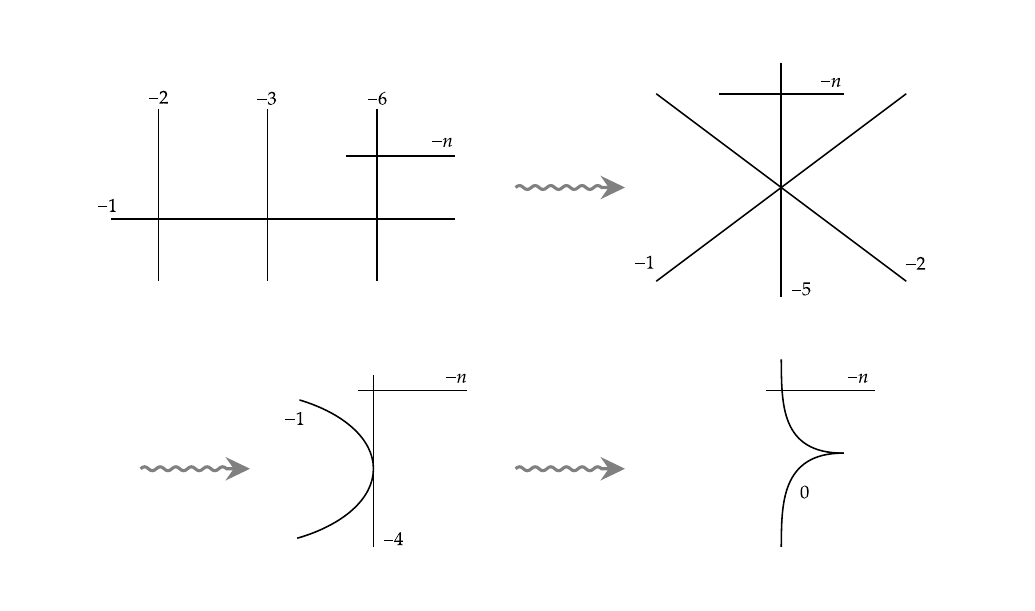}
\caption{Blowing down $(-1)$-curves to get $E(n)$}
  \label{fig:blowdowns}
\end{figure}

\subsubsection{Extending the $S^1$-action over $\widetilde{X} \setminus M(\alpha_1 , \alpha_2 , \alpha_3 )$}


The following is the main ingredient in the proof of Proposition \ref{proposition:E(n)}(i).

\begin{proposition}\label{proposition:extension}
Let $\mathcal{L} \rightarrow C$ be any holomorphic orbifold line bundle over an orbifold Riemann surface $C$. Then the weight one $\mathbb{C}^\ast$-action on on the fibers of $\mathcal{L}$ lifts to the variety $\widetilde{\mathcal{L}}$ obtained by performing the minimal resolution of the cyclic quotient singularities of the total space of $\mathcal{L}$. \end{proposition}

Thus, by applying Proposition \ref{proposition:extension} to the normal bundle $\mathcal{L} = \nu_C$ as above, we have that

\begin{corollary}\label{corollary:extension}
The Seifert $S^1$-action on $\Sigma(\alpha_1 , \alpha_2 , \alpha_3 )$ extends smoothly over to the complement $\widetilde{X} \setminus M(\alpha_1 , \alpha_2, \alpha_3 )$. 
\end{corollary}

\begin{proof}[Proof of Proposition \ref{proposition:extension}]
This follows by the argument in the proof of Proposition \ref{proposition:minres}. Alternatively, it can deduced in a direct way from the description of the minimal resolution of a quotient singularity, which we recall below.
\end{proof}

In order to prove Proposition \ref{proposition:E(n)}(ii) we will need to access more specific information about the $S^1$-action on the complement $\widetilde{X} \setminus M(\alpha_1 , \alpha_2 , \alpha_3 )$. For this we now review the construction of the minimal resolution of the cyclic quotient singularity $N_{\alpha ; 1 , \beta}$.

\subsubsection{Minimal resolution of a cyclic quotient singularity}\label{subsubsection:quotient}

We review the proof of the following 

\begin{lemma}[Lemma 3 in \cite{fujiki}]\label{lemma:opencover}
Fix integers $0 \leq \beta < \alpha$. There exists a variety $X$ with a proper birational morphism $X \xrightarrow{f} N_{\alpha ; 1 , \beta}$, and a Zariski open cover $\{U_1 , U_2\}$ such that $U_1 \cong \mathbb{C}^2$ and $U_2 \cong N_{\beta ; 1 , \gamma}$ where $\gamma \in \mathbb{Z}$ is the unique integer with $0 \leq \gamma < \beta$ and
\begin{align}
\alpha / \beta  = b -  \frac{1}{ \beta / \gamma} \label{equation:fraction}
\end{align}
for some integer $b \geq 2$.
\end{lemma}

Iterated application of Lemma \ref{lemma:opencover} results in the construction of a resolution of the quotient singularity $N_{\alpha ; 1 , \beta}$. Namely, applying Lemma \ref{lemma:opencover} to $N_{\alpha ; 1 , \beta}$ yields the variety $X$ with an open cover $U_1 = \mathbb{C}^2$, $U_2 = N_{\beta ; 1 , \gamma}$. If $\gamma = 0$ then $N_{\beta; 1, \gamma} = (\mathbb{C}/\mathbb{Z}_\beta ) \times \mathbb{C} \cong \mathbb{C}^2$ is also an affine chart of $X$, and thus $X$ is smooth. If $\gamma > 0$ then apply again Lemma \ref{lemma:opencover} to $U_2 = N_{\beta ; 1 , \gamma}$, etc. This way, one obtains a resolution of singularities $\widetilde{ N_{\alpha ; 1 , \beta }} \rightarrow N_{\alpha ; 1 , \beta}$ where $\widetilde{ N_{\alpha ; 1 , \beta }}$ is covered by $m+1$ affine charts, where $m$ is the length of the Hirzebruch--Jung continued fraction expansion of $\alpha / \beta$,
\begin{align}
\frac{\alpha}{\beta} = b_1 - \frac{1}{b_2 - \frac{1}{b_3 - \cdots \frac{1}{b_m}}} \quad , \quad b_i \geq 2 . \label{continuedfraction}
\end{align}
It is easy to see from the construction given in the proof of Lemma \ref{lemma:opencover} that in the resolution of singularities $ \widetilde{ N_{\alpha ; 1 , \beta }}\rightarrow N_{\alpha ; 1 , \beta}$ the exceptional divisor consists of a chain of $m$ rational curves $R_1 , \ldots , R_m$ with self-intersections given by $- b_1 , \ldots , - b_m $, with $R_i$ intersecting $R_{i+1}$ transversely at a single point for each $i = 1, \ldots , m-1$, and with no additional intersections among the $R_i$'s (see \S1.4.1 in \cite{fujiki}). Thus, $\widetilde{ N_{\alpha ; 1 , \beta }}$ is the \textit{minimal} resolution of $N_{\alpha ; 1 , \beta }$.

We now review the proof of Lemma \ref{lemma:opencover}, following \cite{fujiki}. Our goal is to establish the diagram from (\ref{diagramX}) below. 

\begin{proof}[Proof of Lemma \ref{lemma:opencover}] Let $\mathbb{C}^{2}_{z_1 , z_2}$ denote $\mathbb{C}^2$ with coordinates $z_1 , z_2$, and write $N_{\alpha ; 1, \beta } = \mathbb{C}^{2}_{z_1 , z_2} / G$ where $G = \mathbb{Z}_\alpha$ acts on $\mathbb{C}^{2}_{z_1 , z_2}$ with weight $(1 , \beta )$. Consider the branched cover $\mathbb{C}^{2}_{t_1 , t_2} \rightarrow \mathbb{C}^{2}_{z_1 , z_2}$ of degree $\beta$ given by $(z_1 , z_2 ) = (t_1  , t_{2}^{\beta} )$. Then the $G = \mathbb{Z}_{\alpha}$-action on $\mathbb{C}^{2}_{z_1 , z_2}$ lifts to the weight $(1 , 1 )$ action on $\mathbb{C}^{2}_{t_1 , t_2}$. In addition, the branched covering has automorphism group $H = \mathbb{Z}_{\beta}$ acting on $\mathbb{C}^{2}_{t_1 , t_2}$ with weight $(0 , 1 )$ ; and thus $\mathbb{C}^{2}_{t_1 , t_2}/ G \times H  = N_{\alpha ; 1 , \beta }$.

Next, we blow up $\mathbb{C}^{2}_{t_1 , t_2}$ at the origin and obtain $W = \mathrm{Bl}_{(0,0)}\mathbb{C}^{2}_{t_1 , t_2} \xrightarrow{\sigma} \mathbb{C}^{2}_{t_1 , t_2}$. The blowup $W$ can be described as the subvariety $W \subset \mathbb{C}^{2}_{t_1 , t_2} \times \mathbb{P}^{1}_{\xi_1 , \xi_2}$ (where $\mathbb{P}^{1}_{\xi_1  , \xi_2}$ stands for the projective line $\mathbb{P}^1$ with homogeneous coordinates $\xi_1 : \xi_2 )$ cut out by the equation $t_1 \xi_2 = t_2 \xi_1$, and $\sigma$ is the projection to $t_1 , t_2$. The $G \times H$ action on $\mathbb{C}^{2}_{t_1 , t_2}$ lifts to $W$ in the following way: $G = \mathbb{Z}_\alpha$ acts on $W \subset \mathbb{C}^{2}_{t_1 , t_2} \times \mathbb{P}^{1}_{\xi_1 , \xi_2}$ with weight $(1,1,0,0)$ on the variables $z_1 , z_2 , \xi_1 , \xi_2$ and $H = \mathbb{Z}_\beta $ acts with weight $(0,1,0,1)$.

The blowup $W$ has the open cover $W = W_1 \cup W_2$ where $W_i = W \cap \{ \xi_i \neq 0\}$. Both are affine charts:
\begin{align*}
W_1  & = \{ t_1 \xi_2 = t_2 \} \subset \mathbb{C}^{3}_{t_1 , t_2 , \xi_2} \,\,\, , & W_2 & = \{ t_1 = \xi_1 t_2\} \subset \mathbb{C}^{3}_{t_1 , t_2 , \xi_1}\\
& \cong \mathbb{C}^{2}_{t_1 , \xi_2}  \,\,\, , &   & \cong \mathbb{C}^{2}_{t_2 , \xi_1 }.
\end{align*}
The $G \times H$ action on $W$ preserves each chart $W_i$. In the affine coordinates $(t_1 , \xi_2)$ for $W_1$, and $(t_2 , \xi_1 )$ for $W_2$, the $G$-action has weights $(1,0)$ and $(1,0)$, respectively, and the $H$-action has weights $(0,1)$ and $(1,-1)$, respectively. The gluing of the two charts $W_1 , W_2$ to recover $W$ is done along the domain of definition of the rational map $W_1 \dashrightarrow W_2$ specified by $t_2 = t_1 \xi_2$ and $\xi_1 = \xi_{2}^{-1}$. 

The variety $X$ is then defined as the quotient $X = W / G \times H $, and the map $f : X = W / G \times H \rightarrow \mathbb{C}^{2}_{t_1 , t_2}/G\times H = N_{\alpha ; 1 , \beta}$ is induced by the blow-down map $\sigma$. We obtain the open cover $\{ U_1 , U_2\}$ of $X$ as $U_i = W_i / G \times H$. Let us describe each of these opens. From the weights of the $G \times H$-action on $W_1$, we have 
\[
U_1 = (\mathbb{C}_{t_1}/\mathbb{Z}_\alpha ) \times ( \mathbb{C}_{\xi_2} / \mathbb{Z}_\beta ) \cong \mathbb{C}^{2}_{u_1 , v_1}
\]
where $u_1 = t_{1}^{\alpha}$ and $v_1 = \xi_{2}^{\beta}$. On the other hand $U_2$ is given by 
\[
U_2 = (\mathbb{C}_{t_2} / \mathbb{Z}_\alpha \times \mathbb{C}_{\xi_1} ) / H \cong     \mathbb{C}^{2}_{u_2 , v_2} / H = N_{\beta ; \alpha , -1}
\] 
where $u_2 = t_{2}^{\alpha}$ and $v_2 = \xi_1$, and $H$ acts on $\mathbb{C}^{2}_{u_2 , v_2}$ with weight $(\alpha , -1)$ on $(u_2 , v_2 )$ to give $N_{\beta ; \alpha, -1}$. However, for the unique integer $0 \leq \gamma < \beta$ given by (\ref{equation:fraction}), we have $N_{\beta ; \alpha , -1} = N_{\beta ; - \gamma , -1} = N_{\beta ; \gamma , 1}$, thus $U_2 = N_{\beta ; \gamma , 1}$, as required.
\end{proof}

The following diagram summarises the construction of the variety $X$ and the proper birational map $X \rightarrow N_{\alpha ;1 , \beta}$ given in the proof of Lemma \ref{lemma:opencover},
\begin{equation} \label{diagramX}
\begin{tikzcd}
U_1  = \mathbb{C}^{2}_{u_1 , v_1} \arrow{rrd}{f|_{U_1}}   \arrow[dd,dashed] &   &   \\
 &  & \mathbb{C}^{2}_{z_1 , z_2} / \mathbb{Z}_\alpha  = N_{\alpha ; 1 , \beta} \\
 U_2 = N_{\beta ; \gamma , 1}  = \mathbb{C}^{2}_{u_2 , v_2}/ \mathbb{Z}_\beta \arrow{rru}{f|_{U_2}} &  &  
\end{tikzcd} 
\end{equation}
where $f|_{U_1}$ is given by $z_1 = (u_1)^{1/\alpha}$, $z_2 = v_1 (u_1)^{\beta / \alpha}$, $f|_{U_2}$ is given by $z_1 = v_{2} (u_2 )^{1 / \alpha}$, $z_2  = (u_2 )^{\beta / \alpha}$. The variety $X$ is obtained from $U_1$ and $U_2$ by gluing along the domain of definition of the rational map $U_1 \dashrightarrow U_2$ given by $u_2 = u_{1} (v_1 )^{\alpha / \beta}$, $v_2  = (v_{1})^{-1/ \beta}$.


\subsubsection{The $\mathbb{C}^\ast$-action on $\widetilde{ N_{\alpha ; 1 , \beta }}$}

Regarding $N_{\alpha ; 1 , \beta} = \mathbb{C}^{2}_{z_1 , z_2} / \mathbb{Z}_\alpha$ as the total space of the orbifold line bundle over $\mathbb{C}/\mathbb{Z}_\alpha$ with Seifert data $(0 ; \beta )$, the weight one $\mathbb{C}^\ast$ action along the fibers corresponds to the weight $(0,1)$ action in the coordinates $z_1 , z_2$. We now record some relevant information about the lift of this $\mathbb{C}^\ast$-action to the minimal resolution $\widetilde{N_{\alpha ; 1 , \beta}}$ which we will need later.

We start by verifying once more that this $\mathbb{C}^\ast$-action indeed lifts (thus giving another, more direct proof of Proposition \ref{proposition:extension}). More generally, we will consider the $\mathbb{C}^\ast$-action on $N_{\alpha ; 1 , \beta}$ defined by specifying \textit{fractional} weights $(\frac{r}{\alpha} , \frac{s}{\alpha} )$ (here $r,s \in \mathbb{Z}$; the weight $(0,1)$ action is the case $r= 0 , s = \alpha$). Such weights give a well-defined $\mathbb{C}^\ast$-action on $N_{\alpha ; 1 , \beta}$ \textit{provided} that 
\begin{align}
\alpha \text{ divides } s - r \beta .  \label{actioncondition}
\end{align}
From the formula for the map $f|_{U_1} : U_1 = \mathbb{C}^{2}_{u_1 , v_1} \rightarrow N_{\alpha ; 1 , \beta}   = \mathbb{C}^{2}_{z_1 , z_2} / \mathbb{Z}_\alpha = N_{\beta ; \gamma , 1}  $ in diagram (\ref{diagramX}), there exists a unique way to lift the $\mathbb{C}^\ast$ action to $  U_1 = \mathbb{C}^{2}_{u_1 , v_2}$, by giving the $(u_1 , v_1)$ variables the integer weights
\begin{align}
( \,  r  \, , \, \frac{s - r \beta}{\alpha} \, ) \in \mathbb{Z}^{2} . \label{eq:weightsU1}
\end{align}
From the formula for the gluing map $U_1 = \mathbb{C}^{2}_{u_1 , v_1} \dashrightarrow U_2 = \mathbb{C}^{2}_{u_2 , v_2} / \mathbb{Z}_\beta$ in diagram (\ref{diagramX}), the $\mathbb{C}^\ast$-action on $U_1$ induces a $\mathbb{C}^\ast$-action on $U_2$ where the $u_2 , v_2$ variables have weight
\[
( \, \frac{s}{\beta} \, , \, \frac{ - (s - r \beta)/\alpha }{\beta} \, )  .
\]
However, for these fractional weights to define a $\mathbb{C}^\ast$-action on $U_2 = N_{\beta ; \gamma , 1}$ we must still ensure that the analogue of (\ref{actioncondition}) is true. The condition is now that 
\[
\beta \text{ should divide } s + \frac{ (s - r \beta ) \gamma}{\alpha} .
\]
But $\gamma = b \beta - \alpha$ from (\ref{equation:fraction}), which then gives
\begin{align*}
s + \frac{ (s - r \beta ) \gamma}{\alpha} =  \frac{(s - r \beta)}{\alpha} \cdot b \beta  + r \beta \in \beta \mathbb{Z} .
\end{align*}
Recall that the minimal resolution of $N_{\alpha ; 1 , \beta}$ is obtained by iterated applications of Lemma \ref{lemma:opencover}. The above argument shows, in particular, that the weight $(0, 1)$ $\mathbb{C}^\ast$-action on $N_{\alpha ; 1 , \beta}$ lifts to the varieties provided by the successive applications of Lemma \ref{lemma:opencover}, and thus to the minimal resolution $\widetilde{N_{\alpha ; 1 , \beta}}$.

$ $

We record the following

\begin{lemma}\label{lemma:weightR1}
The $\mathbb{C}^\ast$-action on the minimal resolution $\widetilde{N_{\alpha ; 1 ,\beta}}$ which lifts the $\mathbb{C}^\ast$-action of weight $(0 , 1)$ on $N_{\alpha ; 1 , \beta}= \mathbb{C}^{2}_{z_1 , z_2} / \mathbb{Z}_\alpha$ acts on the first rational curve $R_1$ of the exceptional divisor (i.e. the one that intersects the strict transform of the zero section $\{z_2 = 0\}$) with \textrm{weight one} at the fixed point given by the unique intersection point of $R_1$ with the strict transform of $\{z_2 = 0 \}$.
\end{lemma}
\begin{proof}
In the affine chart $U_1 = \mathbb{C}^{2}_{u_1, v_1}$ from diagram (\ref{diagramX}) the curve $R_1 $ is given by $u_1 = 0$, and the strict transform of the zero section $\{z_2 = 0\}$ is given by $v_1 = 0$. They intersect at the point $u_1 = v_1= 0$ which is fixed by the $\mathbb{C}^\ast$-action. By (\ref{eq:weightsU1}) specialised to $r = 0 $, $s = \alpha$ we immediately see that the weight of the $\mathbb{C}^\ast$-action on $R_1$ at this fixed point has weight one.
\end{proof}

\subsubsection{Proof of Proposition \ref{proposition:E(n)}}

We first consider the family $M = M(2,3,6n-5)$, $n \geq 2$. The associated complex surface $\widetilde{X}$ of \S \ref{subsection:compactification} is $E(n)$. Assertion (i) follows directly from Lemma \ref{lemma:loopextension} and Corollary \ref{corollary:extension}. We then verify the assertion (ii) (i.e. that $\tau_{\Sigma(2,3,6n-5)} = \tau_{S^3}$ in $\MCG (E(n) \setminus B )$). To see this, note that the intersection point $x$ of the central curve in $\widetilde{C}$ with the first curve $R_1$ of any one of the three chains emanating from the central curve is a fixed point of the $S^1$-action. By Lemma \ref{lemma:weightR1}, together with the fact that the central curve in $\widetilde{C}$ is fixed pointwise under the $S^1$-action, we deduce that the weights of the $S^1$-action on $E(n) \setminus M$ at the fixed point $x$ are given by $(0,1)$. Note that the weight $(0,1)$ $S^1$-action on $S^3$ yields the non-trivial element in $\pi_1 \mathrm{Diff}(S^3) = \mathbb{Z}/2$. Thus, removing a ball $B$ centered at $x$ establishes assertion (ii) for the family $M = M(2,3,6n-5)$, $n\geq 2$.

We now turn to $M = M(2,3,6n-1)$, $n \geq 2$. Again by Corollary \ref{corollary:extension} the $S^1$-action on $\Sigma(2,3,6n-1)$ extends over to $\widetilde{X} \setminus M(2,3,6n-1)$. However, recall that $\widetilde{X}$ is not a minimal complex surface, and that to obtain $X^\prime = E(n)$ from $\widetilde{X}$ we need to blow down $(-1)$-spheres as in Figure \ref{fig:blowdowns}. We first argue that the $S^1$-action on $\widetilde{X} \setminus M(2,3,6n-1)$ descends to $E(n) \setminus M(2,3,6n-1)$. For this, note that each of the $(-1)$-spheres $E = \mathbb{P}^1$ that get blown down at each stage are setwise preserved by the $S^1$-action. This action is holomorphic, and the induced action on the normal bundle $\mathcal{O}_{\mathbb{P}^1}(-1) = \{ ( \, (\xi_1 : \xi_2 )  \, , (z_1 , z_2 ) ) \in \mathbb{P}^1 \times \mathbb{C}^2 \, | \, (z_1 , z_2 ) \in (\xi_1 : \xi_2 ) \}$ of $E = \mathbb{P}^1$ has the form
\[
( \, (\xi_1 : \xi_2 ) \,  , \, (z_1 , z_2 ) \, )  \mapsto ( \, (\lambda^{a_1} \xi_1 : \lambda^{a_2} \xi_2 ) \, , \, (\lambda^{a_1} z_1 , \lambda^{a_2} z_{2} ) \, ) \quad , \quad \lambda \in S^1 
\]
for some weight $(a_1,a_2 ) \in \mathbb{Z}^{2}$. The blow down map is the projection $\mathcal{O}_{\mathbb{P}^1}(-1) \rightarrow \mathbb{C}^2$ to $(z_1 , z_2 )$, and thus the $S^1$-action descends to $\mathbb{C}^2$ as the weight $(a_1,a_2)$ action on $\mathbb{C}^2$. Thus, each $(-1)$-curve $E$ can be blown down $S^1$-equivariantly, and assertion (i) follows by Lemma \ref{lemma:loopextension}.

To verify assertion (ii) we will find a fixed point of the $S^1$-action on $E(n) \setminus M(2,3,6n-1)$ whose weight is $(a_1 , a_2 )$ with $a_1 + a_2$ an \textit{odd} integer. Granted that, the assertion (ii) follows as before. We claim that the singular point on the cuspidal curve from the last stage in Figure \ref{fig:blowdowns} has this property. In order to see this we analyse the $S^1$-action at each stage of the blowing down process of Figure \ref{fig:blowdowns}. In the second stage of Figure \ref{fig:blowdowns}, the triple point is a fixed point with weight $(1,1)$ by Lemma \ref{lemma:weightR1}. From this one deduces that in the third stage of Figure \ref{fig:blowdowns} the point of tangency has weight $(2,1)$, where the weight $1$ occurs along the  $(-4)$-curve by Lemma \ref{lemma:weightR1} (and hence the weight $2$ along the normal direction). In the fourth and last stage we deduce that the cuspidal point has weight $(3,2)$. This establishes (ii) for the family $M = M(2,3,6n-1)$, $n \geq 2$.

\section{Monopoles on Seifert-fibered $3$-manifolds} \label{section:MOY}

This section mostly consists on background material based on the work of Mrowka--Ozsváth--Yu \cite{MOY}, which describes the Seiberg--Witten theory of Seifert-fibered $3$-manifolds in terms of suitable algebro-geometric data. We will make extensive use of this theory in subsequent sections.

\subsection{Spin-c structures and Seifert fibrations}\label{subsection:spin}

We now review some relevant differential geometry of Seifert-fibered $3$-manifolds, following \cite{MOY} and \cite{furutasteer}. This will be provide the ingredients for setting up the Seiberg--Witten equations on such manifolds.

Fix a closed oriented Seifert-fibered $3$-manifold $Y \rightarrow C$ over a closed oriented orbifold surface $C$. Thus, $Y$ is identified as the unit circle bundle $Y = S(N)$ associated to a hermitian line bundle $N \rightarrow C$. We let $x_i \in C$, for $i = 1 , \ldots , n$, denote the orbifold points of $C$, with corresponding isotropy groups  $\mathbb{Z}/\alpha_i \mathbb{Z}$ with $\alpha_i \geq 2$. We then make the following auxiliary choices. First, we fix \textit{holomorphic} data:
\begin{enumerate}
\item[(i)] an orbifold holomorphic structure on $C$, making $C$ into an orbifold Riemann surface.
\item[(ii)] a holomorphic structure on the orbifold line bundle $N \rightarrow C$ up to isomorphism. That is, we fix a lift of the class of $N$ in the topological Picard group $\mathrm{Pic}^{t}(C)$ to the holomorphic Picard group $\mathrm{Pic}(C)$ (for the holomophic structure on $C$ fixed in (i) above). 
\end{enumerate}
Secondly, we fix \textit{geometric} data: 
\begin{enumerate}
\item[(iii)] an orbifold Kähler metric\footnote{An orbifold Kähler metric on the orbifold $C$ can be regarded as an ordinary Kähler metric on the Riemann surface $|C| \setminus \{ x_1 , \ldots , x_n \}$ with \textit{cone singularities} of cone angle $2\pi / \alpha_i $ at the points $x_i$.} $g_C$ of \textit{constant curvature} on the orbifold Riemann surface $C$. Such a metric is unique up to rescaling by a positive constant and up to holomorphic automorphisms of $C$. We denote by $\omega_C$ the area $2$-form of $g_C$. 
\item[(iv)] a hermitian connection $i \eta$ on $N \rightarrow C$ with \textit{constant curvature}, and such that the associated holomorphic structure on $N \rightarrow C$ given by the Cauchy--Riemann operator of $i \eta$ belongs in the isomorphism class fixed in (ii) above. Here we regard $\eta$ as a $1$-form on $Y = S(N)$, and the constant curvature condition on $i\eta$ is the hermitian Yang--Mills condition:
\begin{align}
d \eta = - \frac{2 \pi \deg (N) }{\mathrm{Area}(C, g_C)} \omega_C 
. \label{eq:constantcurvature}
\end{align}
\end{enumerate}

\begin{remark}\label{remark:choices} Some observations are in order:
\begin{enumerate}
\item In fact, not every orbifold Riemann surface $C$ admits a Kähler metric of constant curvature. Such a metric exists when $n \geq 3$ or $n = 2$ and $\alpha_1 \neq \alpha_2$ when the genus of $C$ is zero (see \cite{furutasteer}, Theorem 1.2). Thus, we make the tacit agreement that all orbifold Riemann surfaces that we consider from now on admit such a metric, and we will not worry about the few that don't.

\item The space of holomorphic stuctures on $N \rightarrow C$ up to complex gauge transformations is acted upon \textit{freely and transitively} by the \textit{Jacobian torus}
\[
H^{1}(C, i\mathbb{R})/H^{1}(C, 2 \pi i \mathbb{Z})  .
\]
In particular, when $C$ has genus \textit{zero}, i.e. $|C| \cong \mathbb{C}P^1$, then the choice in (ii) is unique.

\item The space of holomorphic structures on $N \rightarrow C$ up to complex gauge transformations is identified with the space of hermitian connections on $N$ with constant curvature and up to hermitian gauge transformations, by a standard argument. Namely, given a holomorphic structure $\overline{\partial}$ on $N$ there exists a complex gauge transformation $g$ of $N$, which is unique up to hermitian gauge transformations, such that the Chern connection associated to the holomorphic structure $g ( \overline{\partial} )$ has constant curvature.

Thus, the choice in (iv) is \textit{unique up to hermitian gauge transformations} (granted choices (i)-(iii) have been made).
\end{enumerate}
\end{remark}

\begin{remark}
For convenience, for the rest of the article we will often drop the adjective ``orbifold" (e.g. we might say ``line bundle" instead of ``orbifold line bundle") when it is clear from the context.
\end{remark}

On $C$ there is a \textit{canonical} spin-c structure $\mathfrak{s}_C$. It has spinor bundle given by \[S_C = \mathbb{C} \oplus K_{C}^{-1} \rightarrow C \quad , \] where $K_C$ is the canonical bundle of $C$, equipped with the hermitian metric induced from $g_C$, and the Clifford multiplication $\rho_C : TC \rightarrow \mathrm{Hom}(S_0, S_0 )$ is given by the symbol of the Dirac operator $\sqrt{2} ( \overline{\partial} + \overline{\partial}^\ast ) : \Gamma ( C , S_C ) \rightarrow \Gamma (C , S_C )$. The set of (isomorphism classes of) spin-c structures on $C$ is in one-to-one correspondence with orbifold complex line bundles $E_0 \rightarrow C$ with hermitian metric, up to isomorphism. The spin-c structure corresponding to $E_0$ is denoted $\mathfrak{s}_{E_0}$, and its spinor bundle is $S_{E_0} = S_C \otimes E_0 $ with Clifford multiplication $\rho_{E_0}(v) = \rho_{C}(v) \otimes \mathrm{id}_{E_0}$.

The spin-c structure $\mathfrak{s}_C$ on $C$ determines another spin-c structure $\mathfrak{s}_Y$ on $Y$ through the Seifert fibration $Y \rightarrow C$, which we call the \textit{canonical spin-c structure} on $Y$. In order to describe its spinor bundle we need to introduce a Riemannian metric $g_Y$ on $Y$, and we take 
\[
g_Y = \eta^2 + \pi^\ast g_C .\]
The connection $i \eta$ gives us a splitting of the tangent bundle of $Y = S(N)$, 
\begin{align}
TY = \mathbb{R} \cdot \zeta \otimes \pi^\ast TC \label{eq:splittingTY}
\end{align}
where $\pi^\ast TC = \mathrm{Ker}(\eta )$ is the horizontal distribution of the connection and $\mathbb{R} \cdot \zeta$ is the vertical distribution, which is trivialised by the vector field $\zeta$ that generates the $S^1$-action on $Y = S(N)$. The splitting (\ref{eq:splittingTY}) is orthogonal with respect to $g_Y$, but the Levi-Civita connection of $g_Y$ does not, in general, preserve this splitting. The natural connection on $TY$ to consider is the \textit{adiabatic connection}\footnote{The adiabatic connection $\nabla^{o} $ can be regarded as the limit when $t \rightarrow 0$ of the Levi--Civita connections of the metrics $t^2 \eta^2 + \pi^\ast g_C$, see \cite{nicolaescu}.} $\nabla^{o} = d \oplus \pi^\ast \nabla^{C} $ instead (here $\nabla^{C}$ is the Levi-Civita connection associated to $g_C$), which does preserve the splitting and is also compatible with $g_Y$.

Given the metric $g_Y$ and the compatible connection $\nabla^{o}$, the canonical spin-c structure $\mathfrak{s}_Y$ on $Y$ has spinor bundle
\begin{align}
S_Y = \pi^\ast S_C =   \mathbb{C} \oplus  \pi^\ast K_{C}^{-1}\label{eq:spinorbundle} \rightarrow Y \quad ,
\end{align}
and Clifford multiplication $\rho_Y : TY \rightarrow \mathrm{Hom}(S_Y , S_Y )$ is defined as $\rho_C$ along the horizontal distribution $\pi^\ast TC \subset TY$, and $\rho_Y (\eta )= +i$ (resp. $-i$) on $\mathbb{C}$ (resp. on $\pi^\ast K_{C}^{-1}$). Similarly as before, the spin-c structures on $Y$ are in one-to-one correspondence with complex line bundles $E \rightarrow Y$ with hermitian metric, up to isomorphism, and the spin-c structure $\mathfrak{s}_E$ associated with $E$ has spinor bundle $S_E = S_Y \otimes E$ with Clifford multiplication denoted $\rho_E$.

We also note the following:

\begin{lemma}\label{lemma:contact}
If $\mathrm{deg}N < 0$, then $\xi : = \mathrm{Ker}(\eta)$ is a (positively-oriented) contact structure on $Y$. The spin-c structure $\mathfrak{s}_\xi$ on $Y$ determined by the contact structure $\xi$ agrees with the canonical one $\mathfrak{s}_Y$ determined from the Seifert fibration $Y\rightarrow C$.
\end{lemma}

\begin{proof}
From (\ref{eq:constantcurvature}) we have that the Hodge star operator acting on the $2$-form $d \eta$ is
\[
\ast d \eta = - \frac{2 \pi \mathrm{deg}(N)}{\mathrm{Area}(C,g_C)}  \eta
\]
and hence $\eta \wedge d \eta $ is either a positive or negative or identically vanishing $3$-form on $Y$, according as to whether the orbifold degree of $N$ is negative, positive or zero, respectively.
\end{proof}

Given a hermitian connection $B$ on a hermitian line bundle $E \rightarrow Y$, the associated \textit{spin-c connection} $\nabla_B$ on the hermitian vector bundle $S_E  = S_Y \otimes E$ is uniquely defined by the property that Clifford multiplication $\rho_E : TY \rightarrow \mathrm{Hom}(S_E ,S_E)$ is parallel \textit{with respect to the adiabatic connection} $\nabla^{o}$ on $TY$, and $\nabla_B$ on $S_E$, and that the induced connection on $\mathrm{det}S_E = E^2 \otimes \pi^\ast K_{C}^{-1}$ is
$B^2 \otimes \pi^\ast B_{K_C}^{-1}$ (where $B_{K_C}$ is the Levi--Civita connection on $K_C$). In terms of the splitting (\ref{eq:spinorbundle}) this is just the connection $B \oplus (B \otimes \pi^\ast B_{K_C}^{-1})$. The \textit{Dirac operator} $D_B: \Gamma(Y, S_E) \rightarrow \Gamma(Y, S_E)$ twisted by the connection $B$ is the differential operator given by $
D_B \psi = \rho( \nabla_B \psi )$, which in terms of the natural splitting $S_E = E \oplus ( E \otimes \pi^\ast K_{C}^{-1} )$ has the block form
\begin{align}
D_B = \begin{pmatrix}
i \nabla^{B}_{\zeta} & \sqrt{2} \cdot \overline{\partial}_{B}^\ast \\
\sqrt{2} \cdot\overline{\partial}_B & -i \nabla^{B}_{\zeta} 
\end{pmatrix}.\label{eq:diracblock}
\end{align}
Here $\overline{\partial}_{B}^\ast$ is the adjoint of the differential operator $\overline{\partial}_{B} : \Gamma (Y, E ) \rightarrow \Gamma (Y , E \otimes \pi^\ast K_{C}^{-1} )$, which only involves derivatives along $C$. Namely, choosing a local coordinate $z = x+iy$ defined on the open set $\tilde{U} \subset \mathbb{C}$ associated to an orbifold chart $(\tilde{U}, \Gamma , \varphi )$ of $C$ we have
\[
\overline{\partial}_B \alpha = \big( \nabla^{B}_{\partial/\partial x} \alpha + i \nabla^{B}_{\partial /\partial y} \alpha \big) d \overline{z} ,
\]
where $\partial / \partial x$ and $\partial / \partial y$ are regarded as vectors tangent to $Y$ using the splitting (\ref{eq:splittingTY}). 

The Dirac operator $D_B$ thus decomposes into \textit{vertical} and \textit{horizontal} components $D_B = D_{B}^v + D_{B}^{h}$
\begin{align*}
D_{B}^v = \begin{pmatrix}
i \nabla^{B}_{\zeta} & \0 \\
0 & -i \nabla^{B}_{\zeta} 
\end{pmatrix} \quad , \quad D_{B}^h = \sqrt{2} \cdot \begin{pmatrix}
0 & \overline{\partial}_{B}^\ast \\
\overline{\partial}_B & 0
\end{pmatrix}.
\end{align*}
For future reference, we record here the following result regarding the ``super-commutator" of the two components, We have $\{ D_{B}^v , D_{B}^h \} := D_{B}^v D_{B}^h + D_{B}^h D_{B}^v$ ,
\begin{lemma}[Lemma 5.3.1 \cite{MOY}] \label{lemma:commutator}
$\{ D_{B}^v , D_{B}^h \}  = - \rho \big( \Pi_{\xi} \ast (F_B - \frac{1}{2} \pi^\ast F_{K_C}  ) \big)$, where $\Pi_{\xi}$ is the orthogonal projection onto the horizontal distribution $\xi = \mathrm{Ker}(\eta) \subset TY$.
\end{lemma}

With this in place, we can now describe the Seiberg--Witten equation on $Y$.

\subsection{Seiberg--Witten monopoles on $Y$}

The \textit{Seiberg-Witten equation} on the Seifert-fibered $3$-manifold $Y = S(N)$ with the spin-c structure $\mathfrak{s}_E$ is given by the following partial differential equation on a pair $(B, \psi ) $, where $B$ is a connection on the line bundle $E\rightarrow Y$ and $\psi$ a section of $S_E = E \oplus E \otimes \pi^\ast K_{C}^{-1}$:
\begin{align}
F_B - \frac{1}{2} \pi^\ast F_{K_C} & = \rho^{-1} (\psi \psi^\ast )_0 \label{eq:3dSW1}\\
D_B \psi & = 0 \label{eq:3dSW2}.
\end{align}

Above, $\pi^\ast F_{K_C}$ stands for the curvature of the Levi--Civita connection on $K_C$ pulled back along $\pi: Y \rightarrow C$. 

A solution to (\ref{eq:3dSW1}-\ref{eq:3dSW2}) will be referred to as a \textit{monopole} on $Y$ in the line bundle $E$. The equations are invariant under the infinite-dimensional \textit{gauge group} $\mathcal{G}(Y)$ of automorphisms of the bundle of Clifford modules $(S_E , \rho_E )$, which is just the group of smooth maps $g : Y \rightarrow U(1)$. If $\mathcal{C}(Y, \mathfrak{s}_E )$ denotes the configuration space of pairs $(B , \psi )$ (which is a Fréchet space when equipped with the $C^\infty$ topology), then the monopoles on $Y$ in the line bundle $E$ are the critical points of the \textit{Chern--Simons--Dirac functional} $\mathcal{L} : \mathcal{C}(Y, \mathfrak{s}_E ) \rightarrow \mathbb{R}$ given by
\begin{equation}\label{eq: CSD}
\mathcal{L}(B, \psi ) = - \int_Y (B- B^\prime ) \wedge (F_B + F_{B^\prime} ) + \int_{Y} \langle D_B \psi , \psi \rangle .
    \end{equation}

Here $B^\prime$ is a fixed connection on $E$.

Let $\mathcal{C}^\ast (Y, \mathfrak{s}_E ) \subset \mathcal{C}(Y, \mathfrak{s}_E )$ denote the space of \textit{irreducible} configurations, i.e. those $(B, \psi )$ with $\psi$ not vanishing identically. Let $\mathcal{B}(Y, \mathfrak{s}_E )$ (resp. $\mathcal{B}^\ast (Y, \mathfrak{s}_E )$) denote the quotient of the configuration space $\mathcal{C}(Y, \mathfrak{s}_E )$ (resp. $\mathcal{C}^\ast (Y, \mathfrak{s}_E )$) by the action of the group of gauge transformations $\mathcal{G}(Y)$. It is a standard fact that one obtains a Hilbert manifold by considering instead the space of $L^{2}_k$ irreducible configurations $(B, \psi )$ modulo $L^{2}_{k+1}$ gauge transformations for suitable $k > 0$ (see \cite{KM}), but we will often ignore this and simply refer to $\mathcal{B}^\ast (Y, \mathfrak{s}_E )$ as a manifold, for the sake of exposition.

The Chern--Simons--Dirac functional yields a function $\mathcal{L} : \mathcal{B}(Y, \mathfrak{s}_E ) \rightarrow \mathbb{R}/ 4 \pi^2 \mathbb{Z}$. In the case when $c_1 (\mathfrak{s}_E )$ is a torsion class, then $\mathcal{L}$ is valued in $\mathbb{R}$ rather than the circle $\mathbb{R}/4 \pi^2 \mathbb{Z}$. The former condition holds precisely when $E$ is isomorphic to the pullback of an orbifold line bundle from $C$ (see Theorem 2.0.19 and Remark 2.0.20 in \cite{MOY}), and we shall see that these are the only spin-c structures we will have to care about.

We denote by $\mathfrak{C}(Y, E) \subset \mathcal{B}(Y, \mathfrak{s}_E )$ the \textit{moduli space of monopoles} on $Y$ in the line bundle $E$, for the moment regarded as a topological space. Next, we review the description of the moduli spaces $\mathfrak{C}(Y,E)$ on Seifert-fibered $3$-manifolds from \cite{MOY}.

\subsubsection{Irreducible monopoles}


Let $\mathfrak{C}^\ast (Y , \mathfrak{s}_E ) \subset \mathfrak{C}(Y, \mathfrak{s}_E )$ be the moduli space of \textit{irreducible} monopoles in $E$.

\begin{proposition}[Theorem 5, \cite{MOY}]\label{proposition:MOYirreducible} The moduli space of irreducible monopoles $\mathfrak{C}^\ast (Y,E) \subset \mathcal{B}^\ast (Y, \mathfrak{s}_E)$ is a Morse--Bott critical manifold of the Chern--Simons--Dirac functional. It decomposes as a disjoint union (indexed by isomorphism classes of orbifold line bundles $E_0$ on $C$ with $E \cong \pi^\ast E_0$ and $0\leq \mathrm{deg}E_0 < \frac{1}{2}\mathrm{deg}K_C$)
\[
\mathfrak{C}^\ast (Y, E ) = \bigsqcup_{ E \cong \pi^\ast E_0 \, , \, 0 \leq \mathrm{deg}E_0 < \frac{1}{2}\mathrm{deg}K_C }\Big( \mathfrak{C}^{+} (E_0 ) \sqcup \mathfrak{C}^{-}(E_0) \Big) ,\]
where $\mathfrak{C}^+ (E_0)$ is diffeomorphic to the moduli space $\mathcal{D}(C,E_0 )$ of orbifold effective divisors on $C$ in the line bundle $E_0$ (cf. Definition \ref{definition:divisor} and Lemma \ref{lemma:effRiemannsurface}), and the charge-conjugation symmetry of the configuration space identifies $\mathfrak{C}^{-}(E_0 )$ diffeomorphically with $\mathfrak{C}^{+}(E_{0}^{-1} \otimes K_C )$.
\end{proposition}

We recall the meaning of Morse--Bott critical manifold, following \cite{FLINthesis}:

\begin{definition} \label{def:MorseBott}
Let $[S] \subset \mathcal{B}^\ast (Y, \mathfrak{s}_E )$ be finite dimensional submanifold of the configuration space consisting entirely of irreducible monopoles, and let $S\subset \mathcal{C}^\ast (Y, \mathfrak{s}_E )$ be its inverse image. We say $[S]$ is a \textit{Morse--Bott} critical manifold if for every $(B, \psi ) \in S$ the Hessian of $\mathcal{L}$ at $(B, \psi )$, which gives a self-adjoint differential operator $\mathcal{D}_{(B, \psi )}$ acting on $T_{(B, \psi )} \mathcal{C}^\ast (Y, \mathfrak{s}_E ) = \Omega^1 (Y, i \mathbb{R} ) \oplus \Gamma (Y, S_E )$, satisfies that the inclusion 
\[
T_{(B, \psi )} [S] \rightarrow \mathrm{Ker} \mathcal{D}_{(B, \psi )}/ \mathrm{Im}\mathbf{d}_{(B, \psi )}
\]
is an isomorphism, where $\mathbf{d}_{(B, \psi )} : \Omega^0 (Y, i \mathbb{R}) \rightarrow T_{(B, \psi )} \mathcal{C}^\ast (Y, \mathfrak{s}_E )$ is the linearisation at $1 \in \mathcal{G}(Y)$ of the gauge action $g \in \mathcal{G}(Y) \mapsto (B - g^{-1} dg , g \psi )$.
\end{definition}

For completeness, we now summarise how the above result is established in \cite{MOY}, since the theme of the proof will reappear at various points throughout the article:

\begin{proof}[Proof summary] \textit{Step 1: From monopoles to vortices.}

Denote by $\zeta$ the vector field on $Y$ that generates the Seifert $S^1$ action. Let $(B, \psi )$ be an irreducible monopole on $Y$ in the line bundle $E$. Write $\psi = (\alpha , \beta) \in \Gamma(Y, S_E)$ according to the splitting $S_E = E \oplus E \otimes \pi^\ast K_{C}^{-1}$.

 The first step towards the proof of Proposition \ref{proposition:MOYirreducible} is an integration by parts argument (see \S 5.5 of \cite{MOY}), which shows that a monopole $(B, \psi )$ satisfies the identities
\begin{align}
D_{B}^{v}\psi = 0 \, \,\, , \,\,\, D_{B}^{h} \psi = 0 \,\,\, \text{and} \,\,\, |\alpha| \cdot |\beta| = 0. \label{eq:identities}
\end{align}
Because $\psi$ is not identically vanishing, the first identity in (\ref{eq:identities}) shows that the connection $B$ on $E$ has \textit{trivial holonomy along the fibers} of $Y \rightarrow C$, and together with (\ref{eq:3dSW1}), that the curvature form $F_B$ \textit{pulls back} from $C$. This means that the line bundle with connection $(E, B)$ on $Y$ is the pullback of a pair of orbifold line bundle with connection $(E_0 , B_0 )$ on $C$. Similarly, $\alpha$ and $\beta$ descend to sections $\alpha_0$ and $\beta_0$ of $E_0$ and $E_0 \otimes K_{C}^{-1}$, respectively. By the second identity in (\ref{eq:identities}), $\alpha_0$ is a holomorphic section of $(E_0 , \overline{\partial}_{B_0})$ and $\overline{\beta_0}$ a holomorphic section of $ E_{0}^{-1} \otimes K_C$. The third identity in (\ref{eq:identities}) and unique continuation then say that either $\alpha_0 = 0$ or $\beta_0 = 0$ identically on $C$.

At this point it follows that the triple $(B_0 , \alpha_0 , \beta_0 )$, consisting of a connection $B_0$ on the hermitian orbifold line bundle $E_0$, and sections $\alpha_0 \in \Gamma( C , E_0)$, $\beta \in \Gamma(C , E_0 \otimes K_{C}^{-1})$, solves the \textit{vortex equation}:
\begin{align}
& 2 F_{B_0} - F_{K_C}  =   i ( |\alpha_0 |^2 - |\beta_0 |^2 )\omega_C \label{eq:vortex1}\\
& \overline{\partial}_{B_0} \alpha_0 = 0   \text{ 
 and  }  \overline{\partial}_{B_0}^{\ast} \beta_0 = 0 \label{eq:vortex2} \\
& \alpha_0  = 0  \text{  or  }  \beta_0 = 0 . \label{eq:vortex3}
\end{align}

A solution $(B_0 , \alpha_0 , \beta_0 )$ to $(\ref{eq:vortex1}-\ref{eq:vortex3})$ with $\alpha_0 \neq 0$ or $\beta_0 \neq 0$ will be referred to as a \textit{vortex} on $C$ in the line bundle $E_0$. In the first case we call it a \textit{positive vortex}, and in the second a \textit{negative vortex}. Denote by $M_{v} (C , E_0 ) = M_{v}^{+}(C, E_0 ) \sqcup M_{v}^{-}(C , E_0 )$ the corresponding moduli space of vortices, modulo the unitary gauge group $\mathcal{G}(C) = \mathrm{Map}(C, U(1) )$.

Conversely, a vortex on $C$ in the line bundle $E_0$ pulls back to a monopole on $Y$ in $E = \pi^\ast E_0$. All combined, this establishes an identification (of topological spaces, for the moment) of the moduli space $\mathfrak{C}^\ast (Y , E )$ of irreducible monopoles on $Y$ in the line bundle $E$ with the disjoint union of the moduli spaces $M_{v} (C , E_0 ) $ of vortices over all line bundles $E_0$ with $E \cong \pi^\ast E_0$.

$ $

\textit{Step 2: From vortices to effective divisors}

We have natural identifications (for the moment, as topological spaces)
\begin{align}
M^{\pm}_{v} (C, E_0 ) = M^{\mp}_{v}(C, E_{0}^{-1} \otimes K_C ) \, , \quad (B_0 , \alpha_0 , \beta_0 ) \mapsto (B_{K_C} - B_0 , \overline{\beta_0} , \overline{\alpha_0} ). \label{eq:vortexduality}
\end{align}

Thus, we focus on the positive vortices only from now on.  A positive vortex $(B_0 , \alpha_0 )$ on $C$ in $E_0$ defines an effective orbifold divisor $(\overline{\partial}_{B_0} , \alpha_0 )$ on $C$ in $E_0$. Thus, the existence of a positive vortex in $E_0$ necessarily implies that $\mathrm{deg}E_0 \geq 0$, and from the equation (\ref{eq:vortex1}) that $\mathrm{deg}E_0 < \frac{1}{2} \mathrm{deg}K_C$. Conversely, given an effective divisor $(\overline{\partial} , s )$ in a line bundle $E_0$ with $0 \leq \mathrm{deg}E_0 < \frac{1}{2}\mathrm{deg}K_C$ we want a complex gauge transformation $g \in \mathcal{G}_{\mathbb{C}}(C) =  \mathrm{Map}(C ,\mathbb{C}\ast )$ such that if $B_0$ is the Chern connection of $g \cdot \overline{\partial}$ (with respect to the hermitian metric of $E_0$) and $\alpha_0 = g \cdot s$, then $(B_0 , \alpha_0 )$ is a positive vortex in $E_0$, and $g$ is unique up to unitary gauge transformation. Such $g$ is obtained by solving a Kazdan--Warner equation (see e.g. Lemma 7.2.4 in \cite{morgan}), and the existence and uniqueness follows from the properties of said equation. 

All combined, this establishes an identification of the positive vortex moduli space $M_{v}^{+}(C, E_0 )$ and the moduli space of effective divisors $\mathcal{D}(C , E_0 )$. We refer to \cite{MOY} for the proof of the Morse--Bott assertion.




\end{proof}

\subsubsection{Reducible monopoles}

A monopole $(B, \psi)$ on $Y$ in the line bundle $E$ is called \textit{reducible} if $\psi$ vanishes identically. Thus, it consists of a unitary connection $B$ on $E$ such that 
\begin{align}
F_B = \frac{1}{2} \pi^\ast F_{K_C}.\label{eq:reduciblemonopole}
\end{align}
Let us denote by $\mathfrak{C}^{\redu}(Y, E ) \subset \mathcal{B}(Y, \mathfrak{s}_E)$ the moduli space of reducible monopoles in $Y$. 


\begin{proposition}[\S 5.8 in \cite{MOY}] \label{proposition:MOYreducible} A line bundle $E$ on $Y$ supports a reducible monopole if and only if $E \cong \pi^\ast E_0$ for some orbifold line bundle $E_0$ on $C$. Fixing one such $E_0$ gives a homeomorphism of the moduli space of reducible monopoles $\mathfrak{C}^{\redu}(Y, E)$ with the moduli space $\mathfrak{T}(E_0)$ of holomorphic structures on $E_0$. 

Let $B$ be a reducible monopole on $E$. The reducible locus $\mathfrak{C}^{\redu}(Y,E) \subset \mathcal{B}(Y,\mathfrak{s}_E)$ is \textbf{not} a Morse--Bott critical manifold of the Chern--Simons--Dirac functional near $B \in \mathfrak{C}^{\redu}(Y, E)$ precisely when both of the following occur:
\begin{itemize}
\item there exists a line bundle $E_0$ on $C$ with $E \cong \pi^\ast E_0$ and $ \mathrm{deg}E_0 = \frac{1}{2} \mathrm{deg}K_C$, and
\item the holomorphic structure $\mathcal{E}_0$ on $E_0$ corresponding to the reducible monopole $B$ has $H^{0}(C , \mathcal{E}_0 ) \oplus H^{1}(C , \mathcal{E}_0 )\neq 0$.
\end{itemize}
\end{proposition}

The assertion that $\mathfrak{C}^{\redu}(Y, E )$ is a Morse--Bott critical manifold needs some comments, since the the configuration space $\mathcal{B}(Y, \mathfrak{s}_E )$ is not a manifold near the locus of reducible configurations $\mathcal{B}^{\redu}(Y, \mathfrak{s}_E )$. However, the locus $\mathcal{B}^{\redu}(Y, \mathfrak{s}_E )$ can be regarded as a manifold (again, for this one should consider $L^{2}_{k}$ configurations modulo $L^{2}_{k+1}$ gauge transformations), and then Definition \ref{def:MorseBott} adapts immediately to this case: 

\begin{definition}
Let $[S]$ be a finite-dimensional submanifold of $\mathcal{B}^{\redu}(Y, \mathfrak{s}_E)$, and let $S \subset \mathcal{C}(Y, \mathfrak{s}_E )$ denote its inverse image. We say $[S]$ is a \textit{Morse--Bott} critical manifold if for every $(B, 0 ) \in S$ the Hessian of $\mathcal{L}$ at $(B, 0 )$, which gives a self-adjoint differential operator $\mathcal{D}_{(B, 0 )}$ acting on $T_{(B, 0 )} \mathcal{C} (Y, \mathfrak{s}_E ) = \Omega^1 (Y, i \mathbb{R} ) \oplus \Gamma (Y, S_E )$, satisfies that the inclusion 
\[
T_{(B, 0 )} [S] \rightarrow \mathrm{Ker} \mathcal{D}_{(B, \psi )}/ \mathrm{Im}\mathbf{d}_{(B, 0 )}
\]
is an isomorphism, where $\mathbf{d}_{(B, 0 )}$ is the linearisation of the gauge action.
\end{definition}

Some additional remarks about Proposition \ref{proposition:MOYreducible} are in order:
\begin{remark} 
$ $
\begin{enumerate}
\item The moduli space $\mathfrak{T}(E_0)$ of holomorphic structures on an orbifold line bundle $E_0$ on $C$, modulo the complex gauge group $\mathcal{G}_{\mathbb{C}}(C) = \mathrm{Map}(C, \mathbb{C}^\ast)$, becomes identified with the \textit{Jacobian torus} $H^{1}(C , i \mathbb{R})/H^{1}(C , 2  \pi i \mathbb{Z} )$ upon choosing one particular holomorphic structure (see Remark \ref{remark:choices} (3)). 
\item If there exists a line bundle $E_0$ on $C$ pulling back to $E$, then it is not unique: the line bundles with this property are precisely $E_0 \otimes N^{k}$ for $k \in \mathbb{Z}$ (see Lemma \ref{lemma:pullbackline}).

\item Following \cite{MOY}, a topological line bundle $E_0$ on $C$ with $\mathrm{deg}E_0 = \frac{1}{2} \mathrm{deg}K_C$ will be referred to as an \textit{orbi-spin} bundle.
\end{enumerate}
\end{remark}



For completeness, we summarise the proof of Proposition \ref{proposition:MOYreducible}, following \cite{MOY}.

\begin{proof}[Proof summary]

If $E$ admits a reducible monopole, then by the equation (\ref{eq:reduciblemonopole}) we see that $c_1 (E)$ is a torsion class, and hence by [Theorem 2.0.19, \cite{MOY}] we have $E \cong \pi^\ast E_0$ for some line bundle $E_0$. Such line bundle $E_0$ is not unique -- the line bundles on $C$ pulling back to $C$ are given by the $E_0 \otimes N^{k}$ where $k \in \mathbb{Z}$ (see Theorem 2.0.19 in \cite{MOY}).

Conversely, if $E = \pi^\ast E_0$ for some $E_0$, then we may choose a connection $B_0$ on $E_0$ with constant curvature, and we obtain a reducible monopole $B$ as 
\begin{align}
B = \pi^\ast B_0 + \frac{\frac{1}{2} \mathrm{deg}K_C - \mathrm{deg}E_0 }{\mathrm{deg}N} \cdot i \eta .\label{eq:reducible}
\end{align}
All reducible monopoles are of this form (up to gauge transformation), for some choice of constant curvature connection $B_0$. The latter space of connections (modulo hermitian gauge transformations) is identified with the Jacobian torus.

Next, we discuss the Morse--Bott assertion. It is easy to see that the reducible locus $\mathfrak{C}^{\redu}(Y, E)$ is a Morse--Bott critical manifold precisely when the Dirac operator $D_B$ has no kernel, for all reducible monopoles $B$ in $E$. By Lemma \ref{lemma:commutator}, the super-commutator of $\{D_{B}^v , D_{B}^h\}$ vanishes at a reducible monopole $B$, hence 
\[
(D_{B})^2 = (D_{B}^v)^2 + (D_{B}^h )^2 .
\]
Suppose that $D_B$ has non-trivial kernel. Then $B$ has trivial fiberwise holonomy and the curvature form $F_B$ pulls back from $C$ by (\ref{eq:reducible}), from which we deduce that $(E , B )$ is the pullback of some pair $(E_0 , B_0 )$ with $F_{B_0} = \frac{1}{2} F_{K_C}$. Consider the holomorphic line bundle $\mathcal{E}_0 = (E_0 , \overline{\partial}_{B_0} )$. It also follows that the kernel of $D_B$ is identified with $H^{0}(C , \mathcal{E}_0 ) \oplus H^{1}(C , \mathcal{E}_0 )$ and $\mathrm{deg} \mathcal{E}_0 = \frac{1}{2} \mathrm{deg}K_C$, from which one direction is proved, and the converse implication follows similarly.\end{proof}

Later we will work with the \textit{blowup model} for monopole Floer homology (see \cite{KM}). For this we will need to have a holomorphic description of the blown-up reducible monopoles. We will take up this task in \S \ref{section:dirac}.

\begin{definition}A \textit{blown-up reducible monopole} on $Y$ in the line bundle $E$ is a pair $(B, \psi )$ consisting of a connection $B$, and a spinor $\psi$ with \textit{unit $L^{2}$ norm} ($||\psi ||_{L^{2}(Y)} = 1$), satisfying the equations
\begin{align}
F_B &= \frac{1}{2} F_{K_C} \label{eq:blownupred1}\\
D_B \psi &= \lambda \psi \text{  for some } \lambda \in \mathbb{R} \label{eq:blownupred2}.
\end{align}
\end{definition}

In the case when $C$ has genus zero (i.e. $|C| \cong \mathbb{P}^1$), then the space of blown-up reducible monopoles in $E$, modulo gauge transformations, is the union of the projectivised eigenspaces $\mathbb{P}(V_\lambda )$ of the Dirac operator $D_{B}$ twisted by the unique reducible monopole $B$ in $E$. Furthermore, $\mathbb{P}(V_\lambda )$ is a Morse-Bott critical manifold of the blown-up configuration space in the sense of (\cite{FLINthesis}, Definition 1.2) precisely when $\{(B, 0)\} \in \mathcal{B}(Y, \mathfrak{s}_E)$ is a non-degenerate critical point.

\subsection{Flows of the Chern--Simons--Dirac functional}

We will now review the holomorphic description of the flowlines of the Chern--Simons--Dirac functional $\mathcal{L} : \mathcal{B}(Y , \mathfrak{s}_E ) \rightarrow \mathbb{R}/4\pi^2 \mathbb{Z}$, following \cite{MOY}. From this point on, whenever we discuss gradient flowlines of the Chern--Simons--Dirac functional, we will mean \textit{negative} flowlines (i.e. flowlines of $- \nabla \mathcal{L}$). 

Before proceeding, we discuss some general features about the Chern--Simons--Dirac functional on the Seifert-fibered manifolds $Y = S(N)$ (with $\mathrm{deg}N < 0$, as usual throughout the article). By Proposition \ref{proposition:MOYirreducible} and Proposition \ref{proposition:MOYreducible}, the spin-c structures $\mathfrak{s}_E$ on $Y$ for which $\mathcal{L}$ has non-empty critical locus must have \textit{torsion} first Chern class, because $E$ pulls back from $C$ (see also Lemma \ref{lemma:pullbackline}). It follows that $\mathcal{L}$ takes values on $\mathbb{R}$, rather than the circle $\mathbb{R}/4 \pi^2 \mathbb{Z}$ (see \cite{KM}, Lemma 4.1.3). Furthermore, we have the following special feature:

\begin{lemma}[Lemma 4.9.5 in \cite{MOY}] Suppose $\mathrm{deg}N < 0$ and fix a torsion spin-c structure $\mathfrak{s}_E$ on $Y = S(N)$. After shifting $\mathcal{L} : \mathcal{B}(Y, \mathfrak{s}_E) \rightarrow \mathbb{R}$ by a constant so that $\mathcal{L} = 0$ on the reducible critical locus, then $\mathcal{L} > 0$ on the irreducible critical locus. In particular, there exist no flowlines from the reducible critical locus to the irreducible critical locus. 
\end{lemma}

\subsubsection{The ruled surface}\label{subsubsection:ruled}

The holomorphic description of the space of flows now involves the \textit{orbifold ruled surface} given as the total space of the projective orbifold bundle 
\[
R = \mathbb{P} (\mathbb{C} \oplus N ) \xrightarrow{\Pi} C
\] 
i.e. the projectivisation of the orbifold holomorphic vector bundle $\mathbb{C} \oplus N \rightarrow C$. The \textit{zero section} of the ruled surface $R$ is the suborbifold $C_- = \mathbb{P}( \mathbb{C} \oplus 0 ) \subset R$, and the \textit{section at infinity} is the suborbifold $C_+ = \mathbb{P}(0 \oplus N ) \subset R $, both of which are orbifold complex curves biholomorphic to $C$. The orbifold locus of $R$ is given by the union of the orbifold loci of the curves $C_\pm$. 

The ruled surface $R$, being an orbifold projective bundle, carries an orbifold \textit{tautological holomorphic line bundle} $\mathcal{O}_R (-1) \rightarrow R$. This has the property that
\begin{align}
\mathcal{O}_{R}(-1)|_{C_-} = \mathcal{O}_R \text{  and  } \mathcal{O}_{R}(-1)|_{C_+} = N  \label{formula:restrictions} .\end{align}
Dualising the tautological bundle $\mathcal{O}_{R}(-1)$ we obtain the line bundle denoted $\mathcal{O}_{R}(1)$, and its $n$th tensor power is denoted $\mathcal{O}_{R}(n)$. When $n \geq 1$, the line bundle $\mathcal{O}_{R}(n)$ carries a canonical holomorphic section, denoted $z^{n}$, which vanishes precisely over the section at infinity $C_+ \subset R$ to order $n$.

The Picard groups of topological or holomorphic orbifold line bundles on $R$ are given in terms of those of the curve $C$, 
\begin{align}
\mathrm{Pic}^{t}(C) \times \mathbb{Z}  \cong \mathrm{Pic}^t (R) \, , \quad (E_0 , k ) \mapsto ( \Pi^{\ast} E_0 )\otimes \mathcal{O}_{R}(k)  \label{pict} \\
\mathrm{Pic} (C) \times \mathbb{Z}  \cong \mathrm{Pic} (R) \, ,  \quad (\mathcal{E}_0 , k ) \mapsto ( \Pi^{\ast} \mathcal{E}_0 )\otimes \mathcal{O}_{R}(k)   . \label{pic}
\end{align}
In particular, note that since $N$ is non-torsion, then by (\ref{pict}-\ref{pic}) and (\ref{formula:restrictions}) any orbifold line bundle (topological or holomorphic) over $R$ is determined uniquely from its restrictions to the curves $C_{\pm} \subset R$.

We will be interested in effective divisors on the ruled surface. 

\begin{definition}An \textit{effective orbifold divisor} on $R$ in an orbifold line bundle $L$ consists of a pair $(\overline{\partial} , s )$ where $\overline{\partial}$ is a holomorphic structure on $L \rightarrow R$, and $s$ is a \textit{non-trivial} holomorphic section of $L \rightarrow X$, i.e. $\overline{\partial} s = 0$. We denote by $\mathcal{D}(R, L )$ the quotient of the space of such holomorphic pairs $(\overline{\partial} , s )$ by the group of complex gauge transformations $\mathcal{G}_{\mathbb{C}}(R)  = \mathrm{Map}(R , \mathbb{C}^\ast )$ of automorphisms of $L$, equipped with the Whitney $C^\infty$ topology.

The following moduli space will also come into play. Fix a basepoint $x \in R$. A \textit{based} or \textit{framed effective orbifold divisor} on $X$ in the line bundle $L$ consists of a triple $( \overline{\partial} , s , f)$, where $(\overline{\partial},s)$ is an effective orbifold divisor on $X$ in the line bundle $L$, and $f$ is an isomorphism of vector spaces $f : \mathbb{C} \xrightarrow{\cong} L_x $ (a framing of $L$ at $x$). We denote by $\mathcal{D}^{o} (R, L )$ the quotient of the space of such holomorphic triples by the action of $\mathcal{G}_{\mathbb{C}}(R)$. Alternatively, $\mathcal{D}^{o}(R, L )$ can be constructed as the quotient of the space of holomorphic pairs ($\overline{\partial} , s )$ by the \textit{based} complex gauge group $\mathcal{G}^{o}_{\mathbb{C}}(R) := \{ g \in \mathcal{G}_{\mathbb{C}}(R) \, | \, g(x) = 1\}$. There is a residual $\mathbb{C}^\ast$-action on the space $\mathcal{D}^{o}(R, L )$ with quotient the space $\mathcal{D}(R, L )$.
\end{definition}

\subsubsection{Flows between irreducibles}

We fix a spin-c structure $\mathfrak{s}_E$ on $Y = S(N)$ where $\mathrm{deg}N< 0$, and study the flowlines from $\mathfrak{C}^{\ast} (Y , E )$ to $\mathfrak{C}^{\ast}(Y , E)$. Let $E_0 , E_1 \rightarrow C$ be two orbifold complex line bundles with $E \cong \pi^\ast E_0 \cong \pi^\ast E_1$, and recall the decomposition of $\mathfrak{C}^\ast (Y, E )$ from Proposition \ref{proposition:MOYirreducible}. Thus, there exists a unique orbifold complex line bundle $\widehat{E} = \widehat{E} ( E_0 , E_1 )$ over the ruled surface $R$ such that
\[
\widehat{E}|_{C_-} \cong E_0 \quad \text{and} \quad \widehat{E}|_{C_+} \cong E_1 .
\]
Indeed, note that by Lemma \ref{lemma:pullbackline} we have $E_1 = E_{0} \otimes N^{-k}$ for a unique $k \in \mathbb{Z}$, and hence by (\ref{pict}) and (\ref{formula:restrictions}) we have $\widehat{E} = (\Pi^\ast E_0 ) \otimes \mathcal{O}_{R}(k )$ where 
\[
k = \frac{ \mathrm{deg}E_0 - \mathrm{deg}E_1}{\mathrm{deg}N} .
\]

\begin{proposition}[Theorems 7.0.18 and 9.2.5 in \cite{MOY}]\label{proposition:MOYflowsirreducibles} With $E, E_0 , E_1$ as above, the moduli space $M( \mathfrak{C}^{+}(E_0 ) , \mathfrak{C}^{+}(E_1))$ of parametrised flows of $\mathcal{L} : \mathcal{B}(Y , \mathfrak{s}_E ) \rightarrow \mathbb{R}$  from $\mathfrak{C}^{+}(E_0 )$ to $\mathfrak{C}^{+}(E_1)$ is homeomorphic to the open subset of the moduli space $\mathcal{D}(R , \widehat{E} )$ of effective divisors on the ruled surface $R$ in the line bundle $\widehat{E}$ which do not contain both sections $C_\pm \subset R$. If $(\widehat{\mathcal{E}} , \alpha )$ is such an effective divisor, with zero set $D = \alpha^{-1}(0)$, then the moduli space of parametrised flows $M( \mathfrak{C}^{+}(E_0 ) , \mathfrak{C}^{+}(E_1))$ is modelled near $(\widehat{\mathcal{E}} , \alpha )$ on the zero locus of a complex-analytic map $\kappa_{sw} : H^{0}(D , \widehat{\mathcal{E}}|_D ) \rightarrow H^{1}(D , \widehat{\mathcal{E}}|_D )$ such that $\kappa_{sw} (0 ) = 0$ and $(d \kappa_{sw})_{0} = 0$.

The analogous result for $M( \mathfrak{C}^{-}(E_0 ) , \mathfrak{C}^{-}(E_1))$ can be deduced from the above using the charge-conjugation symmetry. On the other hand, there are no flows connecting the $\mathfrak{C}^\pm$ critical manifolds to the $\mathfrak{C}^{\mp}$ critical manifolds. 
\end{proposition}

\begin{remark}\label{remark:orbispace}
Some remarks are in order:
\begin{enumerate}

\item Above, we may tacitly \textit{define} the groups $H^{\ast}(D , \widehat{\mathcal{E}}|_D )$ as $H^{\ast}( R , \widehat{\mathcal{E}}/ \alpha \mathcal{O}_R)$, where $\widehat{\mathcal{E}}/ \alpha \mathcal{O}_R$ is defined as the orbifold sheaf obtained as the cokernel of a morphism of orbifold sheaves:
\[
0 \rightarrow \mathcal{O}_R \xrightarrow{\alpha} \widehat{\mathcal{E}} \rightarrow \widehat{\mathcal{E}}/ \alpha \mathcal{O}_R \rightarrow 0 \quad .
\]
This definition has the advantage of avoiding discussing what sort of singular space is the zero set $D = \alpha^{-1}(0)$, and it shall be sufficient for our purposes. 

Alternatively, one can regard $D$ as a \textit{complex-analytic orbispace}, possibly non-reduced. This is the orbifold version of the notion of a complex-analytic space, and its corresponding notion in algebraic geometry is that of a  Deligne--Mumford stack. The inclusion $\iota : D \rightarrow R$ is a morphism of complex-analytic orbispaces, and now $\widehat{\mathcal{E}}|_D$ can be defined as an orbifold sheaf on $D$ in the familiar way: as the sheaf of $\mathcal{O}_D$-modules $\iota^{\ast}  ( \widehat{\mathcal{E}}/ \alpha \mathcal{O}_R ) \cong ( \iota^{-1} (\widehat{\mathcal{E}}/ \alpha \mathcal{O}_R ) ) \otimes_{\iota^{-1} \mathcal{O}_R } \mathcal{O}_D $, making use of the usual pullback $\iota^\ast$ and inverse image $\iota^{-1}$ operations on sheaves over complex-analytic orbispaces.

\item 
One can show that the moduli space $\mathcal{D}(R, \widehat{E} )$ is a complex-analytic space which is locally modelled near a given divisor on the zeros of a complex-analytic map $\kappa_{div} : H^{0}(D , \widehat{\mathcal{E}}|_D ) \rightarrow H^{1}(D , \widehat{\mathcal{E}}|_D )$ such that $\kappa_{div} (0 ) = 0$ and $(d \kappa_{div})_{0} = 0$ (see Appendix \ref{section:divisors}). That the identification $M( \mathfrak{C}^{+}(E_0 ) , \mathfrak{C}^{+}(E_1))\cong \mathcal{D}(R , \widehat{E} )$ from Proposition \ref{proposition:MOYflowsirreducibles} is an isomorphism of \textit{complex-analytic spaces} is not addressed in \cite{MOY} and it won't be needed for our purposes either. However, this statement is true, and can be proved by establishing that both $M( \mathfrak{C}^{+}(E_0 ), \mathfrak{C}^{+}(E_1))$ and $ \mathcal{D}(R , \widehat{E} )$ are complex-analytic spaces carrying a \textit{universal family} of effective divisors on $R$ in the line bundle $\widehat{E}$. 
\end{enumerate}
\end{remark}

\begin{proof}[Proof summary] The proof involves establishing a dimensional reduction phenomenon, as in the proof of Lemma \ref{proposition:MOYirreducible}. 

$ $

\textit{Step 1: From flowlines to vortices on the cylinder.} The locus $R^{o} = R \setminus ( C_- \cup C_+ ) = N \setminus C_-$ is a non-compact complex manifold diffeomorphic to the cylinder $\mathbb{R} \times Y$. Let us fix a particular identification, using the hermitian metric on $N$: 
\[
v \in N \setminus C_- \mapsto ( \mathrm{log} |v|  ,  v / |v| ) \in \mathbb{R} \times Y.
\]
We fix on $R^{o} = \mathbb{R}_{t} \times Y$ a cylindrical Kähler metric, with $(1,1)$ form
\[
\omega_{R^{o}} = dt \wedge \eta  + \Pi^\ast \omega_C .
\]

If $(A, \Phi)$ is a flowline connecting two irreducible critical manifolds, and since these are Morse--Bott critical manifolds (cf. Proposition \ref{proposition:MOYirreducible}), then $(A, \Phi )$ approaches its limits with \textit{exponential decay} after applying a gauge transformation. The equation for the negative gradient flow of $\mathcal{L} : \mathcal{B}(Y, \mathfrak{s}_E ) \rightarrow \mathbb{R}$ is identified with the Seiberg--Witten equation (in temporal gauge) on the cylinder $R^{o} = \mathbb{R} \times Y$. Combining these with a standard integration by parts argument (e.g. see \cite{morgan}), it was proved in \cite{MOY} that, after a gauge transformation defined on $R^{o}$ approaching $1$ at $\pm \infty$, the pair $(A, \Phi)$ solves the \textit{vortex equations over the non-compact Kähler surface} $R^{o}$: writing $\Phi = (\alpha , \beta ) \in \Gamma ( R^{o} , \pi_{2}^\ast E \otimes \pi_{2}^{\ast} ( \mathbb{C} \oplus K_{C}^{-1} ) ) $ these read 
\begin{align*}
2 \Lambda F_A - \Lambda \Pi^\ast F_{K_C} & = i ( |\alpha|^2 - |\beta |^2 )\\
F_{A}^{0,2} & = 0\\
\overline{\partial}_{A} \alpha = 0 , & \quad \overline{\partial}_{A}^\ast \beta = 0 \\
\alpha = 0 & \text{  or  } \beta = 0 \, ,
\end{align*}
where $\Lambda$ denotes the adjoint Lefschetz operator associated to the Kähler form $\omega_{R^{o}}$, and $\pi_{2} : R^{o} = \mathbb{R} \times Y \rightarrow Y$ is the projection to the second factor. A solution $(A , \alpha , \beta )$ to the above equations over $R^{o}$ with $\alpha \neq 0$ or $\beta \neq 0$ will be referred to as a \textit{vortex on the cylinder} $R^{o}$ in the line bundle $\pi_{2}^\ast E \rightarrow R^{o}$. In the first case we call it a \textit{positive} vortex, and in the second a \textit{negative} vortex. Let $M_{v} (E_0 , E_1)  = M_{v}^{+} (E_0 , E_1) \sqcup M_{v}^{-} (E_0 , E_1 )$ be the moduli space of vortices on $R^{o}$ in the line bundle $\pi_{2}^\ast E$ approaching vortices on $C$ in $E_0$ and $E_1$ as $t \rightarrow - \infty$ and $t \rightarrow + \infty$, respectively, modulo $U(1)$ gauge transformations on $R^{o}$ converging to $1$ as $t \rightarrow \pm \infty$. All combined we have a homeomorphism $M(\mathfrak{C}^{\pm}(E_0 ) , \mathfrak{C}^{\pm} (E_1 )) \cong M_{v}^{\pm} (E_0 , E_1 )$, and we learn that there cannot exist flowlines from the $\mathfrak{C}^{\pm}$ to the $\mathfrak{C}^{\mp}$ critical manifolds. 



$ $

\textit{Step 2: From vortices on the cylinder to divisors on the ruled surface.} 

Let $(A, \alpha )$ be a positive vortex on the cylinder $R^{o}$, converging with exponential decay to (positive) vortices on $E_0$ and $E_1$, as $t \rightarrow - \infty$ and $t \rightarrow + \infty$, respectively.

An equivalent description of the line bundle $\widehat{E} $ over $R$ is that it is obtained by gluing together the line bundles $(\Pi^\ast E_0 )|_{R \setminus C_+} $ and $(\Pi^\ast E_1 )|_{R \setminus C_-}$ over the region $R^{o} = \mathbb{R} \times Y$ using the isomorphisms $(\Pi^\ast E_0 )|_{R^{o}} \cong \pi_{2}^\ast E \cong (\Pi^\ast E_1 )|_{R^{o}}$ (induced from isomorphisms $ \pi^\ast E_{0} \cong E \cong \pi^\ast E_{1}$). The convergence with exponential decay guarantees that $\alpha$ \textit{extends} over to the ruled surface $R$ as a \textit{continuous} section $\widehat{\alpha}$ of $\widehat{E}$, and it is shown in \cite{MOY} that the holomorphic structure $\overline{\partial}_A$ on $\widehat{E}|_{R^{o}}$ \textit{extends} over to $\widehat{E}$. Since $\widehat{\alpha}$ is holomorphic for $\overline{\partial}_A$ over $R^{o}$, the regularity of the Cauchy--Riemann operator implies that $\widehat{\alpha}$ must be holomorphic over the whole $R$.

All combined, this says that the vortex $(A, \alpha )$ extends to an effective divisor on $R$ in the line bundle $\widehat{E}$. Conversely, \cite{MOY} solve a Kazdan--Warner equation in order to establish that: given an effective divisor on $R$ in $\widehat{E}$ there exists a complex gauge transformation $g$ defined on $R^{o}$, unique up to $U(1)$ gauge transformations approaching $1$ as $t \rightarrow \pm \infty$, such that acting by $g$ on the effective divisor yields a vortex on $R^{o}$.

$ $

\textit{Step 3: Deformation theory.} The space of flows $M(\mathfrak{C}^{+} (E_0 ) , \mathfrak{C}^{+}(E_1) )$ is the zero set of a Fredholm section $s$ of a suitable Banach vector bundle $\mathcal{V} \rightarrow \mathcal{B}$. As a matter of general principles, one can then show that near a given flow $(A, \Phi = (\alpha , 0 ) )$ the moduli space is modelled on the zero locus of a smooth map $H^{0}_{sw} (A, \Phi )\xrightarrow{\kappa_{sw}} H^{1}_{sw}(A , \Phi )$, where $H^{0}_{sw} (A, \Phi )$ and $H^{1}_{sw} (A , \Phi )$ are the (finite-dimensional) kernel and cokernel of the (vertical projection of the) derivative of $s$ at $(A , \Phi )$ and $\kappa_{sw}(0) = (d \kappa_{sw})_0 = 0$ (see \S 4.3 in \cite{DK} or Theorem B.9 in \cite{salamon}). 

In \cite{MOY} it is shown that $H^{\ast}_{sw}(A, \Phi ) \cong H^{\ast} ( D , \widehat{\mathcal{E}}|_D )$, where $D = \widehat{\alpha}^{-1}(0)$ and $\widehat{\mathcal{E}} = ( \widehat{E} , \overline{\partial}_A )$ (see Lemma 12.6 in \cite{salamon} for the analogous result in the case of a \textit{compact} Kähler surface). The reason why $\kappa_{sw}$ can be chosen to be a complex-analytic (i.e. holomorphic) map is the following. First, the space of pairs $(A, \Phi )$ has a complex structure: $\Phi$ is a section of a complex vector bundle; and the space of connections $A$ can be identified as the space of Cauchy--Riemann operators (or "half connections") on the corresponding bundle, which is naturally complex as well. The linearisation of the Seiberg--Witten map \textit{with gauge-fixing} is a \textit{complex-linear} map (this operator appears as formula (12.3) in \cite{salamon}). Thus, the kernel and cokernel, which are given by $H_{sw}^{0}(A , \Phi ), H_{sw}^{1}(A , \Phi )$, are complex vector spaces. In this situation, the Kuranishi map $\kappa_{sw} $ can be taken to be holomorphic, which follows from the proof of Theorem B.9 in \cite{salamon} (namely, the pseudo-inverse $T$ in that argument can be chosen to be an $S^1$-equivariant operator, by applying Remark B.8 to the derivative of $f$, where the $S^1$ action is that induced from complex scalar multiplication on the tangent space at $(A, \Phi )$ to the space of pairs).
\end{proof}

\subsubsection{Flows from irreducibles to reducibles}\label{subsubsection:flowsreducible}

Fix a spin-c structure $\mathfrak{s}_E$ on $Y = S(N)$. 

We denote by $\lfloor \frac{K_C}{2} \rfloor$ (resp. $\lceil \frac{K_C}{2} \rceil$) the unique line bundle on $C$ of maximal (resp. minimal) degree among those line bundles $E_0$ with $E \cong \pi^\ast E_0$ (recall that all $E_0$ with $\pi^\ast E_0 \cong E$ differ from each other by a tensor power of $N$, by Lemma \ref{lemma:pullbackline}) and $\mathrm{deg} E_0 \leq \frac{1}{2} \mathrm{deg}K$ (resp. $\mathrm{deg} E_0 \geq \frac{1}{2} \mathrm{deg}K_C$). We stress that the line bundles $\lfloor \frac{K_C}{2}\rfloor, \lceil \frac{K_C}{2} \rceil$ depend on the chosen line bundle $E$ over $Y$.

The following values $\delta_{0}^{\pm} \in [0,1) \cap \mathbb{Q}$ will play an important role throughout the article:
\begin{align}
\delta_{0}^{+} = \frac{\mathrm{deg} \lfloor \frac{K_C}{2} \rfloor - \frac{1}{2} \mathrm{deg}K_C}{\mathrm{deg}N} \,  , \quad
\delta_{0}^{-} = \frac{\frac{1}{2} \mathrm{deg}K_C - \mathrm{deg}\lceil \frac{K_C}{2}\rceil }{\mathrm{deg}N}. \label{eq:fibrewiseeigen}
\end{align}

Fix now another line bundle $E_0$ with $E \cong \pi^\ast E_0$. Denote by $\widehat{E} = \widehat{E}(E_0 )$ the unique line bundle over $R$ with $\widehat{E}|_{C_-} = E_0$ and $\widehat{E}|_{C_+} = \lfloor \frac{K_C}{2}\rfloor $. Namely, this is $\widehat{E} = (\Pi^\ast E_0 ) \otimes \mathcal{O}_{R}(k)$ where 
\[
k = \frac{ \mathrm{deg}E_0  - \mathrm{deg} \lfloor \frac{K_C}{2} \rfloor}{\mathrm{deg} N}.
\]

\begin{proposition}[Theorems 10.0.15  and 10.1.1] \label{proposition:MOYflowsreducible}Assume that the reducible critical manifold $\mathfrak{C}^{\redu}(E)$ is Morse--Bott (cf. Proposition \ref{proposition:MOYreducible}). With $E, E_0$ as above, then the moduli space $M^{o}( \mathfrak{C}^{+}(E_0 ) , \mathfrak{C}^{\redu}(Y,E ))$ of based parametrised flows of $\mathcal{L} : \mathcal{B}(Y , \mathfrak{s}_E ) \rightarrow \mathbb{R}$  from $\mathfrak{C}^{+}(E_0 )$ to $\mathfrak{C}^{\redu}(Y ,E)$ is homeomorphic to the open subset of the moduli space $\mathcal{D}^{o}(R , \widehat{E} )$ of based effective divisors on the ruled surface $R$ in the line bundle $\widehat{E}$ which do not contain the zero section $C_- \subset R$. If $(\widehat{\mathcal{E}} , \alpha )$ is such an effective divisor, with zero set $D = \alpha^{-1}(0)$, then the moduli space of based parametrised flows $M( \mathfrak{C}^{+}(E_0 ) , \mathfrak{C}^{\redu}(Y,E))$ is modelled near $(\widehat{\mathcal{E}} , \alpha )$ on the zero locus of a complex-analytic map $\kappa^{o}_{sw} : H^{0}(D , \widehat{\mathcal{E}}|_D ) \oplus \mathbb{C}\rightarrow H^{1}(D , \widehat{\mathcal{E}}|_D )$ such that $\kappa^{o}_{sw} (0 ) = 0$ and $(d \kappa^{o}_{sw})_{0} = 0$.

The analogous result for $M( \mathfrak{C}^{-}(E_0 ) , \mathfrak{C}^{\redu}(Y,E))$ can be deduced from the above using the charge-conjugation symmetry.
\end{proposition}

\begin{proof}[Proof summary]

The proof follows similar steps as the proof of Proposition \ref{proposition:MOYflowsirreducibles}, and we only review some aspects that are different. Like before, the Morse--Bott assumption on $\mathfrak{C}^{\redu}(E)$ guarantees that flowlines converging to $\mathfrak{C}^{\redu}(E)$ as $t \rightarrow + \infty$ will do so with exponential decay, after applying a gauge transformation. With this in place, Step 1 proceeds as before. 

In Step 2, we start by building the line bundle $\widehat{E}$ over the ruled surface $R$ explicitly by gluing $(\Pi^\ast E_0 )|_{R \setminus C_+}$ with $(\Pi^\ast \lfloor \frac{K_C}{2} \rfloor )|_{R \setminus C_-}$ over $R^{o} = \mathbb{R}_t \times Y$ using the isomorphisms $(\Pi^\ast E_0 )|_{R^{o}} \cong \pi_{2}^\ast E \cong (\Pi^\ast \lfloor \frac{K_C}{2} \rfloor)|_{R^{o}}$ together with multiplication by $e^{t \delta^{+}_0 }$ on $\pi_{2} E \rightarrow R^{o}$ (the reason for this choice of gluing will be clear momentarily). Suppose $(A, \alpha )$ is positive vortex on the cylinder $R^{o}$ in the line bundle $\pi_{2}^\ast E$, converging to a positive vortex on $C$ in $E_0$ as $t\rightarrow - \infty$ and to a reducible monopole $B$ in $E$ as $t \rightarrow + \infty$, with exponential decay. We define an effective divisor on $R^{o}$ in $\widehat{E}|_{R^{o}}$ as the pair
\begin{align}
( e^{t \delta_{0}^+ } \overline{\partial}_A e^{-t \delta_{0}^+} , e^{t \delta_{0}^+ } \alpha ) . \label{eq:divisor}
\end{align}
Note also the identity
\[
e^{t \delta_{0}^+ } \overline{\partial}_A e^{-t \delta_{0}^+} = \overline{\partial}_{A + i \delta_{0}^{+} \eta } .
\]

We claim (see Proposition 10.0.19 in \cite{MOY}) that (\ref{eq:divisor}) \textit{extends} to an effective divisor on the ruled surface $R$ in the line bundle $\widehat{E}$. The extension over to $C_- \subset R$ is established as in the proof of Proposition \ref{proposition:MOYflowsirreducibles}. To establish the extension over $C_+ \subset R$, we first assert that (\ref{eq:divisor}) converges as $t \rightarrow + \infty$ to a holomorphic pair
\[(\overline{\partial}_{B_0} , \alpha_0 )\]
in the line bundle $\lfloor \frac{K_C}{2} \rfloor$ over $C$, where $\overline{\partial}_{B_0}$ is a holomorphic structure on $ \lfloor \frac{K_C}{2} \rfloor$, and $\alpha_0$ a holomorphic section of said bundle, possibly vanishing everywhere).

The assertion is established as follows. The bundle with connection $(E , B + i \delta_{0}^{+} \eta )$ has fibrewise trivial holonomy and its curvature pulls back from $C$, hence descends to a bundle with connection $(E_0 , B_0 )$ on $C$. Upon calculating the degree of $E_0$, one sees that $E_0 \cong \lfloor K_C /2 \rfloor$. On the other hand, the flowline $(A, \alpha )$ converges in the blow-up model of the configuration space to an eigenvector $\alpha_+$ of $D_B$ with eigenvalue $\delta > 0$, which moreover is a section of the subbundle $E \subset E \oplus ( E \otimes \pi^\ast K_{C}^{-1} )$ (since $(A, \alpha )$ is a positive vortex). The section $\alpha$ over $R^{o}$ has asymptotics $\alpha \sim e^{- \delta t} \alpha_+$ as $t \rightarrow + \infty$. By the block form (\ref{eq:diracblock}) of $D_B$ we see that
\begin{align}
i \nabla_{\zeta}^{\pi^\ast B_0 } \alpha_+  = ( \delta - \delta_{0}^{+} ) \alpha_+ \quad , \quad \overline{\partial}_{\pi^\ast B_0} \alpha_+ = 0 . \label{eq:alpha+}
\end{align}
In particular $\delta - \delta_{0}^{+} \in \mathbb{Z}$, and since $\delta \geq 0$, $\delta_{0}^{+} \in [0,1)$ then we see that $\delta \geq \delta_{0}^{+}$. There are two cases to consider now:
\begin{itemize}
\item If $\delta > \delta_{0}^{+}$, then the convergence assertion follows with $\alpha_0 := 0 $.
\item If $\delta = \delta_{0}^{+}$ then by (\ref{eq:alpha+}) the section $\alpha_+$ of $E$ descends to a non-trivial holomorphic section of $(\lfloor K_C / 2\rfloor , \overline{\partial}_{B_0} )$ and the convergence assertion follows with $\alpha_0 := \alpha_+$.
\end{itemize}

The section $e^{t \delta_{0}^+ } \alpha$ thus extends continuously over $R$ as a section of $\widehat{E}$, and is holomorphic over $R^{o}$. Similarly as before, the holomorphic structure $e^{t \delta_{0}^+ } \overline{\partial}_A e^{-t \delta_{0}^+}$ is shown in \cite{MOY} to extend over $R$, and by the regularity of the Cauchy--Riemann operator the section $e^{t \delta_{0}^+ } \alpha$ is holomorphic over the whole $R$ with respect to this holomorphic structure.
\end{proof}

For example, if $D_0$ is an effective divisor on $C$ in the line bundle $E_0$, then 
\[
D = \Pi^\ast D_0 + \frac{ \mathrm{deg}E_0  - \mathrm{deg} \lfloor \frac{K_C}{2} \rfloor}{\mathrm{deg} N} \cdot C_+
\]
is an effective divisor in $\mathcal{D}(R , \widehat{E} )$ not containing the zero section $C_-$. In general, an effective divisor $D$ on $R$ in $\widehat{E}$ can be uniquely written as 
\[
D = D^\prime + n  \cdot C_+
\]
where $D^\prime$ does not contain $C_+$ and $n \in \mathbb{Z}_{\geq 0}$. The integer $n$ is the multiplicity of the vanishing of $D$ along the section at infinity $C_+$, and is related to the eigenvalue of the projectivised eigenspace to which the lift of the flowline to the blowup converges as $t \rightarrow + \infty$, as the following result addresses:


\begin{proposition}\label{proposition:MOYflowsblowup}Assume that the reducible critical manifold $\mathfrak{C}^{\redu}(E)$ is Morse--Bott (cf. Proposition \ref{proposition:MOYreducible}). With $E, E_0$ as above, then the locus in the moduli space $M^{o}( \mathfrak{C}^{+}(E_0 ) , \mathfrak{C}^{\redu}(Y,E ))$ corresponding to blown-up based parametrised flows into a projectivised eigenspace with eigenvalue $ \geq \delta_{0}^{+} + n$, where $n \in \mathbb{Z}_{\geq 0}$, is homeomorphic to the locus in $\mathcal{D}^{o}(R , \widehat{E} )$ consisting of based effective divisors $D$ which contain the section at infinity $C_+ \subset R$ with multiplicity at least $n$, and which do not contain the zero section $C_- \subset R$. 



\end{proposition}

\begin{proof}
Let $(A, \Phi )$ be a flowline whose lift to the blowup converges a projectivised eigenspace with eigenvalue $\delta> 0$. The effective divisor on $R$ corresponding to $(A, \Phi )$ is $( e^{t \delta_{0}^+ } \overline{\partial}_A e^{-t \delta_{0}^+} , e^{t \delta_{0}^+ } \alpha )$, and we have
\[
e^{t \delta_{0}^{+}} \alpha = e^{t (\delta_{0}^{+}  - \delta )} \cdot ( e^{ t\delta} \alpha) \, ,
\] 
where $e^{t \delta} \alpha$ is not identically vanishing over $C_+$ (over $C_+$ it restricts to the section $\alpha_+ \neq 0$ from Step 2 in the proof of Proposition \ref{proposition:MOYflowsreducible}). Now, the function $e^{-t n}$ over $R^{o}$ extends to the canonical holomorphic section $z^n$ of $\mathcal{O}_{R}(n)$ over the ruled surface $R$ (see \S \ref{subsubsection:ruled}), and $\delta_{0}^+ - \delta \leq  - n$ occurs precisely when the effective divisor on $R$ vanishes with multiplicity at least $n$ over $C_+$, which proves the first assertion. 
\end{proof}

\section{The eigenspaces of the Dirac operator on Seifert-fibered $3$-manifolds}\label{section:dirac}

In this section we describe the eigenspaces of the Dirac operator $D_B$ twisted by a \textit{reducible} monopole on a Seifert-fibered $3$-manifold. More precisely, we describe the eigenspaces in the \textit{adiabatic limit} as the \textit{area of the base} of the Seifert fibration $Y \rightarrow C$ \textit{goes to zero}.

Let $Y = S(N)$ be a Seifert-fibered $3$-manifold over an orbifold closed oriented surface $C$, with $\mathrm{deg}N < 0$. We choose holomorphic and geometric data (i)-(iv) on $C$, as in \S \ref{subsection:spin}. In this section we regard the holomorphic data (i.e. the holomorphic structures on $C$ and on $N \rightarrow C$) as \textit{fixed}, whereas we will \textit{vary the geometric data} $g_C $, $i \eta$. Recall (Remark \ref{remark:choices}, item 3) that the connection $i \eta$ is unique (up to hermitian gauge transformation of $N$) given the data (i)-(iii) has been chosen, so essentially we are only varying the metric $g_C$ in its conformal class. Up to conformal automorphisms of $C$, this just amounts to varying $g_C$ through \textit{constant rescalings}, and hence the area $\mathrm{Area}(C, g_C)$ is the only relevant parameter in our discussion.

We fix once and for all a line bundle $E \rightarrow Y$ and \textit{holomorphic data} on $E$, namely
\begin{enumerate}
\item[(v)] a holomorphic structure $\mathcal{E}_0$ up to isomorphism on \textit{some} line bundle $E_0 \rightarrow C$ with $E \cong \pi^\ast E_0$. Since $N$ has a holomorphic structure up to isomorphism (given by the holomorphic data (i) from \S \ref{subsection:spin}), and every line bundle pulling back to $E$ is of the form $E_0 \otimes N^k$, then a holomorphic structure on some $E_0$ singles out holomorphic structures on \textit{every} line bundle on $C$ that pulls back to $E$, compatibly with tensoring by the holomorphic bundle $N$.
\end{enumerate}
By Proposition \ref{proposition:MOYreducible}, the holomorphic data (v) on $E$ determines a unique reducible monopole $B$ in the line bundle $E$ up to hermitian gauge transformations. 

With this in place, let $V_\lambda$ denote the $\lambda$-eigenspace of the Dirac operator $D_B : \Gamma (Y , S_E ) \rightarrow \Gamma (Y, S_E )$. Our goal is to give an explicit description of the eigenspaces $V_\lambda$ for suitable choices of geometric data. 
In order to state the main result of this section, we first introduce some notation. As in \S \ref{subsubsection:flowsreducible}, we denote by $\lfloor \frac{K_C}{2} \rfloor$ (resp. $\lceil \frac{K_C}{2} \rceil$) the line bundle on $C$ of maximal (resp. minimal) degree among those line bundles $E_0$ with $E \cong \pi^\ast E_0$ and $\mathrm{deg} E_0 \leq \frac{1}{2} \mathrm{deg}K$ (resp. $\mathrm{deg} E_0 \geq \frac{1}{2} \mathrm{deg}K_C$), and consider the real numbers $\delta_{0}^{\pm} \in [0,1) \cap \mathbb{Q}$ given by
\begin{align}
\delta_{0}^{+} = \frac{\mathrm{deg} \lfloor \frac{K_C}{2} \rfloor - \frac{1}{2} \mathrm{deg}K_C}{\mathrm{deg}N} \,  , \quad
\delta_{0}^{-} = \frac{\frac{1}{2} \mathrm{deg}K_C - \mathrm{deg}\lceil \frac{K_C}{2}\rceil }{\mathrm{deg}N}. 
\end{align}

\begin{remark}\label{remark:cases} It is worth noting that:
\begin{itemize}
\item If there exists an orbi-spin bundle $E_0$ pulling back to $E$, then $\lfloor \frac{K_C}{2} \rfloor = \lceil\frac{K_C}{2}\rceil $ and hence $\delta_{0}^+ = \delta_{0}^{-} = 0$.
\item Otherwise, we have $\delta_{0}^{\pm} > 0$. We then have $1 = \delta_{0}^+ + \delta_{0}^-$, since in this case $\lfloor \frac{K_C}{2} \rfloor \otimes N^{-1} \cong \lceil\frac{K_C}{2}\rceil $
\end{itemize}
\end{remark}
We consider the following vector spaces of orbifold sheaf cohomology of holomorphic line bundles: 
\begin{align}
\mathcal{U}_{\lambda}^+  = \begin{cases} 0 &\quad \text{if } \lambda \notin \delta_{0}^+ + \mathbb{Z}\\
 H^0 (C, \lfloor \frac{K_C}{2} \rfloor \otimes N^{n} ) &\quad \text{if } \lambda = \delta_{0}^+ + n \,  (\in \delta_{0}^+ + \mathbb{Z})
\end{cases}
\end{align} \label{sheafcohomologies}
\[
\mathcal{U}_{\lambda}^-  = \begin{cases} 0 &\quad \text{if } \lambda \notin \delta_{0}^- + \mathbb{Z}\\
 H^1 (C, \lceil \frac{K_C}{2} \rceil \otimes N^{-n} ) &\quad \text{if } \lambda = \delta_{0}^- + n \,  (\in \delta_{0}^- + \mathbb{Z})
 \end{cases}
 \]
where we recall once more that $N$, $\lfloor \frac{K_C}{2} \rfloor$ and $\lceil \frac{K_C}{2} \rceil$ all carry preferred holomorphic structures (up to isomorphism) determined from the holomorphic data (ii) from \S \ref{subsection:spin} and the holomorphic data (v) on $E$, which will always be implicit in our notation in what follows.

With this in place, the main result of this section is the following

\begin{theorem}\label{theorem:eigenspaces}
Fix holomorphic data (i)-(ii) on $C$ as in \S \ref{subsection:spin}, together with holomorphic data (v) on the line bundle $E \rightarrow Y$, and a positive constant $\Lambda > 0$. There exists a constant $\epsilon > 0$ (depending on the previous data) such that if $(g_C, i \eta)$ is a choice of geometric data on $C$ with $\mathrm{Area}(C, g_C) < \epsilon$ and $B$ is a reducible monopole in $E$ corresponding to the holomorphic data (v) on $E$ (cf. Proposition \ref{proposition:MOYreducible}), then the $\lambda$-eigenspace $V_\lambda$ of the Dirac operator $D_B$ is given by $V_\lambda \cong \mathcal{U}^{+}_\lambda \oplus \mathcal{U}^{-}_\lambda$ for all $\lambda \in [-\Lambda , \Lambda ]$.

\end{theorem}

The next three subsections are devoted to the proof of Theorem \ref{theorem:eigenspaces}.

\subsection{The Dirac operator on a Seifert fibration}

Next, we describe some general properties of the Dirac operator $D_B$ when $B$ is a reducible monopole in $E$. We may fix an identification $E = \pi^\ast E_0$ where $E_0 =\lfloor \frac{K_C}{2} \rfloor$ and take $B$ to be the connection given by the formula (\ref{eq:reducible}).

The operator $D_B$ can be decomposed into vertical and horizontal components
\begin{align*}
D_B = D_{B}^v & + D_{B}^h \\
D_{B}^v = \begin{pmatrix} i \nabla_{\zeta}^B & 0 \\ 0 & - i \nabla_{\zeta}^B \end{pmatrix} \, &, \quad D_{B}^h = \begin{pmatrix} 0 & \sqrt{2} \cdot \overline{\partial}_{B}^\ast \\ \sqrt{2} \cdot \overline{\partial}_B & 0 \end{pmatrix}.
\end{align*}
The operator $D_{B}^v$ (resp. $D_{B}^h$) only involves derivatives in the fiber (resp. base) directions and is a bounded linear operator when regarded as an operator between Sobolev spaces of sections $L^{2}_k (Y, S_E )\rightarrow L^{2}_{k-1}(Y, S_E )$ (however, it is \textit{not} a Fredholm operator). The differential operator $D_{B}^v$ restricts naturally to a differential operator along each fiber $\pi^{-1} (x) \cong S^1$ of the Seifert fibration $\pi : Y \rightarrow C$. From the formula (\ref{eq:reducible}) we immediately recognise this as the operator on $\pi^{-1}(x)$ given by the sum of two (twisted) Dirac operators on the circle:
\begin{align*}
(i \partial_\zeta + \delta_{0}^+ ) \oplus ( -i \partial_\zeta - \delta_{0}^+ ) : \Gamma (\pi^{-1}(x) , \mathbb{C} \oplus \mathbb{C} ) \rightarrow \Gamma (\pi^{-1}(x) , \mathbb{C} \oplus \mathbb{C} ).
\end{align*}

Thus, the set of eigenvalues of the restriction to $\pi^{-1}(x)$ of the operator $D_{B}^v$ \textit{do not depend on} $x \in C$, and are given by the discrete set $\sigma := ( \delta_{0}^+ + \mathbb{Z} ) \cup (- \delta_{0}^+ + \mathbb{Z} ) \subset \mathbb{R}.$ From Remark \ref{remark:cases} we have
\begin{align*}
\sigma = (\delta_{0}^+ + \mathbb{Z}) \cup (\delta_{0}^{-} + \mathbb{Z} ).
\end{align*}

Because the fiberwise eigenvalues are constant, a given spinor $\psi \in \Gamma (Y, S_E )$ can be \textit{uniquely} decomposed as a sum 
\begin{align}
\psi = \sum_{\delta \in \sigma} \psi_\delta \label{eq:fourier}
\end{align}
where  $\psi_\delta \in \Gamma (Y, S_E )$ satisfies the equation $D_{B}^v \psi_\delta = \delta \psi_\delta$, and the $\psi_\delta$ are pairwise $L^2$ orthogonal (in fact, fiberwise $L^2$ orthogonal). 

This is essentially a Fourier decomposition of $\psi$ along the fibers. Because $B$ is a \textit{reducible} monopole, then the horizontal operator $D_{B}^h$ interacts nicely with this kind of decomposition. By Lemma \ref{lemma:commutator} we have
\[
D_{B}^v D_{B}^h + D_{B}^h D_{B}^v = 0
\]
and thus we deduce
\begin{align}
D_{B}^v \big( D_{B}^h \psi_\delta \big)  = - \delta D_{B}^h \psi_\delta \label{eq:diracidentity}
\end{align}
i.e. the operator $D_{B}^h$ sends a fiberwise $\delta$-eigenvector to a fiberwise $(- \delta)$-eigenvector. We will exploit this fact in what follows.

\subsection{Topological eigenspaces}

The eigenspace $V_\lambda$ of $D_B$ depends on the holomorphic and geometric data (i)-(v), but it contains the following subspace $U_\lambda \subset V_{\lambda}$ which, as we will see now, only depends on holomorphic data (i)-(ii) on $C$ and holomorphic data (v) on $E$:
\[
U_\lambda = \Big\{ \psi \in \Gamma (Y, S_E ) \, | \, D_{B}^v \psi = \lambda \psi \, , \, D_{B}^h \psi = 0 \Big\}.
\]
Decomposing $\psi = (\alpha , \beta )$ according to the splitting (\ref{eq:spinorbundle}), the equations defining $U_\lambda$ can be written as
\[
i \nabla^{B}_\zeta \alpha = \lambda \alpha \,, \quad - i \nabla^{B}_{\zeta} \beta = \lambda \beta \, , \quad \overline{\partial}_B \alpha = 0 \, , \quad \overline{\partial}^{\ast}_B \beta = 0 .
\]
Thus, we have a natural splitting $U_\lambda = U^{+}_\lambda \oplus U^{-}_\lambda$ where 
\begin{align*}
U_{\lambda}^{+} &= \Big\{ \alpha \in \Gamma (Y, E ) \, | \, \,\, i \nabla^{B}_\zeta \alpha = \lambda \alpha  \, , \, \overline{\partial}_B \alpha = 0
\Big\}\\
U_{\lambda}^{-} &= \Big\{ \beta \in \Gamma (Y, E \otimes \pi^\ast K_{C}^{-1} ) \, | \,\,\,  -i \nabla^{B}_\zeta \beta = \lambda \beta  \, , \, \overline{\partial}^{\ast}_B \beta = 0
\Big\}.
\end{align*}

\begin{lemma}\label{lemma:eigenspacepulled}
For every $\lambda \in \mathbb{R}$, there are isomorphisms $U^{\pm}_\lambda \cong \mathcal{U}_{\lambda}^{\pm}$ with the sheaf cohomology vector spaces from (\ref{sheafcohomologies}). 
\end{lemma}

Thus, the elements contained in the subspace $U_{\lambda} \subset V_\lambda$ sit there for a \textit{topological} reason (meaning that they are stable under changes of geometric data, and only depend on holomorphic data). This motivates the following

\begin{definition}
The eigenspace $V_\lambda$ of $D_B$ is \textit{topological} if $U_\lambda = V_\lambda$.
\end{definition}

\begin{proof}[Proof of Lemma \ref{lemma:eigenspacepulled}]
Let us start understanding a given element $ \alpha \in U_{\lambda}^+$. The equation $i \nabla^{B}_\zeta \alpha = \lambda \alpha$ can be rewritten as $i \nabla^{B + i \lambda \eta }_\zeta \alpha = 0$. If $\alpha$ is not identically zero, then it follows that $B$ is a connection with trivial fiberwise holonomy. Furthermore, since $2F_B = F_{K_C}$ then the curvature form $F_B$ is pulled back from the base. These two conditions imply that the bundle with connection $(E, B + i \lambda \eta)$ is isomorphic to the pullback of a corresponding pair $(E_0 , B_0 )$ over $C$, and $\alpha$ gives a holomorphic section of $E_0$ (with respect to the holomorphic structure $\overline{\partial}_{B_0}$). From $\pi^\ast B_0 = B + i \lambda \eta$ we deduce the identity
\begin{align}
\mathrm{deg} E_0 = \frac{1}{2} \mathrm{deg}K_C + \lambda F_N .\label{eq:degrees}
\end{align}
On the other hand, $\lambda$ is a fiberwise eigenvalue and, more specifically, it is one of the form $\lambda = \delta_{0}^+ + n$ for some $n \in \mathbb{Z}$. Together with (\ref{eq:degrees}), this shows that there is an isomorphism of holomorphic bundles $( E_0 , \overline{\partial}_{B_0} ) \cong \lfloor \frac{K_C}{2} \rfloor \otimes N^{n}$. From this, an isomorphism $U_{\lambda}^+ \cong \mathcal{U}_{\lambda}^+$ follows. 

The case of a non-zero $\beta \in U_{\lambda}^-$ is similar. In this case, we similarly deduce that the bundle with connection $(E, B-i\lambda \eta )$ is the pullback of a line bundle $(E_0 , B_0 )$ on $C$, and $\beta$ yields an element of $H^1 (C, E_0 )$. In this case $\lambda$ will be of the form $\delta_{0}^{-} + n $, and from this we deduce $( E_0  , \overline{\partial}_{B_0} ) \cong \lceil \frac{K_C}{2} \rceil \otimes N^{-n}$.
\end{proof}

For convenience, we introduce the following notation 

\begin{align}
E_{\delta}^+ &:= \lfloor \frac{K_C}{2} \rfloor \otimes N^n \, ,  \quad \text{for } \delta  = \delta_{0}^+ + n \label{eq:line+}\\
E_{\delta}^- &:= \lceil \frac{K_C}{2} \rceil \otimes N^{-n} \, , \quad \text{for } \delta = \delta_{0}^- + n \label{eq:line-}.
\end{align}

\subsection{From fiberwise eigenvectors to Laplace eigenvectors}

On the Kähler orbifold $(C, g_C)$ we have the Hodge Laplacians acting on $(0,0)$ forms and $(0,1)$ forms. These operators can be twisted by an orbifold hermitian line bundle with connection $(E_0 , B_0 )$ to give operators
\[
\overline{\partial}_{B_0}^{\ast} \overline{\partial}_{B_0} : \Gamma (C, E_0 ) \rightarrow \Gamma (C, E_0 ) \, , \quad \overline{\partial}_{B_0} \overline{\partial}_{B_0}^{\ast} : \Gamma (C, E_0 \otimes K^{-1}_C ) \rightarrow \Gamma (C, E_0 \otimes K^{-1}_C).
\]
The line bundles $E_{\delta}^\pm$ defined in (\ref{eq:line+}-\ref{eq:line-}) carry holomorphic structures up to isomorphism (provided by the holomorphic data (v)). Each $E_{\delta}^{\pm}$ has a unique hermitian connection $B_{\delta}^{\pm}$ with constant curvature inducing the given isomorphism class of holomorphic structure. The following is a key calculation:

\begin{lemma}\label{lemma:eigenlaplace}
Let $\psi \in V_\lambda$ be an eigenvector of the Dirac operator $D_B$. If some $\psi_\delta$ in the decomposition $(\ref{eq:fourier})$ is not identically zero, then at least one of the Laplacians $\overline{\partial}_{B_{\delta}^+}^{\ast} \overline{\partial}_{B_{\delta}^+}$, $\overline{\partial}_{B_{\delta}^-} \overline{\partial}_{B_{\delta}^-}^{\ast}$ has an eigenvalue at $\lambda^2 - \delta^2$.
\end{lemma}

In particular, because the Laplacians are non-negative operators, we deduce that if $\psi \in V_\lambda$ then the decomposition (\ref{eq:fourier}) of $\psi$ only has \textit{finitely-many} non-zero $\psi_\delta$'s, namely
\[
\psi = \sum_{\delta \in \sigma \, , \, \delta^2 \leq \lambda^2} \psi_\delta .
\]

\begin{proof}
Decomposing $\psi$ according to (\ref{eq:fourier}), and applying $D_B$ we deduce 
\[
\sum_{\delta \in \sigma} \lambda \psi_\delta = \sum_{\delta \in \sigma} \Big( \delta \psi_\delta + D^{h}_B \psi_\delta \Big) .
\]
Because the decomposition (\ref{eq:fourier}) is $L^2$-orthogonal, together with the key property (\ref{eq:diracidentity}), we deduce 
\begin{align}
(\lambda - \delta ) \psi_\delta = D^{h}_B \psi_{-\delta} \, , \quad \forall \delta \in \sigma . \label{eq:diracidentity2}
\end{align} If we write this equation in components $\psi_\delta = (\alpha_\delta , \beta_\delta )$ this says
\begin{align*}
(\lambda - \delta) \alpha_\delta &= 
\overline{\partial}_{B}^\ast \beta_{-\delta} \\
(\lambda - \delta )\beta_\delta & =  \overline{\partial}_{B} \alpha_{-\delta} .
\end{align*}
 Applying $\overline{\partial}^{\ast}_B$ to the second equation, and then using the first equation yields
 \begin{align}
 \overline{\partial}^{\ast}_B \overline{\partial}_{B} \alpha_\delta = (\lambda^2 - \delta^2 )\alpha_\delta \,, \quad \forall \delta \in \sigma \label{eq:eigenlaplace}
 \end{align}
and similarly
 \begin{align*}
 \overline{\partial}_B \overline{\partial}^{\ast}_{B} \beta_\delta = (\lambda^2 - \delta^2 )\beta_\delta \, , \quad \forall \delta \in \sigma .
 \end{align*}
Let us suppose $\alpha_\delta$ is not identically zero (the case when $\beta_\delta$ is not identically zero is similar). Then, by the argument from the proof of Lemma \ref{lemma:eigenspacepulled}, we have that the pair $(E,  B + i \delta \eta )$ is isomorphic to the pullback of the pair $(E_{\delta}^+ , B_{\delta}^+ )$. Then $\alpha_\delta$ becomes a non-trivial section of $E_{\delta}^+$ which, by (\ref{eq:eigenlaplace}), satisfies the eigenvector equation for the Hodge Laplacian
\[
\overline{\partial}^{\ast}_{B_{\delta}^+} \overline{\partial}_{B_{\delta}^+} \alpha_\delta = (\lambda^2 - \delta^2 )\alpha_\delta .
\]
\end{proof}


\begin{lemma}\label{lemma:smallarea}
Given a constant $\Lambda > 0$ exists a constant $\epsilon > 0$ (depending on $\Lambda$ and on the holomorphic data (i)-(ii) on $C$ and the holomorphic data (v) on $E$) such that: if $\mathrm{Area}(C,g_C) < \epsilon$ and $\lambda \in [-\Lambda , \Lambda]$, then neither of the two Laplacians $\overline{\partial}_{B_{\delta}^+}^{\ast} \overline{\partial}_{B_{\delta}^+}$, $\overline{\partial}_{B_{\delta}^-} \overline{\partial}_{B_{\delta}^-}^{\ast}$ has an eigenvalue at $\lambda^2 - \delta^2$ for some $\delta \in \sigma$ with $\delta^2 < \lambda^2$.
\end{lemma}

\begin{proof}
Re-scaling the metric $g_C$ by a factor of $a >0$ has the effect of re-scaling the Laplacians $\overline{\partial}_{B_{\delta}^+}^{\ast} \overline{\partial}_{B_{\delta}^+}$, $\overline{\partial}_{B_{\delta}^-} \overline{\partial}_{B_{\delta}^-}^{\ast}$ (and therefore also their eigenvalues) by a factor of $1/a$. For each $\delta \in \sigma$ with $\delta^2 < \Lambda^2$ we can therefore choose $a > 0$ small enough that the first positive eigenvalue of both Laplacians $\overline{\partial}_{B_{\delta}^+}^{\ast} \overline{\partial}_{B_{\delta}^+}$, $\overline{\partial}_{B_{\delta}^-} \overline{\partial}_{B_{\delta}^-}^{\ast}$ (defined using the rescaled metric $a \cdot g_C$) is greater than $\Lambda^2 - \delta^2$. Since there are only finitely many $\delta \in \sigma$ with $\delta^2 < \Lambda^2$, then the desired result follows.
\end{proof}

\begin{proof}[Proof of Theorem \ref{theorem:eigenspaces}]
Let $\lambda \in [-\Lambda , \Lambda]$ be fixed, and let $\psi \in V_\lambda$. If $\mathrm{Area}(C,g_C) < \epsilon$, then it follows from Lemma \ref{lemma:eigenlaplace} and Lemma \ref{lemma:smallarea} that in the decomposition (\ref{eq:fourier}) all $\psi_\delta$ vanish except for possibly $\psi_{\pm\lambda}$. That is, 
\begin{align}
\psi = \psi_\lambda + \psi_{-\lambda}. \label{eq:pmlambda}
\end{align}
From (\ref{eq:diracidentity2}) we also have $D_{B}^h \psi_{-\lambda} = 0$, and thus $\psi_{-\lambda} \in U_{-\lambda} \subset V_{-\lambda}$. Taking the inner product of (\ref{eq:pmlambda}) with $\psi_{-\lambda}$ and integrating over $Y$ yields $\psi_{-\lambda} = 0$ (here $\langle \psi , \psi_{-\lambda \rangle_{L^{2}(Y)}} = 0$ because $\psi \in V_\lambda$ and $\psi_{-\lambda} \in V_{-\lambda}$, and $\langle \psi_\lambda, \psi_{-\lambda} \rangle_{L^{2}(Y)} = 0$ because $\psi_{\pm \lambda}$ is a fibrewise $(\pm \lambda)$-eigenvector of $D_{B}^v$).

Hence, we've shown $\psi = \psi_{\lambda}$, i.e. $\psi \in U_\lambda$. This shows that the eigenspace $V_\lambda$ is topological, i.e. $U_\lambda = V_\lambda$. Together with Lemma \ref{lemma:eigenspacepulled}, the proof of Theorem \ref{theorem:eigenspaces} is now complete.
\end{proof}

\subsection{The Seifert $S^1$-action on the eigenspaces of the Dirac operator}

Next, we discuss how the Seifert $S^1$-action on $Y = S(N)$ interacts with the eigenspaces of the Dirac operator $D_B$. Here $B$ will be a reducible monopole in the spin-c structure corresponding to a line bundle $E \rightarrow Y$.

We fix a \textit{lift} of the $S^1$-action on $Y$ to the spin-c frame bundle. Concretely, this just amounts to fixing an identification of $E$ with the pullback $\pi^\ast E_0$ of an orbifold line bundle on $C$. This lift makes the spinor bundle $S_E \rightarrow Y$ into an $S^1$-equivariant vector bundle, and the Clifford multiplication $\rho : TY \rightarrow \mathrm{Hom}(S_E , S_E )$ into an $S^1$-equivariant bundle map. 


The eigenspaces $V_\lambda$ of the Dirac operator $D_B$ now become complex $S^1$-representations, and we want to describe their weight decomposition. More precisely, we only do this for the subrepresentation $U_\lambda \subset V_\lambda$ (recall from the proof of Theorem \ref{theorem:eigenspaces} that $V_\lambda = U_\lambda$ provided the area of $C$ is small enough). The complex $S^1$-representation $U_\lambda$ further splits as $U_\lambda = U_\lambda^{+} \oplus U_{\lambda}^{-}$ according to the splitting $S_E = E \oplus ( E \otimes \pi^\ast K_{C}^{-1} )$.

\begin{proposition}\label{proposition:weights}
If $U_{\lambda}$ is non-trivial, then the complex $S^1$-representation $U_\lambda^{\pm}$ has weight $c \mp \lambda \in \mathbb{Z}$, where $c =  \frac{ \mathrm{deg}E_0 - \frac{1}{2} \mathrm{deg}K_C}{\mathrm{deg}N}$. 
\end{proposition}

\begin{proof}


Recall that the Dirac operator $D_B$ restricts along each fiber $\pi^{-1}(x)$ of the Seifert fibration $Y \xrightarrow{\pi} C$ to the sum of two twisted Dirac operators 
\[
(i \partial_\zeta + c ) \oplus (- i \partial_\zeta - c ) : \Gamma (\pi^{-1}(x) , \mathbb{C} \oplus \mathbb{C}) \rightarrow \Gamma (\pi^{-1}(x) , \mathbb{C} \oplus \mathbb{C}) .
\]
Let $\psi = (\alpha , \beta) \in U^{+}_{\lambda} \oplus U^{-}_{\lambda} = U_\lambda$. Thus, we have 
\[
i \partial_\zeta \alpha + c \alpha = \lambda \alpha \quad , \quad -i \partial_\zeta  \beta - c \beta = \lambda \beta 
\]
from which it follows that
\[
\alpha ( e^{i \theta} y ) = e^{i (c - \lambda)\theta}\alpha (y) \quad , \quad \beta ( e^{i \theta}y) = e^{i (  c + \lambda)\theta} \beta (y)
\]
at each point $y \in Y$. The assertion now follows from this.
\end{proof}

\subsection{The case when the canonical spin-c structure is self-conjugate }

We conclude this section by spelling out some consequences of Therem \ref{theorem:eigenspaces} and Proposition \ref{proposition:weights} in the case when the \textit{canonical spin-c structure} $\mathfrak{s}_{c}$ on $Y$ (i.e. $E = \mathbb{C}$) is \textit{self-conjugate}, which is assumed throughout this section. This means that $\pi^\ast K_C$ is the trivial line bundle on $Y$, which is equivalent to saying that $K_C \cong N^{-m}$ for some integer $m$, which is necessarily given by 
\[m = - \frac{\mathrm{deg}K_C}{\mathrm{deg}N} = \frac{ \chi (C)}{\mathrm{deg}N}.
\]
We have $\lfloor \frac{K_C}{2} \rfloor \cong N^{- \lfloor m / 2\rfloor}$, $\lceil \frac{K_C}{2} \rceil \cong N^{-\lceil m/2 \rceil}$ and
\[
\delta_0 := \delta_{0}^+ = \delta_{0}^- = \begin{cases} 0 &\quad \text{if $m$ is even}\\
1/2 & \quad \text{if $m$ is odd}.
\end{cases}
\]
We note that the integer $m$ is even precisely when there exists an orbi-spin bundle $E_0$ on $C$ pulling back to the trivial line bundle $E =  \mathbb{C}$. For example, if $Y = \Sigma ( \alpha_1 , \ldots , \alpha_n )$ is a Seifert-fibered integral homology sphere then this occurs precisely when all $\alpha_i$ are \textit{odd}.

We let $B$ be a reducible monopole in $E = \mathbb{C}$. Recall that by Theorem \ref{theorem:eigenspaces}, the eigenspace $V_{\lambda}$ of $D_B$ agrees with $\mathcal{U}_{\lambda}$ in the small area regime. The spaces $\mathcal{U}_{\lambda} = \mathcal{U}_{\lambda}^{+} \oplus \mathcal{U}_{\lambda}^{-}$ are empty unless $\lambda \in \delta_0 + \mathbb{Z}$.

\begin{lemma}\label{proposition:selfconj}
Suppose that the canonical spin-c structure on $Y$ is self-conjugate, with $K_C \cong N^{-m}$. Then for any $n \in \mathbb{Z}$ there is an isomorphism
\begin{align*}
\mathcal{U}^{\pm}_{\delta_{0} + n  } \cong H^0 (C , N^{n-\lfloor m/2 \rfloor}).
\end{align*}
\end{lemma}

\begin{proof}
The result for $\mathcal{U}_{\delta_0 
+ n }^{+}$ is clear. Using Serre duality we have 
\[
\mathcal{U}^{-}_{\delta_0 + n} = H^{1}( C , N^{- \lceil m /2 \rceil - n} ) \cong H^{0} (C , N^{n + \lceil m /2 \rceil - m} ) = H^{0} (C , N^{n - \lfloor m /2 \rfloor }).\] 
\end{proof}



Now fix some lift of the Seifert $S^1$-action to the spin-c frame bundle, i.e. $ E \cong \pi^\ast E_0$. By the following result, as long as the orbifold Euler characteristic $\chi (C)$ is non-zero, there exists a special eigenvalue on whose projectivised eigenspace the Seifert $S^1$-action is non-trivial:

\begin{corollary}
Suppose that the canonical spin-c structure on $Y$ is self-conjugate. If $\lambda =  \frac{1}{2} \frac{\chi (C)}{\mathrm{deg}N}$, then $\mathcal{U}_{\lambda} $ is a complex $2$-dimensional $S^1$-representation of weight 
 \[(c- \frac{1}{2}\frac{\chi(C)}{\mathrm{deg}N} , c + \frac{1}{2}\frac{\chi(C)}{\mathrm{deg}N} ),\]
 where $c \in \mathbb{R}$ is as in Proposition \ref{proposition:weights}. In particular, the $S^1$-action on $\mathbb{P}(\mathcal{U}_\lambda ) \approx \mathbb{P}^1$ induced by the Seifert $S^1$-action on $Y$ has degree $ \frac{\chi(C)}{\mathrm{deg}N}$.
\end{corollary}
\begin{proof}
Noting the identity
\[
\delta_0 + \lfloor m/2 \rfloor =  \frac{1}{2} \frac{\chi (C)}{\mathrm{deg}N} = \lambda  \quad ,
\]
the assertion just follows from Lemma \ref{proposition:selfconj} and Proposition \ref{proposition:weights}.

\end{proof}


When $\delta_0 = 1/2$ then $\mathcal{U}_0 = 0$. The following result pertains the case $\delta_0 = 0$,

\begin{lemma}\label{lem: vanishing homology 1} 
Suppose that the canonical spin-c structure on $Y$ is self-conjugate, with $K_C \cong N^{-m}$. In addition, suppose that $m$ is even and $C$ has genus $0$. Then $H^{0}(C,N^{-\frac{m}{2}})=0$.
\end{lemma}
\begin{proof}
 Denote the Seifert data of $N^{-\frac{m}{2}}$ by $(e,\frac{\beta_{1}}{\alpha_{1}},\cdots, \frac{\beta_{n}}{\alpha_{n}})$ with $e\in \mathbb{Z}$ and $0\leq \beta_{i}<\alpha_{i}$. Note that $N^{-\frac{m}{2}}$ is a square root of $K_{C}$. Since the Seifert data of $K_{C}$ is $(-2,\frac{\alpha_{1}-1}{\alpha_{1}},\cdots,\frac{\alpha_{n}-1}{\alpha_{n}})$ then we have 
\[-2=2e+\sum^{2}_{i=1}\lfloor \frac{2\beta_{i}}{\alpha_{i}}\rfloor.\] 
In particular, we have $e<0$. This implies $H^{0}(C,N^{-\frac{m}{2}})=0$, since the holomorphic sections of the orbifold bundle $N^{-m/2}$ are in 1-1 correspondence with the holomorphic sections of its desingularisation.
\end{proof}
Combining Lemma \ref{proposition:selfconj} and Lemma \ref{lem: vanishing homology 1}, we obtain the following
\begin{corollary}\label{cor: Dirac no kernel}
Suppose that the canonical spin-c structure on $Y$ is self-conjugate and $C$ has genus $0$. Then $\mathcal{U}_{0}=0$. In particular, the blown-up reducible critical manifolds of $Y$ in the canonical spin-c structure are Morse--Bott.    
\end{corollary}


\section{Calculations for rational homology spheres} \label{section:calculations}

Throughout this section $Y = S(N)$ is a Seifert-fibered $3$-manifold over the orbifold Riemann surface $C$, and $\mathrm{deg}N < 0$. In addition, we will also assume that $C$ has \textit{genus zero}, which is equivalent to the condition that $Y = S(N)$ is a \textit{rational homology sphere}. In this situation the calculation of effective divisors on the ruled surface is much more elementary than in the general case, and we will use it to describe the topology of the moduli spaces of flows, and conclude by describing a few examples.

\subsection{Effective divisors on the ruled surface}

The Picard groups (both the holomorphic and the topological) of the ruled surface $R = \mathbb{P}(\mathcal{O}_C \oplus N )$ are given in terms of those of the curve $C$ by (\ref{pict}-\ref{pic}). 
Since $C$ has genus zero, then $\mathrm{Pic}(C) = \mathrm{Pic}^t (C)$ and thus $\mathrm{Pic}(R) = \mathrm{Pic}^t (R)$, i.e. any topological orbifold line bundle over $R$ has a unique holomorphic structure up to isomorphism. Hence, we will not make further mention to the holomorphic structure on a line bundle over $C$ or $R$.

With this in place, we now move on to describing the holomorphic sections of line bundles over the ruled surface $R$. Recall from (\ref{pic}) that any orbifold line bundle on $R$ is of the form $\Pi^\ast E_0 \otimes \mathcal{O}_{R}(n)$ for unique $E_0 \in \mathrm{Pic}(C)$ and $n \in \mathbb{Z}$, where $\Pi : R \rightarrow C$ stands for the bundle projection. 

\begin{lemma} There is a natural isomorphism of orbifold sheaves over $C$,  
\[
\Pi_\ast \big( \Pi^\ast E_0 \otimes \mathcal{O}_R (n ) \big) \cong E_0 \otimes \mathrm{Sym}^n (\mathcal{O}_C \oplus N )^{\ast} = \bigoplus_{j = 0}^{n} E_0 \otimes N^{-j} .
\]
In particular,
\begin{align}
H^0 (R ,  \Pi^\ast E_0 \otimes \mathcal{O}_R (n) ) \cong \bigoplus_{j = 0}^n H^0 ( C, E_0 \otimes N^{-j} ) \label{equation:sectiondecomposition}\\
H^1 (R ,  \Pi^\ast E_0 \otimes \mathcal{O}_R (n) ) \cong \bigoplus_{j = 0}^n H^1 ( C, E_0 \otimes N^{-j} ) . \label{equation:H1}\end{align}
\end{lemma}
\begin{proof}
The first assertion is standard (e.g. Proposition 7.11 in \cite{hartshorne}). It is a generalization of the familiar fact that
\[H^0 (\mathbb{C}P^m , \mathcal{O}_{\mathbb{C}P^m }(n) ) \cong \mathrm{Sym}^n (\mathbb{C}^{m+1})^\ast .\]
The formula (\ref{equation:sectiondecomposition}) follows immediately from the first assertion. 

We now establish (\ref{equation:H1}). Let us denote $L = \Pi^\ast E_0 \otimes \mathcal{O}_{R}(n) $. First, we claim that the orbifold sheaf $R^{1}\Pi_\ast L  $ (i.e. the $1$st derived functor associated to the functor $\Pi_\ast$, acting on $L$) is zero. For this, note that the $H^1$ of the restriction of the sheaf $L$ to any fiber of $\Pi : R \rightarrow C$ vanishes, because the fibers of $\Pi$ are $\mathbb{P}^1$ (in the case of the fiber over an orbifold point of $C$, we consider instead the $\mathbb{P}^1$ fiber on the lift of an orbifold chart) and $H^{1}( \mathbb{P}^1 , L_0) = 0 $ for any line bundle $L_0$. Thus, by Grauert's Semicontinuity Theorem we have $R^{1} \Pi_\ast L = 0$. With this in place, the Leray spectral sequence
\[
E^{p,q}_{2} = H^p (C , R^q  \Pi_\ast L ) \implies H^{p+q}(R , L )
\]
together with $R^1 \Pi_\ast  L = 0$ (and $R^0 \Pi_\ast = \Pi_\ast$ by definition) show that $H^{1}(C , \Pi_\ast L ) \cong H^{1} (R , L )$ as required.
\end{proof}

When $n \geq 1$ the line bundle $\mathcal{O}_R (n)$ has a canonical holomorphic section, which we denote $z^n$, that vanishes precisely over the section at infinity $C_+$ with multiplicity $n$. More generally, a non-trivial holomorphic section of the bundle $\Pi^\ast E_0 \otimes \mathcal{O}_{R}(n)$ over the ruled surface $R$ coming from the subspace $H^0 (C , E_0 \otimes N^{-j} )$ from the decomposition (\ref{equation:sectiondecomposition}) (with $j  \in \{0 , \ldots , n\}$) vanishes over $C_-$ with multiplicity $j$, and over $C_+$ with multiplicity $n-j$.

\begin{proposition}\label{proposition:divisorsgenus0}
Suppose that the genus of $C$ is zero. Then the space $\mathcal{D}(R , L )$ of effective divisors on $R$ in the line bundle $L = \Pi^\ast E_0 \otimes \mathcal{O}_R (n)$ which do not contain $C_{-}$ (resp. $C_+$) 
is homeomorphic to the open subset of the projective space $\mathbb{P}\big( \oplus_{j = 0}^n H^0 (C, E_0 \otimes N^{-j} ) \big) $ given by those $s_0 :  s_1 :  \ldots :  s_n $ with $s_0 \neq 0$ (resp. $s_n \neq 0$). \end{proposition}

\begin{proof}
Let us describe the space of effective divisors not containing $C_-$ (the case of $C_+$ is identical). Consider a non-trivial holomorphic section $s$ of $\pi^\ast E_0 \otimes \mathcal{O}_R (n )$, which we can write as $s = s_0 + \ldots + s_n$ where $s_j$ comes from $H^0 (C, E_0 \otimes N^{-j} )$. If $s_0$ vanishes identically, then the effective divisor $D = \{ s = 0\} \subset R$ contains $C_-$. If $s_0$ is not identically vanishing, then $s$ doesn't vanish identically on $C_-$ (because $s_1 , \ldots , s_n$ all vanish identically over $C_-$ but $s_0$ does not) -- i.e. $D$ doesn't contain $C_-$. The result follows.
\end{proof}


A similar description applies to the moduli space $\mathcal{D}^{o}(R, L)$ of based divisors, which is given by $\big( \oplus_{j = 0}^{n} H^{0}(C , E_0 \otimes N^{-j} )\big) \setminus 0 $. Of course, the residual $\mathbb{C}^\ast$-action is the standard one that yields the projective space $\mathbb{P}\big( \oplus_{j = 0}^n H^0 (C, E_0 \otimes N^{-j} ) \big) $ after quotienting. 

There is also an additional free $\mathbb{C}^\ast$-action on the open locus in the space of effective divisors $\mathcal{D}(R, L )$ (resp. based divisors $\mathcal{D}^{o}(R,L)$) which do not contain $C_\pm$ which extends the \textit{translation} $\mathbb{R}$-action on the moduli space of parametrised flows (resp. based parametrised flows). We will thus refer to it as the translation $\mathbb{C}^\ast$-action, and it can identified as follows. The ruled surface $R = \mathbb{P}(\mathbb{C} \oplus N )$ has the $\mathbb{C}^\ast$-action given by the trivial action on $\mathbb{C}$ and the standard (weight $1$) action on $N$. This $\mathbb{C}^\ast$-action on $R$ has fixed locus $C_\pm$, with $C_-$ (resp. $C_+$) being repelling (resp. attracting). It induces a natural action on the space of effective divisors $D \subset R $ in $L$ not containing $C_\pm$. In terms of the decomposition (\ref{equation:sectiondecomposition}) of a holomorphic section $s$ of $L = \Pi^\ast E_0 \otimes \mathcal{O}_R (n)$ as $ s = s_0 + \cdots s_n $, the translation $\mathbb{C}^\ast$ action is given by 
\[
\lambda \cdot s = s_0 + \lambda s_1 + \cdots \lambda^n s_n \quad , \quad \lambda \in \mathbb{C}^\ast .
\] 

At this point, Proposition \ref{proposition:divisorsgenus0}, Proposition \ref{proposition:MOYflowsirreducibles}, Proposition \ref{proposition:MOYflowsreducible} and Proposition \ref{proposition:MOYflowsblowup} can used to determine the topology of the moduli spaces of flows (either parametrised or unparametrised or based) between critical manifolds. We record the following:

\begin{corollary}\label{cor:topologymoduli}
Suppose that the genus of $C$ is zero. Let $E$ be a line bundle on $Y$, and $E_0, E_1$ be orbifold line bundles on $C$ pulling back to $E$. Then
\begin{itemize}
\item Let $k = \frac{ \mathrm{deg}E_0 - \mathrm{deg}E_1}{\mathrm{deg}N}$. The moduli space $M (\mathfrak{C}^{+} (E_0 ) , \mathfrak{C}^{+}(E_1 )$ of parametrised flows between irreducible critical manifolds is homeomorphic to the open subset of the projective space $\mathbb{P} \big( \bigoplus_{j = 0}^{k} H^{0}(C , E_0 \otimes N^{-j} )  \big) $ given by the complement of the projective subspaces $\mathbb{P} \big( \bigoplus_{j = 1}^{k} H^0 (C , E_0 \otimes N^{-j}) \big)$ and $\mathbb{P} \big( \bigoplus_{j = 0}^{k-1} H^0 (C , E_0 \otimes N^{-j} ) \big)$
\item Let $k = \frac{ \mathrm{deg}E_0 - \mathrm{deg} \lfloor K_C /2 \rfloor }{\mathrm{deg}N}$. The moduli space $\breve{M} (\mathfrak{C}^{+}(E_0 ) , \mathfrak{C}^{\redu}(E ) )$ of unparametrised flows into the reducible critical manifold is homeomorphic to the open subset of the projective space $\mathbb{P} \big( \bigoplus_{j = 0}^{k} H^{0}(C , E_0 \otimes N^{-j} )  \big) $ given by the complement of the projective subspace $\mathbb{P} \big( \bigoplus_{j = 1}^{k} H^0 (C , E_0 \otimes N^{-j} ) \big)$

\item Let $k$ be the same as in the previous item, and $n \geq 0$ a non-negative integer. The moduli space $\breve{M} ( \mathfrak{C}^{+}(E_0 ) , \mathbb{P}(V_{\delta_{0}^{+} + n} ) )$ of blown-up unparametrised flows into the eigenspace $\mathbb{P} (V_{\delta_{0}^{+} + n})$ of the Dirac operator at the reducible monopole in $E$ is homeomorphic to the open subset of the projective space $\mathbb{P} \big( \bigoplus_{j = 0}^{k-n} H^{0} (C , E_0 \otimes N^{-j} ) \big) $ given by the complement of the projective subspaces $\mathbb{P}\big(  \bigoplus_{j = 1}^{k-n} H^{0}(C , E_0 \otimes N^{-j}) \big)$ and $\mathbb{P}\big( \bigoplus_{j = 0}^{k-n-1} H^{0} (C , E_0 \otimes N^{-j} ) \big)$.
\end{itemize}
\end{corollary}

In particular, the moduli spaces of unparametrised flows between irreducibles are \textit{odd dimensional}, whereas the moduli spaces of unparametrised flows into reducibles are \textit{even dimensional}.


As for their deformation theory, we note the following:

\begin{corollary}\label{cor: Kuranishi}
Suppose that the genus of $C$ is zero. Then the moduli spaces of parametrised flows $M( \mathfrak{C}^{+}(E_0 ) , \mathfrak{C}^{+}(E_1 ) )$ 
are ``clean intersections" i.e. the Kuranishi maps \[\kappa_{sw} : H^{0}(D , \widehat{\mathcal{E}}|_D ) \rightarrow H^{1} (D , \widehat{\mathcal{E}}|_D )\] from Proposition \ref{proposition:MOYflowsirreducibles} all vanish identically. The same holds for the moduli spaces of based parametrised flows into reducibles (i.e. the Kuranishi maps from Proposition \ref{proposition:MOYflowsreducible} 
vanish identically).

\end{corollary}

\begin{proof}
We consider only the case pertaining Proposition \ref{proposition:MOYflowsirreducibles} (the other case is similar). Since $\kappa_{sw}$ is complex-analytic, then its zero set is a complex-analytic space with dimension bounded above by $\mathrm{dim}_{\mathbb{C}} H^{0} (D , \widehat{\mathcal{E}}|_D )$, the dimension of its Zariski tangent space. In the case at hand, with $C$ of genus zero, we also know that it is homeomorphic to an open subset of the projective space $\mathbb{P}(H^{0}(R , \widehat{\mathcal{E}} )  )$, hence the dimension of the zero locus of $\kappa_{sw}$ is $\mathrm{dim}_{\mathbb{C} } \mathbb{P}(H^{0}(R , \widehat{\mathcal{E}} )  )$. But $ \mathrm{dim}_{\mathbb{C} } \mathbb{P}(H^{0}(R , \widehat{\mathcal{E}} )  ) = \mathrm{dim}_{\mathbb{C}} H^{0} (D , \widehat{\mathcal{E}}|_D )$ due to the exact sequence
\[
0 \rightarrow \mathcal{O}_R \rightarrow \widehat{\mathcal{E}} \rightarrow \widehat{\mathcal{E}}|_D \rightarrow 0 
\]
and the vanishing $H^{1}(R , \mathcal{O}_R ) = 0$. It follows that the dimension of the zero locus of $\kappa_{sw}$ agrees with the dimension of the Zariski tangent space, and so $\kappa_{sw}$ must vanish.
\end{proof}

In the remainder of this section we provide some sample computations of Seiberg--Witten moduli spaces on Seifert-fibered rational homology spheres.

\subsection{The canonical monopoles}
When the canonical bundle $K_C$ has \textit{positive degree} ($\mathrm{deg} K_C > 0$) then there is a distinguished effective divisor on $C$ with degree less than $ \frac{1}{2}\mathrm{deg}K_C$, namely the \textit{empty} divisor. This gives, under the correspondence from Proposition \ref{proposition:MOYirreducible}, an irreducible monopole $(B , \psi = (\alpha , \beta ))$ on $Y$ in the trivial line bundle $E = \mathbb{C}$ with $B$ given by the trivial connection, $\beta = 0$ and $\alpha = \mathrm{const.}$ where
\[
\mathrm{const.} = \sqrt{ \frac{2 \pi \mathrm{deg}K_C}{\mathrm{Area}C}}.
\]
We refer to $(B,\psi)$ as the \textit{canonical positive monopole} (whenever $\mathrm{deg}K_C > 0$). This is always a non-degenerate critical point of the Chern--Simons--Dirac functional. Under charge conjugation, the canonical positive monopole corresponds to a negative monopole (i.e. with $\alpha = 0$) in the line bundle $\pi^\ast (K_C )$ that we refer to as the \textit{canonical negative monopole}.

We will be interested in flowlines from the canonical positive monopole into a particular eigenspace:

\begin{corollary}\label{corollary:flowscontact}
Suppose that $C$ has genus zero and $\mathrm{deg}K_C > 0$. For the canonical spin-c structure $E = \mathbb{C}$ on $Y$, let $n \geq 0$ be the unique integer with $\lfloor \frac{K_C}{2} \rfloor \cong N^{-n}$. Then the space of unparametrised blown-up flowlines from the canonical positive monopole into the critical manifold $\mathbb{P}(V_{\delta_{0}^{+} + n}) \cong \mathbb{P}^1$ consists of a single flowline. 

If, in addition, the first positive eigenvalue of the Dirac operator at the reducible monopole in $E = \mathbb{C}$ is given by $\delta_{0}^{+}+n$, then this flowline is transversely cut-out if and only if $\bigoplus_{j = 1}^{n} H^{1}(C , N^{-j} ) = 0$. 
\end{corollary}
\begin{proof}
We use the third item in Corollary \ref{cor:topologymoduli}, with $E_0 = \mathbb{C}$ and $k = n = \frac{ - \mathrm{deg} \lfloor K_C /2 \rfloor }{\mathrm{deg}N}$. It follows that the moduli space consists of a single flowline: $\mathbb{P} ( H^0 ( C , \mathcal{O}_C ) ) = \{ \text{point} \}$. This unique flowline corresponds to the divisor $D = n C_+$ in the bundle $\widehat{E} = \mathcal{O}_R (n)$ cut out by the canonical section $z^n$.

When the eigenvalue $\delta_{0}^{+}+n$ is the first positive eigenvalue of the Dirac operator at the reducible, then by Proposition \ref{proposition:MOYflowsreducible} the obstruction space of this flowline is given by $H^{1} ( D , \mathcal{O}_{R}(n)|_D )$, and the flowline is transversely cut-out precisely when this vanishes. In order to compute the obstruction space we can use the short exact sequence of orbifold sheaves 
\[0 \rightarrow \mathcal{O}_R \xrightarrow{z^n} \mathcal{O}_{R}(n) \rightarrow \mathcal{O}_{R}(n)|_D \rightarrow 0\] which induces a long exact sequence in cohomology, a piece of which is 
\[
\cdots \rightarrow H^{1}( R , \mathcal{O}_R ) \rightarrow H^1 (R , \mathcal{O}_R (n) ) \rightarrow H^{1} ( D , \mathcal{O}_R (n)|_D ) \rightarrow 0 .
\]
By (\ref{equation:H1}) we have $H^{1}(R , \mathcal{O}_R ) \cong H^{1} (C , \mathcal{O}_C )$ and $H^{1}(R , \mathcal{O}_R (n) ) \cong \bigoplus_{j = 0}^{n} H^{1}(C, N^{-j} )$. Since $C$ has genus zero then $H^{1}(C , \mathcal{O}_C ) = 0$, so it follows that $H^{1}(D , \mathcal{O}_R (n)|_D ) \cong \bigoplus_{j = 1}^{n} H^{1}(C , N^{-j} )$.
\end{proof}

\subsection{Some examples}

Below we describe the critical loci and the moduli spaces of flowlines for some rational homology spheres.




Before, we make some comments. We will not discuss the spaces of flowlines in the blow-up that connect reducibles with reducibles; these are easily described as a complement of projective subspaces in a suitable projective space (see \cite{FLINthesis}, Lemma 3.16). We adopt the convention that whenever an assertion is made about a particular eigenspace of the Dirac operator $D_B$ for the reducible monopole $B$ in a spin-c structure then it is understood that this assertion holds only after possibly taking $\mathrm{Area}(C,g_C) \ll 1$ possibly depending on that eigenvalue (so that we can apply Theorem \ref{theorem:eigenspaces}).

\subsubsection{Unit tangent bundle}
Let $Y = S(TC)$ be the unit tangent bundle of a hyperbolic orbifold closed oriented surface of genus zero, as in Example \ref{example:S(TC)}. It follows that $N = TC$ has negative orbifold degree, by the hyperbolicity condition. Note that the special case when $C$ has $3$ orbifold points of isotropy $2  , 3 , 7$ gives the Brieskorn integral homology sphere $Y = \Sigma(2,3,7)$. We will consider only the canonical spin-c structure $E = \mathbb{C}$, which is self-conjugate (since $K_C $ is a power of $N = TC$, namely $K_C = N^{-1}$). By Proposition \ref{proposition:MOYirreducible}, the positive irreducible monopoles in $E = \mathbb{C}$ correspond to effective divisors in a line bundle $E_0 = N^{-k} = K_{C}^{k}$ with $0 \leq \mathrm{deg}E_0 < \frac{1}{2} K_C$. Thus, the only possibility is $E_0 = \mathbb{C}$, and hence there is only the canonical positive monopole $\mathfrak{c}$. So the irreducible critical manifols are given by $\mathfrak{c}$ and its conjugate $\overline{\mathfrak{c}}$. As for the boundary stable critical manifolds, by Lemma \ref{proposition:selfconj} these are given by the $\mathbb{P} (V_{1/2 + n} ) = \mathbb{P} ( H^0 ( C , N^n )^{\oplus 2} )$ with $n \geq 0$. These are non-empty only for the case $n = 0$, i.e. $\mathbb{P} (V_{1/2} ) \cong \mathbb{P}^1$. By Corollary \ref{corollary:flowscontact} there exists a unique flowline from each of $\mathfrak{c}$ and $\overline{\mathfrak{c}}$ into the critical manifold $\mathbb{P}^1$, and this flowline is transversely cut out.


\subsubsection{$Y = \Sigma(2,3,11)$}

Now $Y = S(N)$ where $\mathrm{deg}N = - 1/66$, and $K_C = N^{-5}$. The only line bundles $E_0$ on $C$ with $0 \leq \mathrm{deg}E_0 < \frac{1}{2} \mathrm{deg}K_C = 5 / 132$ are $E_0 = \mathbb{C} , N^{-1}$ or $N^{-2}$. One can show that the background degrees of $N^{-1}$ and $N^{-2}$ (i.e. the degrees of the desingularisations $|N^{-1}|$ and $|N^{-2}|$ over $|C|$) are both negative (equal to $-1$). But the holomorphic sections of an orbifold line bundle $E_0$ are in correspondence with those of its desingularisation $|E_0|$. Hence again the only positive irreducible monopole is the canonical one $\mathfrak{c}$. As for boundary-stable critical manifolds, these are given by Lemma \ref{proposition:selfconj} as the $\mathbb{P}(V_{1/2 + n} ) = \mathbb{P} ( H^{0}(C , N^{n-2} )^{\oplus 2} )$, $n \geq 0$. If $n > 2$ these are empty since $\mathrm{deg} N^{n-2} < 0$, and for $n = 0,1$ we calculated that these are empty as well above. The only reducible critical manifold is thus $\mathbb{P} (V_{5/2} ) \cong \mathbb{P}^1$. Again, there is exactly one flowline from each irreducible (i.e. $\mathfrak{c}$ and $\overline{\mathfrak{c}}$) into $\mathbb{P}^1$, which are transverse.

\subsubsection{$Y = \Sigma(2,3,13)$}

Now one has $K = N^{-7}$, and there are again exactly two irreducible critical points $\mathfrak{c}, \overline{\mathfrak{c}}$ and one boundary-stable reducible manifold $\mathbb{P} (V_{7/2} ) \cong \mathbb{P}^1$, together with exactly one flowline from each irreducible to the reducible. However, the flowlines are \textit{not transverse} anymore, since $H^1 (C, N^{-1}) = H^{0}(C , N^{-6 }) \cong \mathbb{C}$ is non-trivial (see Corollary \ref{corollary:flowscontact}).

\subsubsection{$Y = \Sigma(2,3,23)$}

Now $K_C = N^{-17}$. In this case there are two positive irreducible monopoles: the canonical one $\mathfrak{c}$ coming from the empty effective divisor in $E_0 = \mathbb{C}$; and an additional one $\mathfrak{a}$ given by the non-trivial holomorphic section in $E_0 = N^{-6}$, which vanishes over the orbifold point of isotropy order $23$ to fractional order $1/23$. Thus, in total there are four irreducibles: $\mathfrak{c} , \overline{\mathfrak{c}} , \mathfrak{a} , \overline{\mathfrak{a}}$. There are two boundary-stable critical manifolds: $\mathbb{P}(V_{1/2 +2}) = \mathbb{P} ( H^0 (C , N^{-6} )^{\oplus 2} ) \cong \mathbb{P}^1$ and $\mathbb{P} (V_{1/2 +8} ) = \mathbb{P} ( H^0 (C , \mathcal{O}_C )^{\oplus 2} ) \cong \mathbb{P}^1$. 
The topology of the spaces of flows is given by:
\begin{itemize}
\item 
The space of unparametrised flows from $\mathfrak{c}$ to $\mathbb{P}(V_{1/2+8})$ consists of a single flowline. 

\item 
The space of unparametrised flows from $\mathfrak{c}$ into $\mathbb{P}(V_{1/2+2})$ is homeomorphic to the complement of two points in the projective space $\mathbb{P}(H^0 (C , N^0 \oplus N^{-6} ) ) \cong \mathbb{P}^1$. 

\item The space of parametrised flows from $\mathfrak{c}$ into $\mathfrak{a}$ is the complement of two points in the same projective space $\cong \mathbb{P}^1$ from the previous item, and hence the space of unparametrised flows from $\mathfrak{c}$ to $\mathfrak{a}$ is a circle.

\item The space of unparametrised flows from $\mathfrak{a}$ into $\mathbb{P}(V_{1/2+8})$ is empty.

\item The space of unparametrised flows from $\mathfrak{a}$ into $\mathbb{P}(V_{1/2 + 2} )$ consists of a single flow.

\item The space of flows from $\mathfrak{a}$ into $\mathfrak{c}$ is empty.
\end{itemize}



\section{Monopole Floer homology and invariants of families}\label{section:familiescobordism}

\subsection{A Morse--Bott model for monopole Floer homology}

We shall use Francesco Lin's \cite{FLINthesis} Morse-Bott approach to monopole Floer homology, so we briefly review it here.
We do not intend to give a self-contained exposition here, since it would just be a lengthy repeat of \cite{FLINthesis}.
Instead, we try to summarize Lin's construction and refer to the appropriate places in \cite{FLINthesis} for details.
We assume readers have a familiarity with Kronheimer--Mrowka's book \cite{KM} on monopole Floer homology.
As in \cite{KM}, there are three flavors $\widecheck{HM}_*(Y)$, $\widehat{HM}_*(Y)$, $\overline{HM}_*(Y)$ for a closed oriented 3-manifold $Y$, but we shall focus on $\widecheck{HM}_*(Y)$ as this is the version we shall need to use.
Also, just for simplicity, we impose $b_1(Y)=0$, which is satisfied for our situation.

\begin{remark}
We will need to work with Morse--Bott models for the monopole groups with $\mathbb{Z}$ coefficients. The construction in \cite{FLINthesis} is done with $\mathbb{F} = \mathbb{Z}/2 $ coefficients only, but this can be promoted over to $\mathbb{Z}$ by considering oriented geometric chains as in \S 6.1 in \cite{miller} and then assigning signs to counts in moduli spaces as in \S22.1 of \cite{KM}.
\end{remark}

\subsubsection{Stratified space and $\delta$-chain}

An ingredient of the whole construction is a model of the homology of a critical manifold $X$ of the Chern--Simons--Dirac functional in terms of geometric chains.

Let $N^d$ be a $d$-dimensional space stratified by manifolds (\cite[Definition 3.4.11]{FLINthesis}).
Namely, there are closed subsets
\[
N^d \supset N^{d-1} \supset \cdots N^0 \supset N^{-1}=\emptyset,
\]
where $N^d \neq N^{d-1}$ and each $N^e \setminus N^{e-1}$ is a smooth manifold of dimension $e$ (possibly disconnected or empty ). The closure $\Delta \subset N^e$ of a connected component of $N^{e} \setminus N^{e-1}$ is called an $e$-dimensional \textit{face}.
Such a stratified space $N$ is called a {\it $d$-dimensional abstract $\delta$-chain} if for every pair of faces $\Delta_1 \subset \Delta_2$ of dimensions $e$ and $e+2$ there exists exactly two faces containing $\Delta_1$ and which are contained in $\Delta_2$, together with some additional properties that we won't discuss (see \cite[Definition 4.1.1]{FLINthesis}).
A typical example of an abstract $\delta$-chain is given by a compactified moduli space of unparametrized trajectories (\cite[Example 4.1.2]{FLINthesis}).

Given a smooth manifold $X$ without boundary, a {\it $\delta$-chain} in $X$ is a pair $\sigma=(\Delta,f)$, where $\Delta$ is an abstract $\delta$-chain and $f : \Delta \to X$ is a continuous map satisfying a few additional properties (\cite[Definition 4.1.6]{FLINthesis}). If $\Delta$ is clear from context 
we just use $f$ or $f(\Delta)$ to denote $(\Delta,f)$.
When $X$ is a critical submanifold of the Chern--Simons--Dirac functional,
a typical example of a $\delta$-chain in $X$ is given by a compactified moduli space $\Delta$ of unparametrized trajectories together with the evaluation map $f=\ev_+$ at the time $+\infty$ (or the evaluation map $f=\ev_-$ at $-\infty$).
We shall use the fact that, given two $\delta$-chains  $\sigma_1$ and $\sigma_2$ in $X$, there is the notion of fiber product $\sigma_1 \times \sigma_2$ which is also a $\delta$-chain in $X$ (\cite[Lemma 4.1.8]{FLINthesis}; see also \cite{miller}, Proposition 6.1).

Lin gave a model of the homology of $X$ in terms of $\delta$-chains as follows.
Fix a countable collection $\mathcal{F}=\{\sigma_\alpha\}_\alpha$ of $\delta$-chains in $X$.
There is a notion of transversality between two $\delta$-chains (\cite[Definition 4.1.7]{FLINthesis}), and one can define a chain group $\tilde{C}_d^{\mathcal{F}}(X)$ over $\mathbb{F}=\Z/2$ generated by $\delta$-chains of dimension $d$ transverse to all $\sigma_\alpha$, up to isomorphism (see just after \cite[Definition 4.1.6]{FLINthesis}) after we quotient by the relation
\[
(\Delta,f) + (\Delta',f') \sim (\Delta \sqcup \Delta', f \sqcup f').
\]
The boundary operator is defined by 
\[
\tilde{\delta} : \tilde{C}_d^{\mathcal{F}}(X) \to \tilde{C}_{d-1}^{\mathcal{F}}(X) \quad ; \quad [\Delta,f] \mapsto \sum_{\Delta'}[\Delta',f|_{\Delta'}],
\]
where the sum is taken over all codimension-one faces $\Delta' \subset \Delta$.

To get a model of the homology of $X$, one needs to introduce the notion of {\it negligible} chain (\cite[Definition 4.1.11]{FLINthesis}), which means its image is contained in a lower dimensional $\delta$-chain, which corresponds to a degenerate simplex in the singular homology.
Let $(C_d^{\mathcal{F}}(X), \partial)$ be the quotient of $(\tilde{C}_d^{\mathcal{F}}(X), \tilde{\partial})$ by the subcomplex generated by negligible chains.
Then this chain complex $(C_d^{\mathcal{F}}(X), \partial)$ gives a model of the homology of $X$ (\cite[Proposition 4.1.13]{FLINthesis}).
If there is no risk of confusion, we simply write $(C_d(X), \partial)$ for $(C_d^{\mathcal{F}}(X), \partial)$.

\subsubsection{Floer chain group}

Next, we shall describe the Floer chain group.
Let $\mathfrak{s}$ be a spin$^c$ structure on $Y$.
We fix a metric $g$ on $Y$, and pick a perturbation $\mathfrak{q}_0$.
The perturbation $\mathfrak{q}_0$ is taken to be {\it admissible} (\cite[Definition 4.2.1]{FLINthesis}).
This is a condition to achieve the transversality of evaluation maps to critical manifolds from moduli spaces of flow lines between them.

Let $\mathcal{B}^\sigma(Y,\mathfrak{s})$ denote the blown-up configuration space \cite[Subsection 6.1]{KM}.
On $\mathcal{B}^\sigma(Y,\mathfrak{s})$, there is the blown-up gradient of the Chern--Simons--Dirac functional.
Let $\textsf{C} \subset \mathcal{B}^\sigma(Y,\mathfrak{s})$ denote the set of critical submanifolds of this gradient.
Just as in the Morse case, one can define the notion of {\it boundary-(un)stability} of a reducible critical point (\cite[Definition 3.3.9]{FLINthesis}).
Thus $\textsf{C}$ splits into a disjoint union
\[
\textsf{C} = \textsf{C}^o \cup \textsf{C}^s \cup \textsf{C}^u,
\]
where each set consists of irreducible, boundary-stable and boundary-unstable critical submanifolds, respectively.

Let $[\mathfrak{C}], [\mathfrak{C}']$ be critical submanifolds of the Chern--Simons--Dirac functional.
Then we can consider the compactified moduli space $\breve{M}_z^+([\mathfrak{C}], [\mathfrak{C}'])$ of unparametrized trajectries from $[\mathfrak{C}]$ to $[\mathfrak{C}']$ with respect to a relative path $z$ (\cite[Example 4.1.2]{FLINthesis}).
Then we can obtain a countable family of $\delta$-chains $\mathcal{F}$ in $\mathfrak{C}$ defined by
\[
(\breve{M}_z^+([\mathfrak{C}], [\mathfrak{C}']), \ev_-)
\]
where $[\mathfrak{C}']$ and $z$ run.
Now, we define
\[
C^o = \bigoplus_{[\mathfrak{C}] \in \textsf{C}^o} C_\ast^{\mathcal{F}}([\mathfrak{C}]),\quad
C^s = \bigoplus_{[\mathfrak{C}] \in \textsf{C}^s} C_\ast^{\mathcal{F}}([\mathfrak{C}]),\quad
C^u = \bigoplus_{[\mathfrak{C}] \in \textsf{C}^u} C_\ast^{\mathcal{F}}([\mathfrak{C}]).
\]
We also set
\[
\check{C} = C^o \oplus C^s,\quad
\hat{C} = C^o \oplus C^u,\quad
\bar{C} = C^s \oplus C^u.
\]


The differentials are defined as follows.
Let $\sigma=(\Delta,f)$ be a $\delta$-chain in a critical submanifold $[\mathfrak{C}]$.
Then for each moduli space $\breve{M}_z^+([\mathfrak{C}], [\mathfrak{C}'])$,
the fiber product
\[
\sigma \times \breve{M}_z^+([\mathfrak{C}], [\mathfrak{C}'])
\]
is an abstract $\delta$-chain and this gives a $\delta$-chain in $[\mathfrak{C}']$ by the evaluation map 
\[
\ev_+ : \sigma \times \breve{M}_z^+([\mathfrak{C}], [\mathfrak{C}']) \to [\mathfrak{C}'].
\]
It turns out the $\mathrm{ev}_+$ is transverse to every $ev_-$ with codomain $[\mathfrak{C}']$, hence this $\delta$-chain is $\mathcal{F}$-transverse.
Then we can define the operator 
\[
\partial^o_s : C^o_\ast \to C^s_\ast
\]
by
\[
\partial^o_s[\sigma]
= \sum_{[\mathfrak{C}'] \in \textsf{C}^s}[\ev_+ : \sigma \times \breve{M}_z^+([\mathfrak{C}], [\mathfrak{C}']) \to [\mathfrak{C}']].
\]
The other operators 
\[
\partial^o_o : C^o_\ast \to C^o_\ast,\quad
\partial^u_s : C^u_\ast \to C^s_\ast,\quad
\partial^u_o : C^u_\ast \to C^o_\ast
\]
are defined in a similar fashion (\cite[page 101]{FLINthesis}).
Also, when both $[\mathfrak{C}], [\mathfrak{C}']$ are reducible, the operators 
\[
\bar{\partial}^s_s : C^s_\ast \to C^s_\ast,\quad
\bar{\partial}^s_u : C^s_\ast \to C^u_\ast,\quad
\bar{\partial}^u_s : C^u_\ast \to C^s_\ast,\quad
\bar{\partial}^u_u : C^u_\ast \to C^u_\ast
\]
are also defined in a similar manner (\cite[page 102]{FLINthesis}).
Combining these operators exactly in a same way as in the usual monopole Floer homology \cite[Definition 2.4.4]{KM}, we can define operators
\[
\check{\partial} : \check{C}_\ast \to \check{C}_\ast,\quad
\hat{\partial} : \hat{C}_\ast \to \hat{C}_\ast,\quad
\bar{\partial} : \bar{C}_\ast \to \bar{C}_\ast
\]
of degree $-1$, and these make $\check{C}_\ast, \hat{C}_\ast, \bar{C}_\ast$ chain complexes (\cite[Proposition 4.2.4]{FLINthesis}).
We denote by
\[
\widecheck{HM}_\ast(Y,\mathfrak{s}),\quad
\widehat{HM}_\ast(Y,\mathfrak{s}),\quad
\overline{HM}_\ast(Y,\mathfrak{s})
\]
the homologies of these chain complexes.
When perturbations are taken to make all critical points non-degenerate, these groups coincide with the usual monopole Floer homologies defined in \cite{KM} (\cite[Lemma 4.2.6]{FLINthesis}).

When the spin-c structure $\mathfrak{s}$ is torsion, which is the case we are interested in, one can define an absolute $\mathbb{Q}$-grading $\grad(\mathfrak{C})$ for each component $\mathfrak{C}$ of $\textsf{C}$. 
For each $d$-dimensional $\delta$-chain $\sigma$ in $\mathfrak{C}$, set $\grad(\sigma)=\grad(\mathfrak{C})+d$. This makes $\widecheck{HM}_\ast(Y,\mathfrak{s})$ into a $\mathbb{Q}$-graded $\mathbb{F}$-module.

\subsection{Induced map by families of cobordisms} It is a fundamental property that a cobordism between two 3-manifolds induces a map between their monopole Floer homologies. This can be generalized to families of cobordisms and we briefly recall relevant results here, mostly following \cite{JLIN2022}. In what follows all manifolds considered are smooth, oriented, compact (possibly with boundary) and connected.

\begin{definition}
Let $Y_0, Y_1$ be oriented, closed 3-manifolds. A {\it family of cobordisms} $\widetilde{W}/\mathcal{Q}$ from $Y_0$ to $Y_1$ consists of the following data.
\begin{enumerate}
    \item A smooth fiber bundle $W\hookrightarrow \widetilde{W}\to \mathcal{Q}$ whose fiber $W$ is a smooth oriented 4-manifold with boundary and whose base $\mathcal{Q}$ is a smooth, oriented, closed manifold. 
    \item An orientation-preserving diffeomorphism 
    \[
    \partial W\cong \mathcal{Q}\times (\overline{Y}_0\sqcup Y_{1})
    \]
    that covers the identity map on $\mathcal{Q}$, called the boundary parametrization.
\end{enumerate}
We say $\widetilde{W}/\mathcal{Q}$ is a trivial family if the boundary parametrization can be extended to a trivialization of the bundle $\widetilde{W}\to \mathcal{Q}$. We ignore $\mathcal{Q}$ from our notation if $\mathcal{Q}$ is obvious from the context.  
\end{definition}

\begin{assumption}\label{assumption: family} 
To simplify our discussions here, we make the following assumptions.
\begin{itemize}
 \item $\mathcal{Q}\cong S^{1}$. 
\item The monodromy action of $\pi_{1}(\mathcal{Q})$ on $H_{*}(W;\mathbb{R})$ preserves the homological orientation (i.e. the orientation on $\Lambda^{\max}(H^{0}(W;\mathbb{R})\oplus H^{1}(W;\mathbb{R})\oplus H^{+}(W;\mathbb{R}))$). Furthermore, if this action is nontrivial, then we assume $b^{+}(W)\geq 3$.
\end{itemize}
\end{assumption}
A family spin-c structure $\mathfrak{s}_{\widetilde{W}}$ is a lift of the structure group of the vertical tangent bundle 
\[
T^{v}\widetilde{W}:=\ker(T\widetilde{W}\to T\mathcal{Q}).
\]
from $SO(4)$ to  $\operatorname{Spin}^{c}(4)$.
 Let $\mathfrak{s}_{Y_{i}}$ be the restriction of $\mathfrak{s}_{\widetilde{W}}$ on $Y_{i}$. For various $X\subset \widetilde{W}$, we use $\mathfrak{s}_{X}$ to denote the restriction of $\mathfrak{s}_{\widetilde{W}}$ to $X$. Then as explained in \cite[Proposition 4.5]{JLIN2022}, we have induced maps 
\begin{equation}\label{eq: cobordism maps}
\begin{split}
\widecheck{HM}(\widetilde{W},\mathfrak{s}_{\widetilde{W}}): \widecheck{HM}_{*}(Y_0,\mathfrak{s}_{Y_0})\to \widecheck{HM}_{*}(Y_1,\mathfrak{s}_{Y_1}),\\
\widehat{HM}(\widetilde{W},\mathfrak{s}_{\widetilde{W}}): \widehat{HM}_{*}(Y_0,\mathfrak{s}_{Y_0})\to \widehat{HM}_{*}(Y_1,\mathfrak{s}_{Y_1}),\\
\overline{HM}(\widetilde{W},\mathfrak{s}_{\widetilde{W}}): \overline{HM}_{*}(Y_0,\mathfrak{s}_{Y_0})\to \overline{HM}_{*}(Y_1,\mathfrak{s}_{Y_1}).
\end{split}
\end{equation}
by counting solutions on all fibers. 

If we further assume $b^{+}(W)\geq 1$, then there is a well-defined ``mixed map''
\[
\overrightarrow{HM}(\widetilde{W},\mathfrak{s}_{\widetilde{W}})_{\xi}: \widehat{HM}_{*}(Y_0,\mathfrak{s}_{Y_0})\to \widecheck{HM}_{*}(Y_1,\mathfrak{s}_{Y_1})
\]
Unlike those maps in (\ref{eq: cobordism maps}), the mixed map depends on the choice of a chamber $\xi$. Concretely, $\xi$ is a homotopy class of a section of a bundle $V^{+}(W)\hookrightarrow V^{+}(\widetilde{W})\to \mathcal{Q}$, whose fiber over $b\in \mathcal{Q}$ is the positive cone 
\[
V^{+}(W_{b}):=\{\alpha\in H^{2}(W_{b};\mathbb{R})\mid \alpha\cdot \alpha>0\}.
\]
By Assumption \ref{assumption: family}, one can pick a canonical chamber $\xi_{c}$ defined as follows: 
\begin{itemize}
    \item If $b^{+}(W)=1 $ or $2$, then the monodromy action is trivial on $H_*(M)$. As a result, we have a canonical trivialization $V^{+}(\widetilde{W})\cong V^{+}(W)\times \mathcal{Q}$. We let $\xi_{c}$ be represented by a constant section. When $b^{+}(W)=1$, we need to fix an orientation on $H^{+}(W;\mathbb{R})$ and let $\xi_{c}$ take value in the positive component of $V^{+}(W)$.
    \item If $b^{+}(W)\geq 3$, $V^{+}(\widetilde{W})$ has a unique section up to homotopy because $V^{+}(W)\simeq S^{b^{+}(W)-1}$.
\end{itemize}
For simplicity in notation, we will just write $\overrightarrow{HM}(\widetilde{W},\mathfrak{s}_{\widetilde{W}})_{\xi_{c}}$ for $\overrightarrow{HM}(\widetilde{W},\mathfrak{s}_{\widetilde{W}})$ from now on. We next summarize some important properties of family Seiberg-Witten invariant and family cobordism maps. See \cite[Section 4]{JLIN2022} for details.
\begin{enumerate}
    \item Suppose $\widetilde{W}/\mathcal{Q}$ is a trivial family. Then for any $\mathfrak{s}_{\widetilde{W}}$, the maps \[\widecheck{HM}(\widetilde{W},\mathfrak{s}_{\widetilde{W}}),\quad \overrightarrow{HM}(\widehat{W},\mathfrak{s}_{\widetilde{W}}),\quad \overrightarrow{HM}(\overline{W},\mathfrak{s}_{\widetilde{W}}),\quad \overrightarrow{HM}(\widetilde{W},\mathfrak{s}_{\widetilde{W}})\]
    are all trivial.
    \item Let $\widetilde{W}_{01}/\mathcal{Q}$  be a family cobordism from $Y_0$ to $Y_1$ and let $\widetilde{W}_{12}/\mathcal{Q}$ be a family cobordism from $Y_1$ to $Y_2$. Given a family spin-c structure $\mathfrak{s}_{\widetilde{W}_{02}}$ on $\widetilde{W}_{02}$, the cobordism induces maps that satisfy various \textit{gluing formulas} . 
    \begin{enumerate}
        \item Suppose $\widetilde{W}_{01}$ is a trivial family with $b^{+}(W_{01})>0$. Then one has 
        \begin{equation}\label{eq: hm arrow gluing}
        \overrightarrow{HM}(\widetilde{W}_{02},\mathfrak{s}_{\widetilde{W}_{02}})=  \widecheck{HM}(\widetilde{W}_{12},\mathfrak{s}_{\widetilde{W}_{12}})\circ \overrightarrow{HM}(W_{01},\mathfrak{s}_{W_{01}})
         \end{equation}

         \item Suppose $\widetilde{W}_{12}$ is a trivial family with $b^{+}(W_{12})>0$. Then one has
\begin{equation}\label{eq: hmhat gluing}
        \overrightarrow{HM}(\widetilde{W}_{02},\mathfrak{s}_{\widetilde{W}_{02}})=  \overrightarrow{HM}(W_{12},\mathfrak{s}_{W_{12}})\circ \widehat{HM}(\widetilde{W}_{01},\mathfrak{s}_{\widetilde{W}_{01}}).
        \end{equation}
         
        \item Suppose $\widetilde{W}_{12}$ is a trivial family. Then one has 
        \begin{equation}\label{eq: hmcheck gluing}
        \widecheck{HM}(\widetilde{W}_{02},\mathfrak{s}_{\widetilde{W}_{02}})=  \widecheck{HM}(W_{12},\mathfrak{s}_{W_{12}})\circ \widecheck{HM}(\widetilde{W}_{01},\mathfrak{s}_{\widetilde{W}_{01}}).
        \end{equation}
        
    \end{enumerate}
    \item Given a family of closed 4-manifolds $X\hookrightarrow \widetilde{X}\to \mathcal{Q}$ and a family  spin-c structure $\mathfrak{s}_{\widetilde{X}}$ whose restriction  $\mathfrak{s}$ to $X$ satisfies $d(\mathfrak{s})=-1$. Here
    \[
 d(\mathfrak{s}):=\frac{c^{2}_{1}(\mathfrak{s})-2\chi(X)-3\sigma(X)}{4}.
    \] 
    Assume the monodromy acts trivially on $H_{*}(X)$ and assume $b^{+}(X)>1$. By removing tubular neighorhoods of two sections, one obtains a family of cobordisms $X^{\circ}/\mathcal{Q}$ from $S^3$ to $S^3$. Then we have the following formula 
   \begin{equation}\label{eq: FSW formula}
    \overrightarrow{HM}(\widetilde{X}^{\circ},\mathfrak{s}_{\widetilde{X}^{\circ}})(\hat{1})=\operatorname{FSW}(\widetilde{X},\mathfrak{s}_{\widetilde{X}})\cdot \check{1}.
     \end{equation}
      Here $\hat{1}\in \widehat{HM}_{-1}(S^3)$ and $\check{1}\in \widecheck{HM}_0(S^{3})$ denote the canonical generators. And $\operatorname{FSW}(\widetilde{X},\mathfrak{s}_{\widetilde{X}})$ denotes the family Seiberg-Witten invariant \emph{for the constant chamber} (see \cite[Section 2]{JLIN2022}).
\item Consider a smooth 4-manifold $M$ with boundary $\partial M = Y$ and $b^+(M)>0$. Consider $f\in \MCG(M)$ that fix the homological orientation on $M$. We further assume $b^{+}(M)>2$ if $f$ acts nontrivially on $H_{*}(M)$. Let $\mathfrak{s}$ be a spin-c structure such that $f^*(\mathfrak{s})\cong \mathfrak{s}$. Then we can define the family Seiberg-Witten invariant \emph{for diffeomorphisms} as follows. By taking the mapping torus $M\hookrightarrow \widetilde{M}\to \mathcal{Q}$ and removing small balls around a constant section near the boundary,  one obtains a family of cobordisms $M^{\circ}\hookrightarrow\widetilde{M}^{\circ}\to \mathcal{Q}$ which satisfies Assumption (\ref{assumption: family}). Let $\mathfrak{s}_{\widetilde{M}^{\circ}}$ be any family spin-c structure on $\widetilde{M}^{\circ}$ that restricts to $\mathfrak{s}|_{M^{\circ}}$ on fibers. Then we define 
\begin{equation}\label{eq: FSW for diff}
\FSW(M,\mathfrak{s},f):=\overrightarrow{HM}(\widetilde{M}^{\circ},\mathfrak{s}_{\widetilde{M}^{\circ}})(\hat{1})\in\widecheck{HM}(Y,\mathfrak{s}|_{Y})    
\end{equation}
This gives an invariant of the smooth isotopy class of $f$ and the isomorphism class of $\mathfrak{s}$, and is independent of the choice of $\mathfrak{s}_{\widetilde{M}^{\circ}}$. The invariant is \textit{additive}: given $f,g \in \MCG(M)$ both satisfying the above assumptions then
\begin{equation}\label{eq: FSW additive}
\FSW(M,\mathfrak{s},f\circ g)=\FSW(M,\mathfrak{s},f)+\FSW(M,\mathfrak{s},g).    
\end{equation}
\item Consider a smooth, closed 4-manifold $X$ with $b^+(X)>1$. Let $f\in \MCG(X)$ be a diffeomorphism that acts trivially on $H_*(X)$ and let $\mathfrak{s}$ be spin-c structure on $X$ such that $f^*(\mathfrak{s})=\mathfrak{s}$ and $d(\mathfrak{s})=-1$. Then one can similarly define 
\[\FSW(X,\mathfrak{s},f):=\FSW(\widetilde{X},\mathfrak{s}_{\widetilde{X}}).\]
Here $X\hookrightarrow \widetilde{X}\to S^1$ is the mapping torus of $f$ and $\mathfrak{s}_{\widetilde{X}}$ is any family spin-c structure that restricts to $\mathfrak{s}$ on fibers. Again, 
$\FSW(X,\mathfrak{s},f)$ is well-defined and additive. Furthermore, if $\FSW(X,\mathfrak{s},f)\neq 0$, then for any $k\neq 0$,  then $f^k$ is not isotopic a diffeomorphism supported on $D^4$. (See \cite[Proof of Theorem 1.7]{KMPW} for the case $k=1$. The general case follows from additivity.)
\end{enumerate}

\section{Proofs of Theorem \ref{thm: main}, \ref{thm: factorization}}\label{section: proof main}
Let $Y=S(N)$ be the circle bundle for an orbifold line bundle $N\to C$ with $\mathrm{deg}N <0$ over a 2-orbifold $C$. Through this section we assume $Y$ is Floer simple (see Definition  \ref{defi: Floer simple}), and in particular $C$ has genus zero. Let $\mathfrak{s}_c$ denote the canonical spin-c structure on $Y$, which by assumption is self-conjugate. Then we have $K_{C}=N^{-m}$ for $m=\frac{\chi(C)}{\operatorname{deg}(N)} \in \mathbb{Z}$. 

Note that $m > 0$. Indeed, otherwise by Proposition \ref{proposition:MOYirreducible} there are no irreducible critical manifolds, and by Corollary \ref{cor: Dirac no kernel} the reducible critical manifolds in the blowup are Morse--Bott; this would give $HM^{\redu}(Y, \mathfrak{s}_c ) = 0$, contradicting the Floer simple condition.

We fix a metric $g_{C}$ on $C$ such that Theorem \ref{theorem:eigenspaces} applies with $\Lambda=\frac{m}{2}$ and consider the induced metric on $Y$. Let $(B,0)$ be a reducible critical point for the canonical spin-c structure $\mathfrak{s}_{c}$. Then by Corollary \ref{cor: Dirac no kernel}, the Dirac operator $D_{B}$ has a trivial kernel and hence $(B,0)$ is a non-degenerate critical point. Let 
\[\cdots<\lambda_{-2}<\lambda_{-1}<0<\lambda_{0}<\lambda_{1}<\cdots\] be the eigenvalues of $D_{B}$, with eigenspace $V_{\lambda_{i}}$. Then the critical manifold $\mathfrak{C}_{i}:=\mathbb{P}(V_{\lambda_{i}})$ is Morse--Bott and we have decompositions 
\[
\textsf{C}^{s}=\bigsqcup_{i\geq 0}\mathfrak{C}_{i},\quad \textsf{C}^{u}=\bigsqcup_{i<0}\mathfrak{C}_{i}.
\]
 Let $h=h(Y,\mathfrak{s}_{c})$ be the Fr\o yshov invariant. Then for any $i\geq 0$, we have $\grad(\mathfrak{C}_{i})\in -2h+2\mathbb{Z}$ and
\begin{equation}\label{eq: relative grading between reducible}
\grad(\mathfrak{C}_{i+1})-\grad(\mathfrak{C}_{i})=2\operatorname{dim}_{\mathbb{C}}V_{\lambda_{i}}.    
\end{equation}
(See \cite[page 57]{FLINthesis}.)
 Now we consider the irreducible critical manifold $\textsf{C}^{o}$. By Proposition \ref{proposition:MOYirreducible}, $\textsf{C}^{o}$ is Morse-Bott and has a decomposition
\begin{equation}\label{eq: Co}
    \textsf{C}^{o}=\bigsqcup_{ 0\leq l < \frac{m}{2}}\Big( \mathfrak{C}^{+} (N^{-l} ) \sqcup \mathfrak{C}^{-}(N^{-l}) \Big).
\end{equation}
Each $\mathfrak{C}^{\pm} (N^{-l} )$ is diffeomorphic to $\mathbb{P}(H^0(C;N^{-l}))$. Furthermore, since the moduli space of unbased, parametrized flowlines from $\mathfrak{C}^{+} (N^{-l})$ to $\mathfrak{C}_{i}$ (with $i\geq 0$) is odd dimensional (see Corollary \ref{cor:topologymoduli}) and the cokernel is a complex vector space, the relative grading between these two critical manifolds is an odd integer, i.e. we have 
\[
\grad(\mathfrak{C}^{+} (N^{-l}))=\grad(\mathfrak{C}^{-} (N^{-l}))\in -2h+1+2\mathbb{Z}.
\]

\begin{lemma}\label{lem: H0 vanishing} $H^0(C;N^{-l})=0$ for any $1\leq l\leq \frac{m}{2}$.
\end{lemma}
\begin{proof}
The case $l=\frac{m}{2}$ has been proved in Lemma \ref{lem: vanishing homology 1} so we assume $1\leq l<\frac{m}{2}$. 
By \cite[Chapter 2, Theorem 3.17]{FLINthesis}, we may add a perturbation $\mathfrak{q}$ to the Chern--Simons--Dirac functional such that the critical manifolds $\textsf{C}^{o}$ and $\textsf{C}^{s}$ are unchanged, while all moduli spaces of flowlines are transverse. Consider the Floer chain complex $\check{C}$ with $\mathbb{Q}$-coefficient. The perturbed Chern--Simons--Dirac functional gives an action filtration $\check{C}$, inducing a spectral sequence. The absolute $\mathbb{Q}$-grading on $\check{C}$ induces a $\mathbb{Q}$-grading on each page of this spectral sequence. As a graded vector space, we have an isomorphism 
\[
E^{1}\cong \bigoplus_{\mathfrak{C}}  H_{*}(\mathfrak{C})[\grad(\mathfrak{C})].
\]
Here the direct sum is taken over all components of $\textsf{C}^{s}\cup \textsf{C}^{o}$. And $H_{*}(\mathfrak{C})[\grad(\mathfrak{C})]$ denotes the grading shift of $H_{*}(\mathfrak{C})$ up by $\grad(\mathfrak{C})$.
For each $k\in -2h+\mathbb{Z}$, we use $E^{n}_{k}$ to denote the subspace of the $E^n$-page in homologial grading $-2h+k$.
Then we can draw the following conclusions: 
\begin{enumerate}[(I)]
    \item Since $\grad(\mathfrak{C})+2h$ is even when $\mathfrak{C}\subset \textsf{C}^{s}$ and is odd when $\mathfrak{C}\subset \textsf{C}^{o}$, we have 
    \[
\bigoplus_{k \text{ odd}}E^{1}_{k}\cong \bigoplus_{\mathfrak{C}\subset \textsf{C}^{o}}  H_{*}(\mathfrak{C})[\grad(\mathfrak{C})]
    \]
    and
    \[
\bigoplus_{k \text{ even}}E^{1}_{k}\cong \bigoplus_{\mathfrak{C}\subset \textsf{C}^{s}}  H_{*}(\mathfrak{C})[\grad(\mathfrak{C})].
    \]
     \item For any $n\geq 1$ and any even $k$, the differential $d^{n}: E^{n}_{k}\to E^{n}_{k-1}$ is trivial. This is because  there is no flowline from $\textsf{C}^{s}$ to $\textsf{C}^{o}$. 
    \item By (\ref{eq: relative grading between reducible}), $\operatorname{dim}(E^{1}_{k})\leq 1$ for any even $k$.
    \item For any odd $k$, $\operatorname{dim}(E^{1}_{k})$ is even. This follows from the charge-conjugation symmetry between $\mathfrak{C}^{+}(N^{-l})$ and  $\mathfrak{C}^{-}(N^{-l})$.
    \item By (I), (II) and (III), we have 
    \[\operatorname{dim}(E^{1}_{k})-\operatorname{dim}(E^{\infty}_{k})\leq 1.\] 
    \item We have the isomorphism
\begin{equation}\label{eq: E-infinity-odd}
\bigoplus_{k \text{ odd }} E^{\infty}_{k}\cong \bigoplus_{k \text{ odd }}\widecheck{HM}_{-2h+k}(Y,\mathfrak{s})\subseteq HM^{\redu}(Y,\mathfrak{s})\cong \mathbb{Q}.
\end{equation}
\item By (IV) and (V), we see that 
\[
\operatorname{dim}(\bigoplus_{k \text{ odd}}E^{1}_{k})\leq 2.
\]
In other words, we have 
\[
\operatorname{dim}(H_{*}(\bigsqcup_{0\leq l<\frac{m}{2}}(\mathbb{P}(H^0(C;N^{-l}))\sqcup \mathbb{P}(H^0(C;N^{-l}))))\leq 2.
\]
Since $H^0(C;N^{0})\cong \mathbb{C}$, we must have $H^0(C;N^{-l})=0$ for any $1\leq l< \frac{m}{2}$.
\end{enumerate}
\end{proof}

\begin{lemma}\label{lem: lambda_1}
$\lambda_{0}=\frac{m}{2} \in \frac{1}{2} \mathbb{Z}$.    
\end{lemma}

\begin{proof}
Recall that we chose the metric $g_C$ on $C$ such that Theorem \ref{theorem:eigenspaces} applies with $\Lambda = \frac{m}{2}$. Thus, since $\mathcal{U}_{m/2} = \mathbb{C}^2$, Lemma \ref{proposition:selfconj} gives the formula
\[
\lambda_{0}=\min\{\delta_0 + n\mid n \in \mathbb{Z}_{\geq 0} \text{ 
 and  } H^0 (C , N^{n-\lfloor m/2 \rfloor})\neq 0\}.
 \]
 From this and Lemma \ref{lem: H0 vanishing} we obtain that $\lambda_0 = \delta_0 + \lfloor m / 2 \rfloor = m /2$.
\end{proof}


By Lemma \ref{lem: H0 vanishing} and (\ref{eq: Co}), we have $\textsf{C}^{o}$ consists of two points: the canonical monopole $[c]$ and its conjugate $[\bar{c}]$. The boundary stable critical manifold with lowest grading is given by \[\mathfrak{C}_{0}=\mathbb{P}(\mathcal{U}_{\frac{m}{2}})=\mathbb{P}(\mathbb{C}\psi^{+}\oplus \mathbb{C}\psi^{-})\cong \mathbb{P}^{1}.\] 
Here $\psi^{+}$ (resp. $\psi^{-}$) is any nonzero vector in $\mathcal{U}^{+}_{\frac{m}{2}}$ (resp. $\mathcal{U}^{-}_{\frac{m}{2}}$.)

 By Corollary \ref{corollary:flowscontact}, the moduli space of unparametrized flowlines from $[c]$ to $\mathfrak{C}_{i}$ is empty when $i>0$ and consists of a single flowline $\gamma$ when $i=0$, which satisfies
\[\lim_{t\to +\infty}\gamma(t)=[\psi^{+}]\in \mathfrak{C}_{0}.\] Taking its conjugate, we obtain the unique flowline $\bar{\gamma}$ from $[\bar{c}]$ to $\mathfrak{C}_{0}$ that limits to $[\psi^{-}]\in \mathfrak{C}_{0}$. See Figure \ref{fig:flows} for a depiction of these.

\begin{figure}[h!]

    \centering
 \includegraphics[scale=0.15]{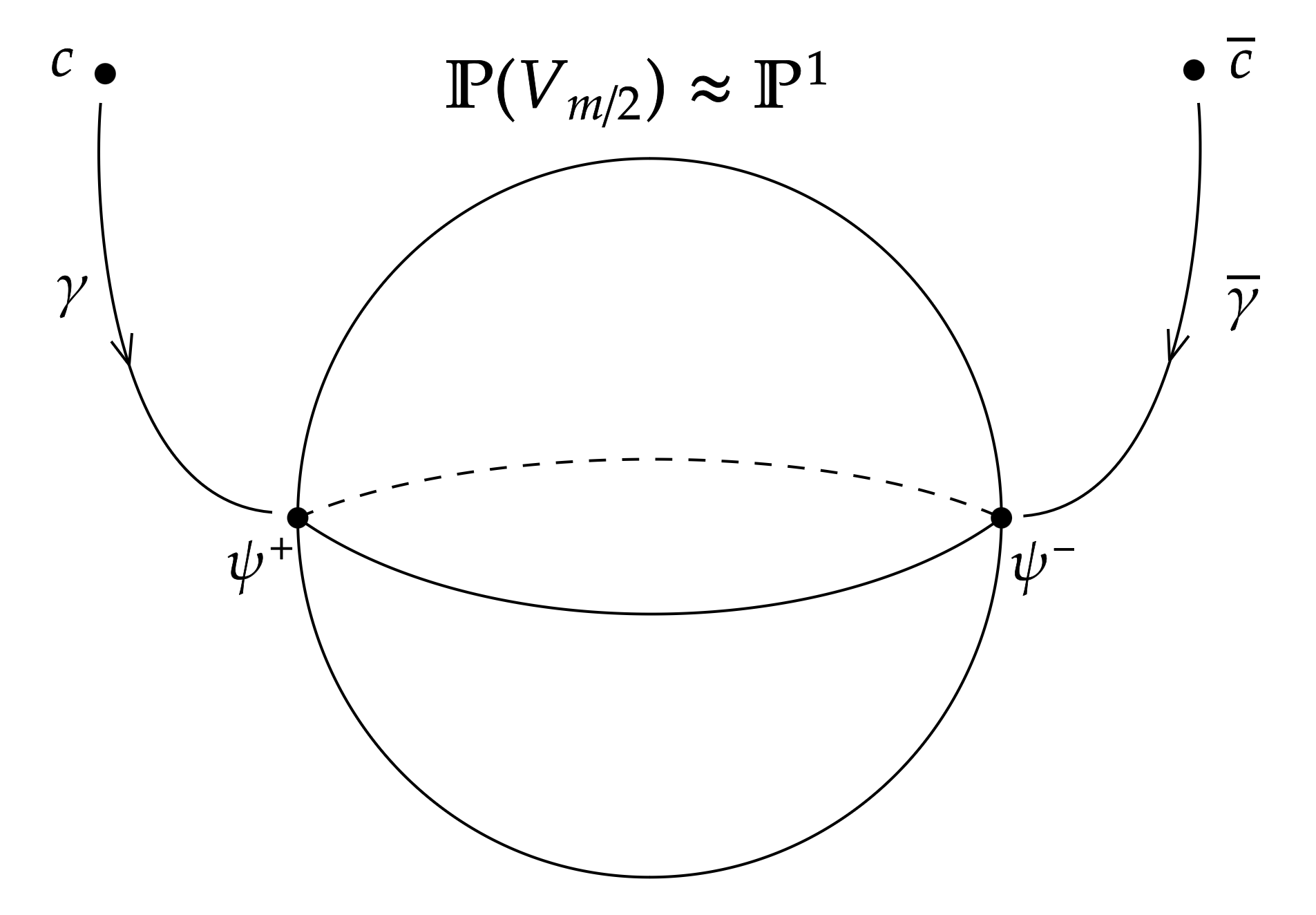}
\caption{The irreducible critical points $c$, $\overline{c}$ and the bottom reducible critical manifold $\mathbb{P}(V_{m/2})$, with the two flowlines $\gamma$, $\overline{\gamma}$ }
  \label{fig:flows}
\end{figure}

\begin{lemma}\label{lem: flowline transverse}
The flowlines $\gamma$ and $\bar{\gamma}$ are transversely cut-out.    
\end{lemma}
\begin{proof}
By conjugation symmetry, it suffices to consider $\gamma$ only. 
Let $H^{0}_{sw}$ and $H^{1}_{sw}$ be the kernel and cokernel of the Seiberg-Witten equations linearized at $[\gamma]$. Since $\gamma$ is an isolated flowline, $H^{0}_{sw} = \mathbb{R}$ by Corollary \ref{cor: Kuranishi}. Thus 
\[
\grad([c])-\grad(\mathfrak{C}_{0})=1-\operatorname{dim}_{\mathbb{R}}H^{1}_{sw}
\]
Suppose $H^{1}_{sw}\neq 0$. Then $\grad([c])\leq \grad(\mathfrak{C}_{0})$. As a result, the spectral sequence in the proof of Lemma \ref{lem: H0 vanishing} collapses on $E^1$-page and we have 
\[\bigoplus_{k \text{ odd }} E^{\infty}_{k}\cong \bigoplus_{k \text{ odd }} E^{1}_{k}\cong \mathbb{Q}^{2}.\]
This is a contradiction with (\ref{eq: E-infinity-odd}). 
\end{proof}

\begin{corollary}\label{cor: moduli spaces regular} Given any two critical manifolds $\mathfrak{C}$ and $\mathfrak{C}'$, the moduli space of unparametrized flowlines from $\mathfrak{C}$ to  $\mathfrak{C}'$ is transversely cut-out. 
\end{corollary}
\begin{proof}This moduli space is non-empty only in the following two cases:
\begin{itemize}
    \item $\mathfrak{C}=[c]$ or $[\bar{c}]$ and $\mathfrak{C}'=\mathfrak{C}_{0}$. This case follows from Lemma \ref{lem: flowline transverse}.
    \item $\mathfrak{C}=\mathfrak{C}_{i}$ and $\mathfrak{C}'=\mathfrak{C}_{j}$ for some $j<i$. In this case, the whole moduli space projects to a single critical point $(B,0)$ downstairs. The transversality follows from the argument \cite[Page 58]{FLINthesis}. 
\end{itemize}
\end{proof}

By Corollary \ref{cor: moduli spaces regular}, we can define all three flavors of monopole Floer homology using the Chern--Simons--Dirac functional (\ref{eq: CSD}), without any further perturbations.\footnote{To define the Chern-Simons-Dirac functional, one must choose an orthogonal connection on $TY$. It is customary to choose the Levi-Civita connection here, but instead we used the adiabatic connection in (\ref{eq: CSD}), which is the more natural choice. The Dirac operators associated to each of the two connections differ by a constant (Lemma 5.2.1 \cite{MOY}); hence the Chern--Simons--Dirac functional of the adiabatic connection is a \textit{tame} perturbation (in the sense of \cite{FLINthesis}) of the Chern--Simons--Dirac functional of the Levi-Civita connection. By Corollary \ref{cor: moduli spaces regular}, this is an \textit{admissible} perturbation in the sense of \cite{FLINthesis}.} Furthermore, $C^{o}$ is freely generated by $[c]$, $[\bar{c}]$ and we have $\check{\partial}([c])= \pm [\psi^{+}]$ and $\check{\partial}([\bar{c}])= \pm [\psi^{-}]$. After possibly replacing $[ \psi^{+} ]$ by $- [\psi^{+}]$ (and similarly for $[\psi^{-}]$) we may assume
\begin{equation}\label{eq: partial c}
\check{\partial}([c])=[\psi^{+}],\quad \check{\partial}([\bar{c}])=[\psi^{-}].
\end{equation}
Using (\ref{eq: partial c}), it is straightforward to compute $\widecheck{HM}(Y,\mathfrak{s}_{c})$ as a $\mathbb{Q}$-graded $\mathbb{Z}[U]$-module. For $d\in \mathbb{Q}$ and $n\in \mathbb{N}$, we use the notation 
\[
T^{+}_{d}=\frac{\mathbb{Z}[U,U^{-1}]}{U\cdot \mathbb{Z}[U]},\quad T^{+}_{d}(n)=\frac{\mathbb{Z}[U^{-n+1},U^{-n+2},\cdots]}{U\cdot \mathbb{Z}[U]}\]
with $\grad(U^{k})=d-2k$.

\begin{proposition}\label{prop: HM-check Y}
$\widecheck{HM}(Y,\mathfrak{s}_{c})\cong T^{+}_{-2h}\oplus T^{+}_{-2h-1}(1)$. Furthermore, a generator $\widecheck{HM}_{-2h}(Y,\mathfrak{s}_{c})\cong \mathbb{Z}$ (denoted by $\check{1}_{Y}$ henceforth) is represented by $(\mathfrak{C}_{0},\operatorname{id})$, and a generator of \[\widecheck{HM}_{-2h-1}(Y,\mathfrak{s}_{c})=HM^{\redu}(Y,\mathfrak{s}_{c})\cong \mathbb{Z}\] is represented by $[c]-[\bar{c}]+\ell$, where $\ell: [0,1]\to \mathfrak{C}_{0}$ is any path from $\psi^{+}$ to $\psi^{-}$. 
\end{proposition}

Let $Z=[0,1]\times Y$, the trivial cylinder cobordism. We consider the family of cylinder cobordisms $\widetilde{Z}$ from $Y$ to itself (which we call the \textit{twisted family of cylinders}), given by the trivial bundle $S^{1}\times Z$ with the \textit{non-trivial boundary parametrisation} 
\[
\partial\widetilde{Z}\xrightarrow{\cong} S^{1}\times (Y\sqcup Y)
\]
defined by 
\begin{equation}\label{eq: boundary parametrization}
(e^{i\theta},0,y)\mapsto (e^{i\theta},y),\quad (e^{i\theta},1,y)\mapsto (e^{i\theta}, e^{i\theta}\cdot y).    
\end{equation}
Here $\cdot$ denotes the circle action. Let $\widetilde{\mathfrak{s}}_{c}$ be the pull back of $\mathfrak{s}_{c}$ to $\widetilde{Z}$.

\begin{proposition}\label{prop: twisted cobordism W}
The cobordism induced map 
\[
\widecheck{HM}(\widetilde{Z},\widetilde{\mathfrak{s}}_{c}): \widecheck{HM}_{-2h-1}(Y,\mathfrak{s}_{c})\to \widecheck{HM}_{-2h}(Y,\mathfrak{s}_{c}). 
\]
sends $[c]-[\bar{c}]+\ell$ to $m \cdot \check{1}_{Y}$. 
\end{proposition}
\begin{proof}
The cobordism map is induced by a chain map $\check{m}: \check{C}_{-2h-1}\to \check{C}_{-2h}$. We care about the elements
\[\check{m}([c]),\ \check{m}([c]),\ \check{m}(\ell).\] By constructions in \cite[Page 113]{FLINthesis}, we have the following explicit descriptions
\begin{enumerate}
    \item $\check{m}([c])$ is given by the $\delta$-chain 
    \[
    \operatorname{ev}_{+}: \mathcal{M}^{+}([c], \widetilde{Z},\mathfrak{C}_{0})\to \mathfrak{C}_{0}
    \]
    Here $\mathcal{M}([c], \widetilde{Z},\mathfrak{C}_{0})$ is  moduli space of solutions on $\widetilde{Z}$ that converge to $[c]$ and $\mathfrak{C}_{0}$ on the two ends. We have 
    \[
    \mathcal{M}([c], \widetilde{Z},\mathfrak{C}_{0})=\mathbb{R}\times S^1
    \]
    given by translations of $\gamma$ on each fiber of $\widetilde{Z}$. And its compactification $\mathcal{M}^{+}([c], \widetilde{Z},\bp)$ is just $[0,1]\times S^{1}$. The map $\operatorname{ev}_{+}$ is the constant map with value $\psi^{+}$. So $\check{m}([c])$ is a negligible $\delta$-chain. We see that $\check{m}([c])=0$. 
    
    \item Similarly, we have $\check{m}([\bar{c}])=0$. 
    \item We pick $\ell$ to be the geodesic from $\psi^{+}$ to $\psi^{-}$.  To compute  $\check{m}(\ell)$, we consider the moduli space 
    \[
    \mathcal{M}^{+}(\mathfrak{C}_0, \widetilde{Z},\mathfrak{C}_0)
    \]
    of solutions converging to $\mathfrak{C}_0$ on both ends. Such solutions have zero energy so must be constant solutions, so we see that 
    \[
    \mathcal{M}^{+}(\mathfrak{C}_0, \widetilde{Z},\mathfrak{C}_0)\cong S^{1}\times \mathfrak{C}_0.
    \]
    By (\ref{eq: boundary parametrization}),
    the evaluation maps are 
    \[
    \operatorname{ev}_{-} :S^1\times \mathfrak{C}_0 \to \mathfrak{C}_0, (e^{i\theta},[\psi])\mapsto [\psi]
    \]
and 
\[
    \operatorname{ev}_{+} :S^1\times \mathfrak{C}_0 \to \mathfrak{C}_0,\quad (e^{i\theta},[\psi])\mapsto  e^{i\theta}\cdot [\psi].
    \]
    Here $\cdot$ denotes the induced circle action on $\mathfrak{C}_{0}=\mathbb{P}(V_{\frac{m}{2}})$. By Proposition \ref{proposition:weights}, we have 
    \[
    e^{i\theta}\cdot [x\psi^{+}+y\psi^{-}]=[e^{im\theta}x\psi^{+}+y\psi^{-}].
    \]
In other words $e^{i\theta}\in S^{1}$ acts as a rotation on $\mathfrak{C}_{0}$ by the angle $m\theta$, with $\psi^{\pm}$ as fixed points. By definition $\check{m}(\ell)$ is given by the $\delta$-chain 
    \[
    \operatorname{ev}_{+}: \operatorname{ev}^{-1}_{-}(\ell)\to \mathfrak{C}_0.
    \]
Note that $\operatorname{ev}^{-1}_{-}(\ell)\cong S^{1}\times [0,1]$ and the induced map 
\begin{equation}\label{eq: ev_{+}}
\operatorname{ev}_{*}:H_{2}(S^{1}\times [0,1];S^{1}\times \{0,1\})\to H_2(\mathfrak{C}_0,\{\psi^{+},\psi^{-}\})
    \end{equation}
is just a multiplication by $m$. So $\check{m}(\ell)$ is homologous to $m\check{1}_{Y}$, see Figure \ref{fig:rotation}. This finishes the proof.
\end{enumerate}
\end{proof}
\begin{figure}[h!]

    \centering
 \includegraphics[scale=0.15]{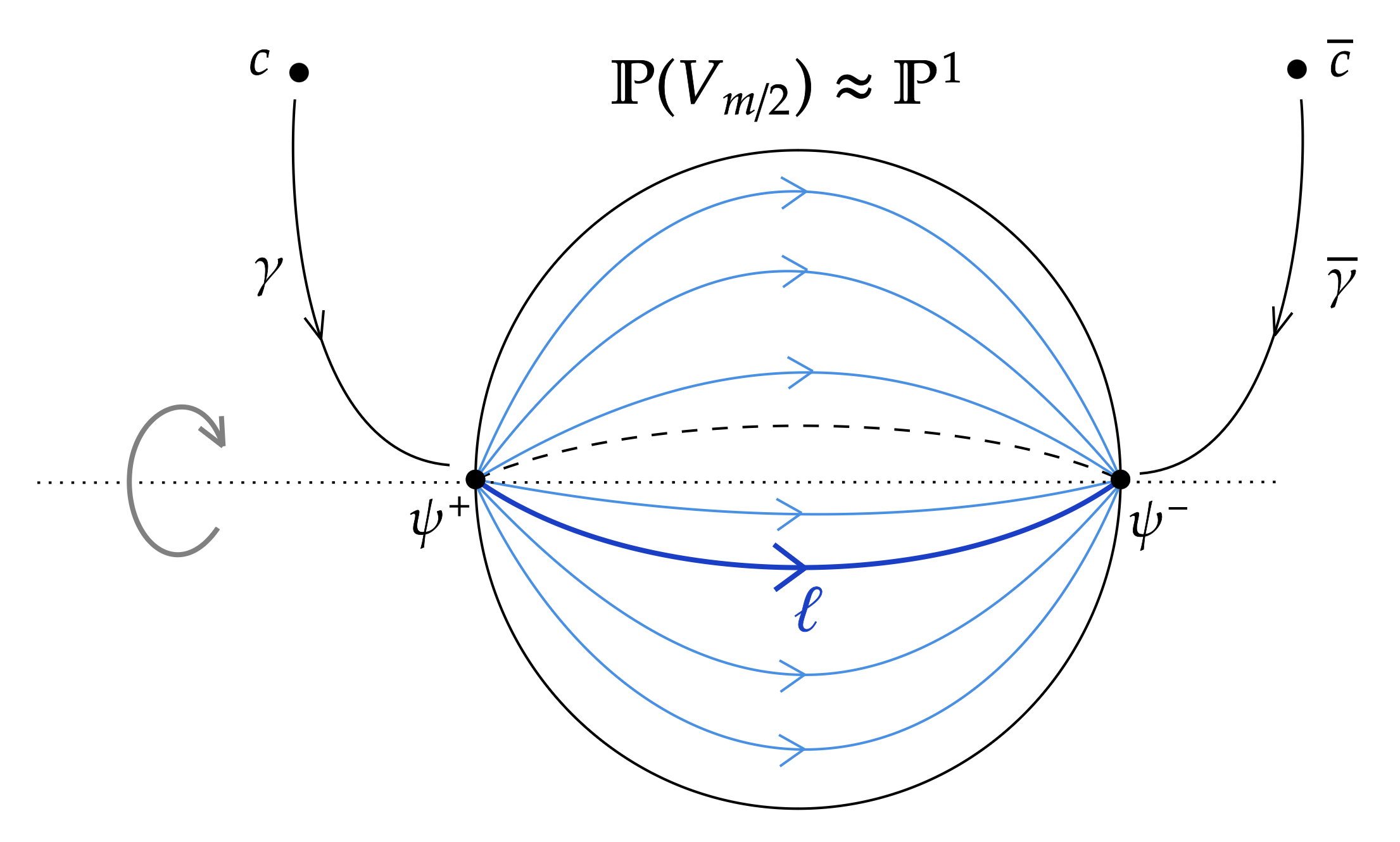}
\caption{The effect of the cobordism map induced by the Dehn twist on the chain $\ell$ is a rotation, tracing out $m$ copies of the fundamental class of the critical manifold $\mathbb{P}(V_{m/2})$}
  \label{fig:rotation}
\end{figure}


\begin{proof}[Proof of Theorem \ref{thm: main}] Let $\mathfrak{s}_{M}$ be the canonical spin-c structure on $M$. 
It suffices to show that $\FSW(M,\mathfrak{s}_{M},\tau^{k}_{M})\neq 0$ for any $k\neq 0$. 

Let $\widetilde{M}$ be the mapping torus of $\tau_{M}$. Then, as a family of 4-manifolds with boundary $\widetilde{M}$ is isomorphic to a  trivial bundle with non-trivial boundary parametrisation 
\[
\partial \widetilde{M}\cong S^{1}\times Y,\quad (e^{i\theta},1,y)\mapsto (e^{i\theta}, e^{i\theta}\cdot y).
\]
By removing small disks in each fiber, we obtain a family cobordism $\widetilde{M}^{\circ}$ from $S^{3}$ to $Y$. We have a decomposition 
\[
\widetilde{M}^{\circ}\cong \overline{M}^{\circ}\cup \widetilde{Z},
\]
where $\overline{M}^{\circ}$ denotes the trivial bundle $S^{1}\times M^{\circ}$ with trivial boundary parametrization. Let $\mathfrak{s}_{\widetilde{M}^{\circ}}$ be any family spin-c structure that restricts to $\mathfrak{s}_{M}$ on fibers. Then by the gluing formula (\ref{eq: hm arrow gluing}), we have 
\[
\FSW(M,\mathfrak{s}_{M},\tau_{M})=\widecheck{HM}(\widetilde{Z},\widetilde{\mathfrak{s}}_{c})\circ  \overrightarrow{HM}(M^{\circ},\mathfrak{s}_{M^{\circ}})(\hat{1})
\]

Since $M$ is a weak symplectic filling of a rational homology $3$-sphere then by a result of Eliashberg \cite{eliashberg} the symplectic form can be modified to give a \textit{strong} symplectic filling (i.e. the symplectic form near boundary is modelled on a neighbourhood of $(Y, \xi )$ in its symplectisation). By gluing $M$ with a symplectic cap with $b^{+}>0$, one obtains a closed symplectic $4$-manifold, whose Seiberg-Witten invariant (in the canonical spin-c structure) is $\pm 1$ \cite{Taubes94}. Therefore, $\overrightarrow{HM}(M^{\circ},\mathfrak{s}_{M^{\circ}})(\hat{1})$ must be a primitive element of $HM^{\redu}(Y,\mathfrak{s}_{c})$, i.e. we have 
\begin{equation}\label{eq: HM-arrow M}
\overrightarrow{HM}(M^{\circ},\mathfrak{s}_{M^{\circ}})(\hat{1})=\pm([c]-[\bar{c}]+\ell).    
\end{equation}
Then by Proposition \ref{prop: twisted cobordism W}, we have 
\[
\FSW(M,\mathfrak{s}_{M},\tau_{M})=\pm \widecheck{HM}(\widetilde{Z},\widetilde{\mathfrak{s}}_{c})([c]-[\bar{c}]+\ell)=\pm m \check{1}_{Y}\in \widecheck{HM}(Y,\mathfrak{s}_{c}).
\]
For any $k\neq 0$, by (\ref{eq: FSW additive}), we have 
\begin{equation}\label{eq: FSW tau nonzero}
FSW(M,\mathfrak{s}_{M},\tau^k_{M})=\pm mk \check{1}_{Y}\neq 0.    
\end{equation}
\end{proof}

The following result, together with (\ref{eq: FSW tau nonzero}), implies Theorem \ref{thm: factorization}. It is also the additional ingredient needed in the proof of Corollary \ref{symplectic torelli}.
\begin{proposition}\label{prop: FSW vanishing for Dehn-Seidel twists}
Let $M$ be a smooth compact $4$-manifold with connected boundary $Y$ (which is \textbf{not} assumed to be Floer-simple). Let $S_{i}$ ($1\leq i\leq n$) be a collection $2$-spheres smoothly embedded in the interior of $M$, with $S_{i}\cdot S_{i}<0$. Let $\delta_{i} \in \MCG(M)$ be a diffeomorphism supported in a tubular neighborhood of $S_{i}$, and let $\mathfrak{s}$ be a spin-c structure on $M$ such that $ c_{1}(\mathfrak{s}) \cdot S_{i} =0$ for all $i$. Assume $b^{+}(M)\geq 3$. Then one has 
\[
\FSW(M,\mathfrak{s},\delta_{1}\circ\cdots\circ \delta_{n})=0\in \widecheck{HM}(Y,\mathfrak{s}|_{Y}).
\]
\end{proposition}
\begin{proof} We first prove that $\FSW(M,\mathfrak{s},\delta_{i})$ is well-defined. Let $N_{i}$ be a tubular neighborhood of $S_{i}$ such that $f_{i}$ is supported on $\operatorname{int}(N_{i})$. Since $b^{+}(N_{i})=b^{1}(N_{i})=0$, $\delta_{i}$ preserves the homological orientation on $M$. Since $\mathfrak{s}|_{N_{i}}$ is the unique self-conjugate spin-c structure on $N_{i}$, $\delta_i$ also preserves $\mathfrak{s}$.

By (\ref{eq: FSW additive}), it suffices to show  $\FSW(M,\mathfrak{s},\delta_{i})=0$. Let $Y_{i}=\partial N_{i}$ and let $M'_{i}=M\setminus \operatorname{int}(N_{i})$. Let $\widetilde{M}_{i}$ and $\widetilde{N}_{i}$ be the mapping torus of $\delta_{i}$ and $\delta_{i}|_{N_{i}}$ respectively. We take a constant section of $\widetilde{N}_{i}$. By removing small balls around this section, we obtain a family cobordism $\widetilde{N}^{\circ}_{i}$ from $S^3$ to $Y_{i}$ and a family cobordism $\widetilde{M}^{\circ}$ from $S^{3}$ to $Y$. Let $\mathfrak{s}_{\widetilde{M}_{i}}$ be a family spin-c structure on $\widetilde{M}_{i}$ that restricts to $\mathfrak{s}$ on each fiber. We use $\mathfrak{s}_{X}$ to denote restriction of  $\mathfrak{s}_{\widetilde{M}_{i}}$ to various $X\subset \widetilde{M}_{i}$. Note the decomposition 
\[
\widetilde{M}_{i}^{\circ}\cong \widetilde{N}^{\circ}_{i}\cup (M'_{i}\times S^1).
\]
By (\ref{eq: hmhat gluing}), we have
\[
\FSW(M,\mathfrak{s},\delta_{i})=\overrightarrow{HM}(\widetilde{M}_{i}^{\circ},\mathfrak{s}_{\widetilde{M}_{i}^{\circ}})(\hat{1})=\overrightarrow{HM}(M'_{i},\mathfrak{s}_{M'_{i}})\circ \widehat{HM}(\widetilde{N}_{i}, \mathfrak{s}_{\widetilde{N}^{\circ}_{i}})(\hat{1})
\]
By the grading shift formula, one has 
\[
\grad(\widehat{HM}(\widetilde{N}_{i}, \mathfrak{s}_{\widetilde{N}^{\circ}_{i}})(\hat{1}))=\grad(\hat{1})+\frac{c^2_{1}(\mathfrak{s}_{N^{\circ}_{i}})-\sigma(N_{i})}{4}+\operatorname{dim}(S^1)=\frac{1}{4}.
\]
But $Y_{i} \cong L(p,1)$ (where $p = - S_i \cdot S_i$) is an $L$-space and $\widehat{HM}(Y_{i},\mathfrak{s}_{Y_{i}})$ is supported in grading $-\frac{3}{4}+2\mathbb{Z}$ \cite[Proposition 3.3]{monolens}. So $\widehat{HM}(\widetilde{N}_{i}, \mathfrak{s}_{\widetilde{N}^{\circ}_{i}})(\hat{1})=0$. This shows that $\FSW(M,\mathfrak{s},\delta_{i})=0$. The proof is hence finished.
\end{proof}



\section{Proofs of Theorems \ref{thm: MCG exotic R4}, \ref{thm: Mazur} and \ref{thm: closed}}
\label{section other main results}
In this section, we will prove Theorems \ref{thm: MCG exotic R4}, \ref{thm: Mazur} and \ref{thm: closed}. We start with the following proposition, which directly implies Theorem \ref{thm: closed}.

\begin{proposition}\label{prop: theorem D refined}
Let $X$, $Y$ and $\mathfrak{s}$ be as in Theorem \ref{thm: closed}. Then we have 
\[
\FSW(X,\mathfrak{s},\tau_{(X,Y)})\neq 0.
\]
In particular, for any $k\neq 0$, $\tau^k_{(X,Y)}$ is not smoothly isotopic to a diffeomorphism supported on $D^{4}$.     
\end{proposition}

\begin{proof} Let $\widetilde{X} \rightarrow S^1$ be the mapping torus of $\tau_{(X,Y)}$ and let $\mathfrak{s}_{\widetilde{X}}$ be any family spin-c structure that restricts to $\mathfrak{s}$ on fibers. As an $S^1$-family of manifolds, the mapping torus of $\tau_{(X,Y)}$ has a decomposition
\[
\widetilde{X}=\overline{M}\cup \widetilde{Z}\cup \overline{W}.
\]
Here $\overline{M}$ and $\overline{W}$ denote trivial families with fiber $M$ and $W$ respectively. And $\widetilde{Z}$ is the twisted family as defined in Section \ref{section: proof main}. 
By removing tubular neighborhoods of constant sections in $\overline{M}$ and   $\overline{W}$, we obtain family cobordisms $\overline{M}^{\circ}$ and $\overline{W}^{\circ}$ respectively. Set $\widetilde{X}^{\circ}=\overline{M}^{\circ}\cup \widetilde{Z}\cup \overline{W}^{\circ}$. Then by equations (\ref{eq: hm arrow gluing}), (\ref{eq: hmcheck gluing}) and (\ref{eq: FSW formula}), we have 
\begin{equation}\label{eq: FSW calculation}
\begin{split}
  \FSW(\widetilde{X}, \mathfrak{s}_{\widetilde{X}})\cdot \check{1}=&\overrightarrow{HM}(\overline{M}^{\circ}\cup \widetilde{Z}\cup \overline{W}^{\circ},\mathfrak{s}_{\overline{M}^{\circ}}\cup \mathfrak{s}_{\widetilde{Z}}\cup \mathfrak{s}_{\overline{W}^{\circ}})(\hat{1})\\
  =& \widecheck{HM}(\widetilde{Z}\cup \overline{W}^{\circ},\mathfrak{s}_{\widetilde{Z}}\cup \mathfrak{s}_{\overline{W}^{\circ}})\circ \overrightarrow{HM}(M^{\circ},\mathfrak{s}_{M^{\circ}})(\hat{1})\\
  =& \widecheck{HM}(W^{\circ},\mathfrak{s}_{W^{\circ}})\circ \widecheck{HM}(\widetilde{Z},\mathfrak{s}_{\widetilde{Z}})\circ \overrightarrow{HM}(M^{\circ},\mathfrak{s}_{M^{\circ}})(\hat{1})\\
  =& \pm m \cdot \widecheck{HM}(W^{\circ},\mathfrak{s}_{W}^{\circ})(\check{1}_{Y})
\end{split}
\end{equation}
On the other hand, since $b^{+}(W)=0$, we have 
\begin{equation}\label{eq: hm-check W}
 \widecheck{HM}(W^{\circ},\mathfrak{s}_{W}^{\circ})(\check{1}_{Y})=U^{-l}\cdot \check{1}   
\end{equation}
 for some $l\geq 0$. Comparing (\ref{eq: FSW calculation}) and (\ref{eq: hm-check W}), we see that $l=0$ and \[\FSW(\widetilde{X},\mathfrak{s}_{\widetilde{X}})=\pm m\neq 0.\] 
\end{proof}
Proof of Theorem \ref{thm: Mazur} relies on the following key proposition. 

\begin{proposition}\label{prop: S1-cobordism}Let $Y$ be one of the families.
\begin{equation}\label{eq: subcollection of Casson-Harer (2)}
 \left\{
\begin{array}{ll}
\Sigma(p,ps+1,ps+2) & \text{($s\geq1$, and odd $p\geq3$)},\\
\Sigma(p,ps-1, ps-2) & \text{($s \geq 2$, and odd $p\geq5$),}
\end{array}
\right.   
\end{equation}
Then there exists a Floer simple manifold $Y'$ and a cobordism $N$ from $Y'$ to $Y$ that satisfies the following conditions:
\begin{enumerate}[(i)]
     \item $b_{1}(N)=0$ and $b^{+}_{2}(N)=0$.
    \item The Seifert $S^1$-action on $\partial N$ extends to a smooth $S^1$-action on $N$.
    \item There exists a spin-c structure $\mathfrak{s}_{N}$ that restricts to the canonical spin-c structure $\mathfrak{s}_{Y'}$ on $Y'$ and satisfies 
    \begin{equation}\label{eq: grading shift N}
    2h(Y',\mathfrak{s}_{Y'})=\frac{c^2_{1}(\mathfrak{s}_{N})-2\chi(N)-3\sigma(N)}{4}.    
    \end{equation}
\end{enumerate}
\end{proposition}
\begin{proof} First consider the case $Y=(p,ps+1,ps+2)$. Let $p=2r+1$. Then the plumbing diagram $\Gamma$ for $Y$ is given in Figure \ref{fig: plumbing Sigma(p,ps+1,ps+2)} (left), with vertices labeled as Figure \ref{fig: plumbing Sigma(p,ps+1,ps+2)} (right).

\begin{figure}[h!]
    \centering
 \includegraphics[scale=0.27]{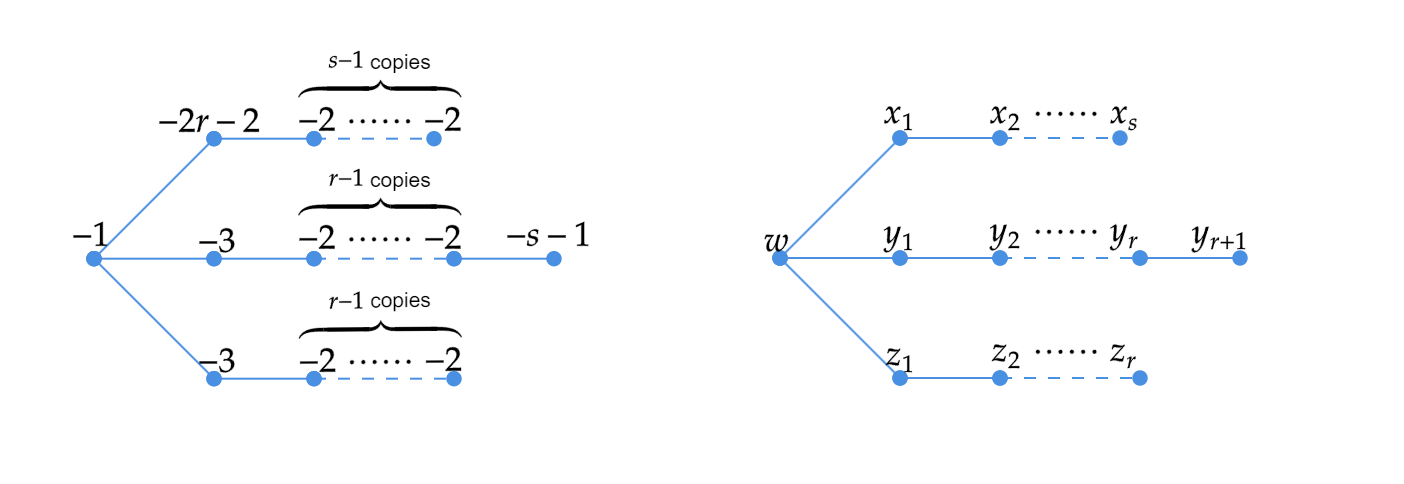}
\caption{Plumbing diagram for $\Sigma(p,ps+1,ps+2)$ (left) with labels on the vertices (right)}
  \label{fig: plumbing Sigma(p,ps+1,ps+2)}
\end{figure}

Let $P$ be the plumbed 4-manifold for $\Gamma$. Then each vertex $v$ of $\Gamma$ corresponds to an element $v\in H_{2}(P;\mathbb{Z})$, represented by an embedded sphere $S(v)\hookrightarrow P$. The $S^1$-action on $Y$ naturally extends to an $S^1$-action on $P$ that fixes each $S(v)$ as a set. (See discussions in \S 2.3.3.) Take an $S^1$-invariant regular neighborhood of 
\[S(w)\cup S(x_1)\cup S(y_1)\cup S(z_1)\] and denote its closure by $P'$. Then $P'$ is diffeomorphic to the plumbed 4-manifold for $\Gamma'$, the subgraph with vertices $w,x_1,y_1,z_1$. We let $Y'=\partial P'$ and let $N=P\setminus \operatorname{int}(P')$. Then $Y'$ is  the link of a triangle singularity and hence Floer simple. Condition (i) and (ii) are obvious from the construction of $N$. So it remains to find $\mathfrak{s}_{N}$ that satisfies (iii). For this, we consider the following equations
\begin{equation}\label{eq: new generators}
\left\{
\begin{array}{ll}
w=w' &\\
x_1=-w'+x'_1+(y'_2+\cdots+y'_{r+1})-(z'_1+\cdots+z'_{r}) &\\
y_1=-w'+y'_1-y'_{2} &\\
z_1=-w'+z'_1-y'_{1} &\\
x_{i}=x'_{i}-x'_{i-1} & \text{(for $2 \leq i\leq s$)}\\
y_{i}=y'_{i}-y'_{i+1} & \text{(for $2 \leq i\leq r$)}\\
z_{i}=z'_{i}-z'_{i-1} & \text{(for $2 \leq i\leq r$)}\\
y_{r+1}=y_{r+1}'-(x'_1+\cdots+x'_{s}) &
\end{array}
\right..
\end{equation}
Assume $x'_*,y'_*,z'_*$ and $w'$ are orthonormal. Using (\ref{eq: new generators}), it is straightforward to calculate the intersection numbers between $x_*,y_*,z_*, w$  and check that they agree with those given by the diagram $\Gamma$.  Therefore, the unique set of solutions  
\[
S'=\{w', x'_{1},\cdots,x'_{s},y'_{1},\cdots,y'_{r+1},z'_{1},\cdots,z'_{r}\}
\]
gives a new set of generators for  for $H_{2}(P;\mathbb{Z})$, under which the intersection form is diagonalized. Now we consider a spin-c structure $\mathfrak{s}_{P}$ on $P$ that satisfies 
\[
c_{1}(\mathfrak{s}_{P})=-w'+(x'_{1}+\cdots + x'_{s})-y'_{1}-y'_{2}+(y'_{3}+\cdots+y'_{r+1})-(z'_{1}+\cdots+z'_{r}).
\]
We set $\mathfrak{s}_{N}=\mathfrak{s}_{P}|_{N}$ and $\mathfrak{s}_{P'}=\mathfrak{s}_{P}|_{P}$. Note that the adjunction formula 
\begin{equation}\label{eq: adjunction}
c_{1}(\mathfrak{s}_{P})\cdot v=v\cdot v+2    
\end{equation}
is satisfied for any vertex $v$ of $\Gamma'$. Therefore, $\mathfrak{s}_{P'}$ is just the canonical spin-c structure on $P'$. This implies $\mathfrak{s}_{N}|_{Y'}=\mathfrak{s}_{P'}|_{Y'}=\mathfrak{s}_{Y'}$.  

Since $N$ and $P'$ are negative definite and since $h(S^3)=h(Y)=0$, the Fr\o yshov invariants satisfy the inequalities 
\begin{equation}\label{eq: Froyshov inequality}
\frac{c^2_{1}(\mathfrak{s}_{N})-2\chi(N)-3\sigma(N)}{4}\leq 2h(Y',\mathfrak{s}_{Y'})
    \end{equation}
and 
\[
\frac{c^2_{1}(\mathfrak{s}_{P'})-2\chi(P'\setminus D^4)-3\sigma(P')}{4}\leq -2h(Y',\mathfrak{s}_{Y'}).
\]
On the other hand, since 
\begin{equation}\label{eq: minimal length}
c^{2}_{1}(\mathfrak{s}_{P})=\sigma(P),    
\end{equation}
we have 
\[
\begin{split}
&\frac{c^2_{1}(\mathfrak{s}_{P'})-2\chi(P'\setminus D^4)-3\sigma(P')}{4}+\frac{c^2_{1}(\mathfrak{s}_{N})-2\chi(N)-2\sigma(N)}{4}\\=&
\frac{c^2_{1}(\mathfrak{s}_{P})-2\chi(P\setminus D^4)-3\sigma(P)}{4}
\\=&\frac{c^2_{1}(\mathfrak{s}_{P})-\sigma(P)}{4}\\=&0.    
\end{split}
\]
Hence ``='' in (\ref{eq: Froyshov inequality}) is achieved. This completes the proof for $Y=\Sigma(p,ps+1,ps+2)$.

The case $Y=\Sigma(p,ps-1,ps-2)$ is similar. Let $p=2r+1$. The plumbing diagram $\Gamma$ is shown in Figure \ref{fig: plumbing Sigma(p,ps-1,ps-2)}. 

\begin{figure}[h!]
    \centering
 \includegraphics[scale=0.27]{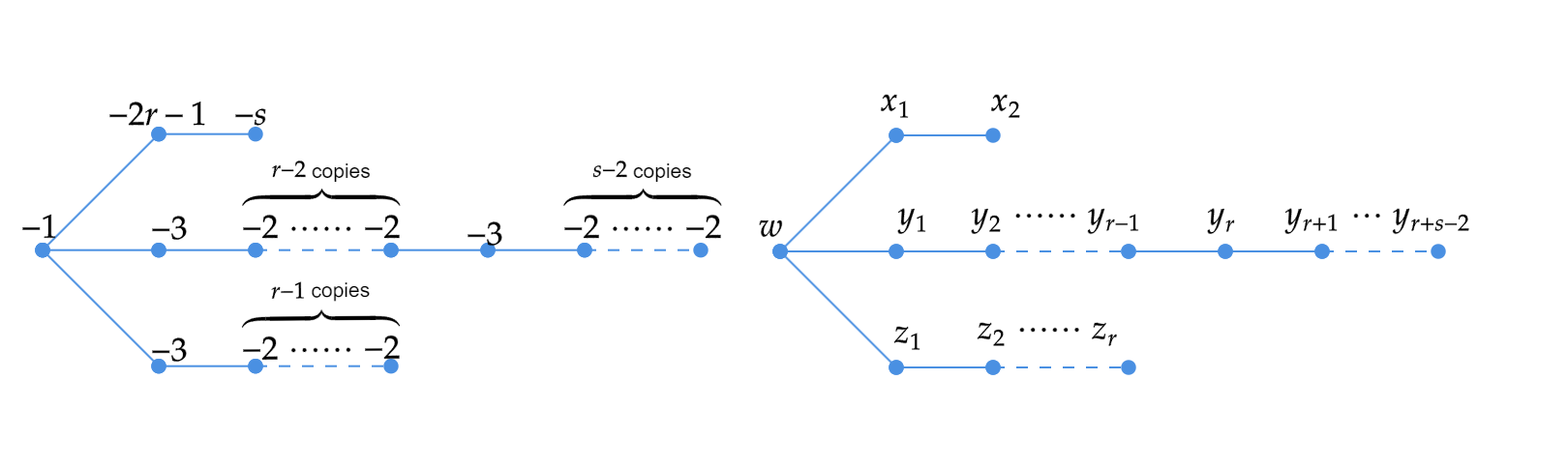}\caption{Plumbing diagram for $\Sigma(p,ps-1,ps-2)$ (left), with labels on the vertices (right)}
  \label{fig: plumbing Sigma(p,ps-1,ps-2)}
\end{figure}

We again let $\Gamma'$ be the subdiagram with vertices $\{w,x_1,y_1,z_1\}$ and let $P$ (resp. $P'$) be the plumbed manifold for $\Gamma$ (resp. $\Gamma'$). Set $N=P\setminus \operatorname{int}(P')$. The intersection form on $H_{2}(P;\mathbb{Z})$ is diagonalized under the generators
\[
S'=\{w', x'_{1},x'_{2},y'_{1},\cdots,y'_{r+s-2},z'_{1},\cdots,z'_{r}\}
\]
determined by the equations.
\[
\left\{
\begin{array}{ll}
w=w' &\\
x_1=-w'-x'_1+x'_2+(y'_1+\cdots+y'_{r-1})-(z'_2+\cdots+z'_{r}) &\\
y_1=-w'-y'_1+z'_{1} &\\
z_1=-w'-z'_1+z'_{2} &\\
x_2=-x'_{2}-(y'_{r}+\cdots+y'_{r+s-2})&\\
y_{i}=y'_{i-1}-y'_{i} & \text{(for $2 \leq i\leq r-1$)}\\
y_{r}=y'_{r}+y'_{r-1}-x'_{2} \\
y_{i}=y'_{i}-y'_{i-1} & \text{(for $r+1 \leq i\leq r+s-2$)}\\
z_{i}=z'_{i+1}-z'_{i} & \text{(for $2 \leq i\leq r-1$)}\\
z_{r}=x'_{1}-z'_{r}
\end{array}
\right.\]
Consider spin-c structure $\mathfrak{s}_{P}$ on $P$ that satisfies 
\[
c_{1}(\mathfrak{s}_{P})=-w'-x'_{1}+x'_{2}-y'_{1}+(y'_{2}\cdots +y'_{r+s-2})-(z'_{1}+\cdots+z'_{r}).
\]
We let $\mathfrak{s}_{N}=\mathfrak{s}_{P}|_{N}$. Then $\mathfrak{s}_{P}$ satisfies (\ref{eq: adjunction}) and (\ref{eq: minimal length}). As a result, we can repeat the argument as the previous case and prove that $\mathfrak{s}_{N}$ satisfies condition (iii).    \end{proof}
Now we are ready to state and prove a slight refinement of Theorem $B$. 
\begin{proposition}\label{prop: theorem B refined} Let $Y$ be one of the families in (\ref{eq: subcollection of Casson-Harer (2)}) and $W$ be a rational homology ball bounded by $Y$. Then there exists a simply-connected smooth 4-manifold $M$ bounded by $Y$ with $b^+(M)>1$, and a spin-c structure $\mathfrak{s}$ on $X=M\cup_{Y}(-W)$ such that
\[
\FSW(X,\mathfrak{s},\tau_{(X,Y)})\neq 0.
\]
In particular, for any $k\neq 0$, the Dehn twists $\tau_{(X,Y)}$ (hence $\tau_{W}$) is not isotopic to a diffeomorphism supported on $D^4$.
\end{proposition}

\begin{proof}By applying Proposition \ref{prop: S1-cobordism} to $Y$, we obtain a spin-c-cobordism $(N,\mathfrak{s}_{N})$ from some Floer simple $Y'$ to $Y$. Let $\xi$ be the contact structure on $Y'$ obtained by perturbing a transverse foliation on $Y$, whose existence was established in \cite{lisca-stipsicz}. Let $M'$ be a simply-connected strong filling of $\xi$ with $b^{+}(M')>1$. We set 
\[
M=M'\cup_{Y'}N.
\]
Consider the spin-c structure $\mathfrak{s}=\mathfrak{s}_{M'}\cup \mathfrak{s}_{N}\cup \mathfrak{s}_{W}$, where $\mathfrak{s}_{M'}$ denotes the canonical spin-c structure on $M'$ and $\mathfrak{s}_{W}$ is any spin-c structure on $-W$. 

By Proposition \ref{prop: HM-check Y}, $HM^{\redu}(Y',\mathfrak{s}_{Y'})$ is supported in the grading $-2h(Y',\mathfrak{s}_{Y'})-1$. By applying the grading shift formula to (\ref{eq: HM-arrow M}), we obtain 
\begin{equation}\label{eq: grading shift M}
-2h(Y',\mathfrak{s}_{Y'})=\frac{c^2_{1}(\mathfrak{s}_{M'})-2\chi(M'\setminus D^4)-3\sigma(M')}{4}=\frac{c^2_{1}(\mathfrak{s}_{M'})-2\chi(M')-3\sigma(M')}{4}+\frac{1}{2}.    
\end{equation}
On the other hand, we have 
\begin{equation}\label{eq: grading shift W}
-\frac{1}{2}=\frac{c^2_{1}(\mathfrak{s}_{W})-2\chi(-W)-3\sigma(-W)}{4}  
\end{equation}
Adding together (\ref{eq: grading shift M}), (\ref{eq: grading shift N}) and (\ref{eq: grading shift W}), we obtain 
\[
-1=\frac{c^2_{1}(\mathfrak{s})-2\chi(X)-3\sigma(X)}{4}.  
\]
This allows us to apply Proposition \ref{prop: theorem D refined} to $(X,Y')$ and conclude that 
\[\FSW(X,\mathfrak{s},\tau_{(X,Y')})\neq 0.\] 
On the other hand, by a straightforward modification of Lemma \ref{lemma:loopextension}, the fact that the Seifert $S^1$-action on $\partial N$ extends smoothly over $N$ gives us
\[\tau_{(X,Y')}=\tau_{(X,Y)}\in \MCG(X).\] This finishes the proof.
\end{proof}

\begin{proof}[Proof of Corollary \ref{cor: exotic disk}] Consider the handle decomposition of the Mazur manifold $M$ into $S^1\times D^3$ and a $2$-handle. Let $D_0$ be the core of the $2$-handle and set 
$D_{i}=\tau^{i}_{M}(D_0)$. Then $D_{i}$ is topologically isotopic to $D_0$ since $\tau^{i}_{M}$ is topologically isotopic to $\operatorname{id}$. It remains to show that $D_{i}$ and $D_{j}$ are not smoothly isotopic if $i\neq j$. Since $\tau_{M}^{j-i}$ sends $D_0$ and $D_i$ to $D_{j-i}$ and $D_j$, respectively, it suffices to show that $D_{i}$ is not smoothly isotopic to $D_0$ for any $i\neq 0$. Suppose that this is not the case. Then by the isotopy extension theorem, one can smoothly isotope $\tau^{i}_{M}$ such that it fixes all points in a neighborhood of $\nu(\partial M)$ and a neighborhood $\nu(D_0)$ (see \cite[Lemma 3.1]{LinMukherjee}), i.e. $\tau^{i}_{M}$ can be smoothly isotoped to be supported in $M\setminus(\nu(D_0)\cup \nu(\partial M))$. On the other hand, there is a diffeomorphism on $M$ that takes $M\setminus(\nu(D_0)\cup \nu(\partial M))$ to a tubular neighborhood of $\gamma=S^1\times \{0\}$. Observe that $\gamma$ is unknotted in $M$ since  $\pi_{1}(M)=0$. Therefore, there exists an embedded $D^4\subset \operatorname{int}(M)$ that contains $M\setminus(\nu(D_0)\cup \nu(\partial M))$. So $\tau^{i}_{M}$ is smoothly isotopic to a diffeomorphism supported in $D^4$. This contradicts Theorem \ref{thm: Mazur}.
\end{proof}

\begin{proof}[Proof of Theorem \ref{thm: MCG exotic R4}] 
Take any $Y$ in the families (\ref{eq: subcollection of Casson-Harer (2)}) and any Mazur manifold $W$ bounded by $Y$. Let $M$, $X$, $\mathfrak{s}$ be obtained by applying Proposition \ref{prop: theorem B refined} to $W$. Consider the smooth manifold $M'=X\# W$. Then both $M$ $M'$ are simply-connected, have the same boundary and intersection form. Therefore, there exists a homeomorphism $\rho': M\to M'$ inducing the identity map on the boundary (see \cite{Boyer86,Stong93}). Extending $\rho'$ by the identity on $W$, we obtain the homeomorphism 
\[
\rho: X\to M'\cup_{Y}-W= (M\setminus \operatorname{int}(D^4))\cup_{S^3}((W\cup_{Y}-W)\setminus \operatorname{int}(D^4)).
\]
Consider the open subset $\mathcal{R} \subset X$ 
\[\mathcal{R}=\rho^{-1}((W\cup_{Y}-W)\setminus D^4),\] equipped with the smooth structure induced from $X$. Since $W$ is Mazur, its double $W\cup-W$ is diffeomorphic to $S^4$ and hence $(W\cup_{Y}-W)\setminus D^4$ is diffeomorphic to $\mathbb{R}^4$. Therefore, $\rho$ restricts to a homeomorhism from $\mathcal{R}$ to $\mathbb{R}^4$. Thus, we have obtained a smooth $4$-manifold $\mathcal{R}$ homeomorphic to $\mathbb{R}^4$ equipped with smooth embeddings
\[
W \hookrightarrow \mathcal{R}\hookrightarrow X.
\]

Consider now the homomorphism $\delta: \MCG_{c}(\mathcal{R})\to \mathbb{Z}$
given by
\[
\delta(f):=\FSW(X,\mathfrak{s},f\cup \operatorname{id}_{X\setminus \mathcal{R}}).
\]
By Proposition \ref{prop: theorem D refined} we have $\delta(\tau_{(\mathcal{R},Y)})\neq 0$. By changing the codomain to $\operatorname{image}(\delta)$, we obtain a surjective homomorphism from $\MCG_{c}(\mathcal{R})$ to $\mathbb{Z}$. Furthermore, there does not exists a smooth embedding $D^4\hookrightarrow \mathcal{R}$ that contains the compact set $M$, since otherwise $\tau_{(X,Y)}$ would be supported in $D^4$ and thus contradicting Proposition \ref{prop: theorem D refined}. Hence $\mathcal{R}$ is not diffeomorphic to $\mathbb{R}^4$. 
\end{proof}

\appendix

\section{Deformation theory of effective divisors}\label{section:divisors}

We provide here some background material related to the deformation theory of effective divisors on complex orbifolds from the complex analytic viewpoint. In particular, we want to discuss Lemma \ref{lemma:effsurface}, which should be well-known to experts but we didn't find a convenient reference in the literature.

\subsection{Effective divisors}

Fix an orbifold complex line bundle $L \rightarrow X$ over a compact complex orbifold $X$. A holomorphic structure on $L$ can be regarded as a Cauchy--Riemann operator $\overline{\partial} : \Omega^{0 , \ast} (X, L) \rightarrow \Omega^{0 , \ast +1}(X , L )$ (or ``half connection") on $L$ satisfying the Newlander--Nirenberg integrability condition $\overline{\partial}^2 = 0$. We recall the following

\begin{definition}\label{definition:divisor}
An \textit{effective orbifold divisor} on $X$ in the orbifold line bundle $L$ consists of a pair $(\overline{\partial} , s )$ where $\overline{\partial}$ is a holomorphic structure on $L \rightarrow X$, and $s$ is a \textit{non-trivial} holomorphic section of $L \rightarrow X$, i.e. $\overline{\partial} s = 0$. We denote by $\mathcal{D}(X, L )$ the quotient of the space of such pairs $(\overline{\partial} , s )$ by the group $\mathrm{Aut}L = \mathrm{Map}(X , \mathbb{C}^\ast )$ of complex $C^\infty$ automorphisms of $L$, equipped with the Whitney $C^\infty$ topology.
\end{definition}


Associated to an effective divisor $(\overline{\partial}, s)$ we have the zero set $D   = s^{-1}(0)$. We may regard $D$ as a (possibly non-reduced) complex-analytic orbispace (see Remark \ref{remark:orbispace}), but this won't be needed for our purposes. Let $\mathcal{L}$ be the orbifold sheaf on $X$ of holomorphic sections of $(L , \overline{\partial} )$. We tacitly
\textit{define} the sheaf $\mathcal{L}|_D$ by the short exact sequence (i.e. as the cokernel of a morphism of orbifold sheaves) 
\[
0 \rightarrow \mathcal{O}_X \xrightarrow{s} \mathcal{L} \rightarrow \mathcal{L}|_D \rightarrow 0 .
\]

We refer to \cite{boyer} for the notion of an orbifold sheaf and related notions. One fact that we shall use below (see \cite{boyer}, Lemma 4.2.4) is that there is a one-to-one correspondence between orbifold sheaves $\mathcal{E}$ on $X$ and sheaves $|\mathcal{E}|$ on the underlying complex-analytic space $|X|$, together with natural isomorphisms $H^{i}(X, \mathcal{E}) \cong H^{i}(|X|, |\mathcal{E}|)$ for all $i \in \mathbb{Z}$. If $X = C$ is a one-dimensional complex orbifold, then an orbifold sheaf $\mathcal{E}$ is invertible (i.e. it corresponds to an orbifold holomorphic line bundle) if and only if $|\mathcal{E}|$ is invertible, but this statement does not hold more generally.




\subsubsection{Effective divisors on orbifold Riemann surfaces}

When $X = C$ is a compact orbifold Riemann surface we have a firm hold on both the local and the global structure of the moduli space: 
\begin{lemma}\label{lemma:effRiemannsurface}
The space $\mathcal{D}(C, L )$ of effective divisors is a complex manifold of dimension $e$ in a natural way, where $e$ is the background degree of $L$ (i.e. the degree of the desingularisation $|L| \rightarrow |C|$). Furthermore, $\mathcal{D}(C, L )$ is biholomorphic to the $e\mathrm{th}$ symmetric power $\mathrm{Sym}^{e} |C|$.
\end{lemma}
\begin{proof}
The deformation theory of an effective divisor $(\overline{\partial} , s ) \in \mathcal{D}(X, L )$ is governed by the elliptic cochain complex (where the left-most term is taken to be in degree $-1$)
\newcommand{\sbm}[1]{{\let\amp=&\left[\begin{smallmatrix}#1\end{smallmatrix}\right]}}
\[
\begin{tikzcd}
\Omega^0 (C ,  \mathbb{C} )  \arrow{r}{ \sbm{ - \overline{\partial}\\ s }} &  \Omega^{0,1}(C ,  \mathbb{C}) \oplus \Omega^{0}(C, L ) \arrow{r} \arrow{r}{ \sbm{ 
s \amp \overline{\partial}  }} & \Omega^{0,1}(C, L )  . \label{effcomplex1}
\end{tikzcd}
\]
Here, the first map is obtained by linearising the action $(\overline{\partial} , s ) \mapsto ( \overline{\partial} - u^{-1} \partial u , u s )$ of $\mathrm{Aut}L = \mathrm{Map}(C , \mathbb{C}^\ast )$ on the space of pairs, and the second by linearising the holomorphicity equation $\overline{\partial} s = 0$ (and recall that the integrability condition $\overline{\partial}^2 = 0$ is automatic in complex dimension $1$, so we don't consider it here). 

Observe that the complex (\ref{effcomplex1}) is just the mapping cone associated to the chain map $\Omega^{0 , \ast} (C ,  \mathbb{C} ) \xrightarrow{s} \Omega^{0,\ast}(C , L )$. By considering the Dolbeault resolutions of the sheaves $\mathcal{O}_C$ and $\mathcal{L}$,  we see that the cohomology of (\ref{effcomplex1}) at the $j$th step calculates the sheaf cohomology group $H^{j} ( D , \mathcal{L}|_D )$, where $\mathcal{L}$ is the orbifold sheaf of holomorphic sections of $(L , \overline{\partial} )$ and $D = s^{-1}(0)$. In particular, the cohomology of \ref{effcomplex1} vanishes at the step $j = -1$. In addition, because $\mathcal{L}|_D$ is supported on the \textit{zero-dimensional} locus where $s$ vanishes, then the cohomology of \ref{effcomplex1} at the step $j = 1$ also vanishes. 

Thus, the upshot is that $\mathcal{D}(C, L )$ is a \textit{complex manifold} with tangent space $H^0 ( D , \mathcal{L}|_D )$ at $(\overline{\partial} , s )$. This follows by a standard application of the Implicit Function Theorem for Banach spaces, which applies after completing with respect to suitable Sobolev norms, together with elliptic regularity. 
Furthermore, there is a one-to-one correspondence between holomorphic sections $s$ of the orbifold line bundle $\mathcal{L}$ and holomorphic sections $|s|$ of its desingularisation $|\mathcal{L}|$, the assignment $(\overline{\partial} , s )  \mapsto Z(|s|) \in \mathrm{Sym}^{e} |C|$, where $Z(|s|)$ is the divisor of zeros of $|s|$, gives the required biholomorphism of $\mathcal{D}(X, L )$ with $\mathrm{Sym}^{e}|C|$. 
\end{proof}

\subsubsection{Effective divisors on orbifold complex surfaces}

We consider a compact orbifold complex \textit{surface} $X$, i.e. $\mathrm{dim}_\mathbb{C} X = 2$.

\begin{lemma}\label{lemma:effsurface}
The space $\mathcal{D}(X, L )$ of effective divisors is a complex-analytic space in a natural way. A neighbourhood of an effective divisor $(\overline{\partial} , s ) $ in $\mathcal{D}(X, L )$ is biholomorphic to the zero locus of a holomorphic map $\kappa_{div} : H^0 ( D , \mathcal{L}|_D ) \rightarrow H^{1} (D , \mathcal{L}|_D )$ with $\kappa_{div}(0) = 0$ and $(d \kappa_{div})_{0} = 0$, where $\mathcal{L}$ is the orbifold sheaf of holomorphic sections of $(L ,\overline{\partial} )$ and $D = s^{-1} (0 )$.
\end{lemma}

\begin{proof}
Consider now the 4-step elliptic cochain complex, where the left-most term is again in degree -1
\newcommand{\sbm}[1]{{\let\amp=&\left[\begin{smallmatrix}#1\end{smallmatrix}\right]}}
\[
\begin{tikzcd}
\Omega^0 (X , \mathbb{C} )  \arrow{r}{ \sbm{ - \overline{\partial} \\ s }} &  \Omega^{0,1}(X ,  \mathbb{C}) \oplus \Omega^{0}(X, L ) \arrow{r} \arrow{r}{ \sbm{ \overline{\partial } \amp 0 \\ s \amp \overline{\partial}  }} & \Omega^{0,2}(X, \mathbb{C} ) \oplus \Omega^{0,1} (X , L ) \arrow{r}{\sbm{s \amp -\overline{\partial}}} &  \Omega^{0,2}(X , L ). \label{effcomplex1}
\end{tikzcd}
\]
As before, the first map is the linearisation of the action of $\mathrm{Aut}(L) = \mathrm{Map}(X , \mathbb{C}^\ast )$, and the second map is the linearisation of the holomorphicity equations $\overline{\partial}^2 = 0$ and $\overline{\partial} s = 0$. 

By the same argument as in Lemma \ref{lemma:effRiemannsurface}, we identify the cohomology of this complex at the $j$th step with $H^j ( D , \mathcal{L}|_D )$. Cohomology vanishes at $j =-1$ for degree reasons, and at $j = 2$ because $\mathcal{L}|_{\mathcal{D}}$ is supported on the \textit{$1$-dimensional} complex-analytic subspace $D = s^{-1}(0)$.

Thus, the differential $d^1 := \sbm{ s \amp -\overline{\partial}}$ in the elliptic complex is surjective at an effective divisor $(\overline{\partial} , s )$. After completing with respect to suitable Sobolev norms, it follows that its kernel defines an infinite-dimensional Banach vector bundle $\mathcal{V}$ over the space $\mathcal{D}(X, L )$. Let $\mathcal{B}^\ast$ denote the space of all Sobolev pairs $(\overline{\partial} , s )$, not necessarily satisfying any equation but with $s$ non-trivial, modulo $\mathrm{Aut}L$. Since the map $d^1$ is still surjective on a small neighbourhood $\mathcal{N}$ of $\mathcal{D}(X, L) \subset \mathcal{B}^\ast$, then $\mathcal{V}$ extends to an infinite-dimensional holomorphic vector bundle over $\mathcal{N}$. The space $\mathcal{D}(X, L )$ is given by the zero locus of the holomorphic Fredholm section $\sigma$ of $\mathcal{V} \rightarrow \mathcal{N}$ given by $\sigma (\overline{\partial} , s ) := (\overline{\partial}^2 , \overline{\partial}s ) $. Note that, indeed $\sigma$ is a section of $\mathcal{V} \rightarrow \mathcal{N}$, since $d^1 \sigma (\overline{\partial} , s ) = s \overline{\partial}^2 - \overline{\partial}^2 s = 0$ for any pair $(\overline{\partial} , s )$ (even if this pair does not necessarily satisfying the holomorphicity equations).

At this point, the required result follows from the general procedure for constructing Kuranishi charts for the zero locus of a Fredholm section of a Banach vector bundle (see \cite{DK}, \S 4.3). That these Kuranishi charts make $\mathcal{D}(X , L )$ into a complex-analytic space is a consequence of the holomorphicity of the infinite-dimensional bundle $\mathcal{V}$ and its section $\sigma$.
\end{proof}

\section{Laufer's list}

The $2$-dimensional hypersurface singularities with $p_g = 1$ are precisely the $2$-dimensional minimally elliptic singularities with degree $\leq 3$ classified by Laufer (see Theorem 3.13 and Tables 1-3 in \cite{laufer}). Here, the the degree of a singularity is defined as minus the self-intersection number of the fundamental cycle associated to the minimal resolution. Those which are weighted-homogeneous and whose link is a rational homology $3$-sphere are recorded below. 

\begin{figure}[h!]

\caption{List of $2$-dimensional weighted-homogeneous hypersurface singularities with $p_g = 1$, rational homology $3$-sphere link, and ...}
\label{figure:table}

$ $

\centering 
\begin{tabular}{ c c c}
\multicolumn{3}{c}{... of degree $1$}\\
\hline\\ 
 $x^2 + y^3 + z^7 = 0$ & $x^2 + y^3 + z^8 = 0$ & $x^2 + y^3 + x^9 = 0$ \\
 $x^2 + y^3 + z^{10}  = 0$ & $x^2 + y^3 + y z^7 = 0$ & $x^2 + y^3 + z^{11} = 0$  \\ 
\end{tabular}

\end{figure}

\begin{figure}[h!]
\centering
\begin{tabular}{ c c c}
\multicolumn{3}{c}{... of degree $2$}\\
\hline\\
$x^2 + z(y^3 + z^4) = 0$ & $x^2 + z(y^5 + z^5) = 0$ & $x^2 + z(y^3 + z^6 ) = 0$ \\
 $x^2 + z (y^3 + z^7)  = 0$ & $x^2 + z(y^3 + y z^5= 0$ & $x^2 + z(y^3 + z^8 )= 0$  \\ 
 $x^2 + y^4 + z^4 y = 0 $ & $x^2 + y^4 + z^6 = 0$ & $x^2 + y^4 + y z^5 = 0 $ \\
 $x^2 + y^4 + z^7 = 0$ & $x^2 + y^5 + z^5 = 0$ & $ x^2 + y(z^4 + y^5 ) = 0$ \\
 $x^2 + y(z^4 + zy^4 ) = 0$ & $x^2 + y^5 + z^6 = 0$ & $x^2 + y^6 + y z^5 = 0$
\end{tabular}

\end{figure}

\begin{figure}[h!]
\centering
\begin{tabular}{ c c c}
\multicolumn{3}{c}{... of degree $3$}\\
\hline\\
 $zx^2 + y^3 + z^4  = 0$ & $zx^2 + y^3 + z^5 = 0$ & $zx^2 + y^3 + z^6 = 0$ \\
 $zx^2 + y^3 + z^7  = 0$ & $zx^2 + y^3 + y z^5 = 0$ & $zx^2 + y^3 + z^8= 0$  \\ 
 $z(xz+y^2 ) + x^3 y = 0 $ & $z(xz +y^2) + x^5 = 0$ & $z(xz + y^2 ) + x^4 y = 0 $ \\
 $z(xz +y^2 ) + x^6 = 0$ & $x^3 + y^3 + z^4 = 0$ & $ x^3 + y^3 + y z^3 = 0$ \\
 $x^3 + y^3 + z^5 =  0$ & $x^4 + y^4 + yz^2 = 0$ & $x^4 + y^5 + y z^2 = 0$\\
 $xy^3 + x^3 z  + yz^2 = 0$ & $ x^4 + xy^4 + yz^2 = 0$ & $x y^3 + x^5 + yz^2 = 0 $\\
 $ x^4 + y^3 z + xz^2 = 0$ & $ x^4 + y^5 + xz^2=0 $ & $ x^3 + y^4 + x z^3 = 0$
\end{tabular}

\end{figure}

\newpage

\printbibliography

\end{document}